\documentclass[11pt]{article} %
\usepackage{times}
\usepackage{tikz}
\usepackage{authblk}
\usepackage[english]{babel}
\usepackage{amssymb}
\usepackage[margin=1in]{geometry}
\usepackage{verbatim}
\usepackage{lineno}
\usepackage{calligra}

\usepackage{amsthm}
\usepackage{amsmath}
\usepackage{amssymb}
\usepackage{graphicx}
\usepackage[colorlinks=true, allcolors=blue]{hyperref}
\usepackage{verbatim}
\usepackage{cleveref}  
\usepackage{enumitem}
\usepackage{bbm}
\usepackage{csquotes}

\usepackage[natbib=true,url=false]{biblatex}
\DeclareSourcemap{
  \maps[datatype=bibtex]{
    \map{
      \step[fieldset=addendum, null]
      \step[fieldset=note, null]
    } 
  }
}

\numberwithin{equation}{section} 
\newcommand{\lv}{\left|}
\newcommand{\rv}{\right|}
\newcommand{\lV}{\left\|}
\newcommand{\rV}{\right\|}
\newcommand{\ldot}{\left.}
\newcommand{\rdot}{\right.}

\newcommand{\E}{\mathrm{E}}
\newcommand{\Var}{\mathrm{Var}}
\newcommand{\Cov}{\mathrm{Cov}}

\newcommand{\Vol}[1]{\mathrm{Vol}(#1)}
\newcommand{\History}{\mathcal{H}}
\newcommand{\history}{\hbar}

\DeclareMathAlphabet{\mathbcal}{U}{dutchcal}{m}{n}

\newcommand{\Bbba}{\mathbb{B}}

\newcommand{\Ebb}{\mathbb{E}}
\newcommand{\Fbb}{\mathbb{F}}

\newcommand{\Hbb}{\mathbb{H}}
\newcommand{\Ibb}{\mathbb{I}}

\newcommand{\Nbb}{\mathbb{N}}

\newcommand{\Pbb}{\mathbb{P}}

\newcommand{\Rbb}{\mathbb{R}}

\newcommand{\Tbb}{\mathbb{T}}

\newcommand{\Acal}{\mathcal{A}}
\newcommand{\Bcal}{\mathcal{B}}
\newcommand{\Ccal}{\mathcal{C}}
\newcommand{\Dcal}{\mathcal{D}}

\newcommand{\Fcal}{\mathcal{F}}

\newcommand{\Hcal}{\mathcal{H}}
\newcommand{\Ical}{\mathcal{I}}

\newcommand{\Lcal}{\mathcal{L}}
\newcommand{\Mcal}{\mathcal{M}}

\newcommand{\Ocal}{\mathcal{O}}
\newcommand{\Pcal}{\mathcal{P}}

\newcommand{\Scal}{\mathcal{S}}
\newcommand{\Rcal}{\mathcal{R}}

\newcommand{\Vcal}{\mathcal{V}}

\newcommand{\kcal}{\mathbcal{k}}
\newcommand{\lcal}{\mathbcal{l}}

\newcommand{\ocal}{\mathbcal{o}}
\newcommand{\pcal}{\mathbcal{p}}

\newcommand{\constant}{\Cbb}

\newcommand{\indexeddata}{\left\{(X_0,a_0),\dots,(X_n,a_n)\right\}}

\newcommand{\naturalset}{\Nbb}

\newcommand\numberthis{\addtocounter{equation}{1}\tag{\theequation}}

\newcommand{\whiteqed}{\hfill$\square$\par\bigskip}

\newtheorem{assumption}{Assumption}
\newcommand{\prob}{\Pbb}
\newcommand{\expec}{\mathbb{E}}
\newcommand{\probl}{\Pbb}

\newcommand{\indicator}{\mathbbm{1}}

\newcommand{\beq}{\begin{eqnarray*}}
\newcommand{\eeq}{\end{eqnarray*}}
\newcommand{\beqn}{\begin{eqnarray}}
\newcommand{\eeqn}{\end{eqnarray}}
\newcommand{\ben}{\begin{enumerate}}
\newcommand{\een}{\end{enumerate}}
\newcommand{\bit}{\begin{itemize}}
\newcommand{\eit}{\end{itemize}}
\newcommand{\hide}[1]{}

\newcommand{\argmin}{\mathop{\mathrm{argmin}}}
\newcommand{\argmax}{\mathop{\mathrm{argmax}}}

\newcommand{\gn}{\, | \,}

\newcommand{\eps}{\varepsilon}
\newcommand{\vertiii}[1]{{\left\vert\kern-0.25ex\left\vert\kern-0.25ex\left\vert #1 
    \right\vert\kern-0.25ex\right\vert\kern-0.25ex\right\vert}}

\renewcommand{\epsilon}{\eps}

\newcommand{\real}{\Rbb}

\newtheorem{definition}{Definition}
\newtheorem{lemma}{Lemma}
\newtheorem{remark}{Remark}
\newtheorem{proposition}[lemma]{Proposition}
\newtheorem{theorem}{Theorem}
\newtheorem{corollary}{Corollary}

\newcommand{\pow}[1]{^{(#1)}}
\newcommand{\lp}{\left(}
\newcommand{\rp}{\right)}
\newcommand{\lc}{\left\{}
\newcommand{\rc}{\right\}}
\newcommand{\lb}{\left[}
\newcommand{\rb}{\right]}

\newcommand{\muc}{\mu_\chi}
\newcommand{\density}{s}
\newcommand{\Test}{T}

\bibliography{biblio1}

\title{\textbf{Adaptive Estimation of the Transition Density of Controlled Markov Chains}}
\usepackage{authblk}
\author[1]{Imon Banerjee}
\author[2]{Vinayak Rao}
\author[3]{Harsha Honnappa}

\renewcommand{\constant}{\mathbcal{c}}
\newcommand{\Constant}{\mathbcal{C}}
\affil[1]{\footnotesize Department of Industrial Engineering and Management Sciences, Northwestern University}
\affil[2]{\footnotesize Department of Statistics, Purdue University}
\affil[3]{\footnotesize Edwardson School of Industrial Engineering, Purdue University}

\date{}
\begin{document}


\maketitle
\begin{abstract}
    Estimating the transition dynamics of controlled Markov chains is crucial in fields such as time series analysis, reinforcement learning, and system exploration. Traditional non-parametric density estimation methods often assume independent samples and require oracle knowledge of smoothness parameters like the H\"older continuity coefficient. These assumptions are unrealistic in controlled Markovian settings, especially when the controls are non-Markovian, since such parameters need to hold uniformly over all control values. To address this gap, we propose an adaptive estimator for the transition densities of controlled Markov chains that does not rely on prior knowledge of smoothness parameters or assumptions about the control sequence distribution. Our method builds upon recent advances in adaptive density estimation by selecting an estimator that minimizes a loss function {and} fitting the observed data well, using a constrained minimax criterion over a dense class of estimators. We validate the performance of our estimator through oracle risk bounds, employing both randomized and deterministic versions of the Hellinger distance as loss functions. This approach provides a robust and flexible framework for estimating transition densities in controlled Markovian systems without imposing strong assumptions.
\end{abstract}

\tableofcontents

\section{Introduction}\label{sec:introduction}


A stochastic process $\{(X_i, a_i)\}$ is called a \textbf{controlled Markov chain (CMC)} \citep{borkar_topics_1991} if the next ``state" $X_{i+1}$ depends only on the current state $X_i$ and the current ``control" $a_i$. Informally, this means:
\begin{align*} 
\prob\bigl(X_{i+1}\in dy \mid X_0, a_0, \dots, X_i, a_i\bigr) 
= \prob\bigl(X_{i+1}\in dy \mid X_i = x_i,\, a_i = l_i\bigr) 
= s(x_i, l_i, y)\,\mu_\chi(dy),
\end{align*}
where $s(x_i, l_i, y)$ gives the probability density of moving from the current state $x_i$ with action $l_i$ to the next state $y$. Here, the actions $a_i$ depend only on the information available up to time $i$. This paper addresses adaptive estimation of the transition density $\density$ of controlled Markov chains.

In general, controlled Markov chains can be used to model both time-homogenous (like i.i.d \citep{tsybakov_introduction_2009}, Markovian \citep{billingsley_statistical_1961}) and time-inhomogenous (like i.n.i.d, time-inhomogenous Markovian \citep{dolgopyat_local_2023,merlevede_local_2022}, Markov decision process \citep{hernandez-lerma_recurrence_1991}) data. However, they also appear in numerous other problems like offline reinforcement learning \cite{levine_offline_2020}, system stabilisation \citep{yu_online_2023}, or system identification \citep{ljung_system_1999,mania_active_2020}. As a specific example, consider prescribing medication to a diabetic patient, where the state is the current blood glucose level, and the control is the prescribed medication \cite{schaefer_modeling_2004}.

There is no reason to believe that the previous examples involve controls that are Markovian. It is known that certain categories of adversarial Markov games \citep{wang_foundation_2024}, reward machines \citep{icarte_using_2018}, and minimum entropy explorations \citep{mutti_importance_2022} induce Markovian state transitions with non-Markovian controls.  This necessitates sharp estimates of the transition dynamics of Markovian systems in the presence of non-Markovian controls. 

Although nonparametric estimation of the density of i.i.d \citep{tsybakov_introduction_2009} or (more recently) Markovian \cite{athreya_kernel_1998,loffler_spectral_2021} samples is a well-studied topic and has wide applications in settings like regression, classification, and unsupervised learning \citep{massart_concentration_2007}, there is little existing work addressing the estimation of controlled Markov chains. An inherent challenge of this setup is non-stationarity.  Recall from \cite{athreya_kernel_1998} that a natural approach to estimating the transition density of a Markov chain is to estimate the joint density $X_i,X_{i+1}$ and the marginal $X_{i}$ density, and then take the ratio. This method works well even if the Markov chains are ergodic rather than stationary.  {However, if the process is non-stationary and non-ergodic, then there are no well-defined estimators for the joint or the marginal, and the conditional cannot be derived from their ratio.} On a related note, a controlled Markov chain may have all amenable properties like recurrence and mixing without being ergodic (see Lemma \ref{lemma:erg-vs-recurring}).
 
Furthermore, non-parametric estimation presents a number of difficulties, being highly sensitive to the choice of hyperparameters like the bandwidth of the estimator. For example, with $n$ samples and assuming that the density $s$ is $\sigma$-H\"older continuous, one can set the bandwidth to be $O(n^{-1/(2\sigma+1)})$ to obtain the minimax risk $O(n^{-2\sigma/(2\sigma+1)})$ \citep[Chapter 1]{tsybakov_introduction_2009}. However, while it is common practice to assume such oracle knowledge about $\sigma$, this is often unrealistic. Such an assumption is especially problematic when the data is generated by a controlled Markovian process since one requires it to hold for allpossible values of controls. Specifically, with $X_{i}$ being the state at time $i$, $a_i$ being the control at time $i$, and $X_{i+1}$ being the state at time $i+1$, one requires
\[
\prob\lp X_{i+1}\in dx| X_i=x,a_i=l \rp =:s(x,l,y) \mu_\chi (dx)
\]
to be $\sigma$-H\"older continuous for all values of $l$.

To avoid such strong assumptions, we rely upon the recent and rapidly evolving techniques of {\em adaptive density estimation}. This technique was pioneered by \citep{barron_risk_1999} and has been further developed in \citep{massart_concentration_2007,baraud_new_2017,baraud_estimating_2009,baraud_rho-estimators_2018,birge_model_2006,sart_estimation_2014}. In this paper, our objective is to adapt this technique and create an adaptive estimator for the transition densities of controlled Markov chains. 

Informally, adaptive estimation selects a best estimator with respect to {loss $\Hcal$} from a known {class $\Mcal$} by minimising a  {\textbf{contrast} (which for us, is \cref{eq:model} below), thereby completely sidestepping the problem of manually setting the bandwidth. We refer the readers to Chapter 1 of the textbook \cite{massart_concentration_2007} for more details.} Two questions remain: 1) Is the optimisation problem introduced by the contrast computationally tractable for our choices of $\Hcal$, and $\Mcal$?, and  2) Is the selected estimator minimax optimal over the class of \textbf{all possible estimators} under appropriate assumptions on the true density? The answer to both of these questions are in the affirmative. For the former, see Remark \ref{remark:computation}, and for the latter, see Theroem \ref{thm:detls2-lb}, and Corollaries \ref{cor:holder}, and \ref{cor:besov}.  {Importantly, the minimaxity guarantee is achieved without prior knowledge about smoothness parameters}.

\paragraph{Technical Contributions:} 
Our main contribution is showing that an optimal histogram estimator (computable in polynomial time) of the transition function $\density$ based on the dyadic partitions satisfies an oracle risk bound irrespective of the distribution of the controls $a_i$ (Theorem \ref{thm:main-riskbd}). Interestingly, we find that the optimal estimator can be constructed \textit{without} any assumptions on the distribution of the control sequence ${a_i}$. We then validate its performance through oracle risk bounds, employing both instance dependent (Theorem \ref{thm:main-riskbd}) and instance independent (Theorems \ref{thm:detlos-1}, and \ref{thm:detlos-2}) versions of the Hellinger distance as our loss function.  Although \cite{banerjee_off-line_2025} recently derived optimal estimators for the transition density of finite-state, finite-control controlled Markov chains (CMCs), there is surprisingly little work attempting to optimally estimate the transition density of a CMC with continuous state-control spaces. In a series of groundbreaking papers, adaptive estimators were developed for transition densities in various settings: i.i.d. data \citep{baraud_estimator_2011}, stationary Markov chains \citep{lacour_adaptive_2007}, non-stationary $\beta$-mixing Markov chains \citep{sart_estimation_2014}, and stationary $\beta$-mixing paired processes \citep{akakpo_inhomogeneous_2011}. This paper generalizes all of these prior works in several directions. Unlike \citep{baraud_estimator_2011,lacour_adaptive_2007,akakpo_inhomogeneous_2011}, we do not assume our process to be stationary. Furthermore, unlike \citep{sart_estimation_2014}, we do not assume our process to be either Markovian or $\beta$-mixing. This generalization brings with it two distinct challenges, which we describe below.

\begin{enumerate}
    \item \textbf{Question of non-stationarity:}  {In general the $n$-step occupation measure for the non-stationary process may not stabilise in the limit.}  In other words, there may not exist a probability measure $\nu$ such that the $n$-step occupation measure  {$\nu_n(A):=\sum_{i=1}^n\prob((X_i,a_i)\in A)/n\xrightarrow{n\rightarrow \infty}\nu(A)$ }. As mentioned above, there is then no meaningful way to estimate $\nu_n$. Our solution to this problem is twofold. First, we show that for a suitable choice of instance dependent loss function $\Hcal$, the estimator $\hat \density$ is optimal for any given  {$n$-step} occupation measure $\nu_n$
    ? (Theorem \ref{thm:main-riskbd}). Second, we demonstrate that even when using the traditional Hellinger loss, the assumption of stationarity—though convenient  (Theorem \ref{thm:detlos-1})—is not necessary (Theorem \ref{thm:detlos-2}). A careful analysis reveals a deeper connection with the return times of the stochastic process $\{(X_i,a_i)\}$. Key in making this connection is a Kac-type lower bound (Lemma \ref{lemma:KAC-lower}) for recurring processes that we derive, which we believe is of independent interest.
    \item \textbf{Question of mixing:}  {A close inspection of existing literature \citep{deb_trade-off_2024,sart_estimation_2014,akakpo_inhomogeneous_2011} on statistics on dependent samples reveal (see, for instance, \cite[Proposition B.1]{sart_estimation_2014}) the usage of the celebrated Berbee's lemma \citep[Lemma 5.1]{rio_asymptotic_2017}, which requires the $\beta$-mixing assumption.} A key contribution of this paper is to demonstrate that such an assumption is  not necessary. In particular, using recent advances on concentration inequalities for $\alpha$-mixing processes~\citep{merlevede_bernstein_2009}, we derive sharp bounds on the transition density estimator for $\alpha$-mixing CMCs (Theorems \ref{thm:detlos-1} and \ref{thm:detlos-2}). Since there are $\alpha$-mixing processes which are not $\beta$-mixing \citep{bradley_examples_1993}, this provides an important relaxation of the mixing assumptions. 
\end{enumerate}

\subsection{Notation}

Let $\naturalset$ and $\mathbb{R}$ denote the natural and real numbers, and the symbol $\lfloor\cdot\rfloor$,  the floor function. 
All random variables in this paper will be defined with respect to a filtered probability space $(\Omega, \Fcal, \Fbb, \Pbb)$, where $\Fcal$ is a $\sigma$-algebra and $\Fbb := \{\Fcal_i\}_{i\geq 0}$, with $\Fcal_i \subset \Fcal$, is a given filtration. 
Let $\{(X_i,a_i)\}$ represent a discrete-time stochastic processes adapted to $\mathbb F$, and taking values in $\chi\subseteq \Rbb^{d_1}$, $\Ibb\subseteq \Rbb^{d_2}$. We call $\chi$ and $\Ibb$  the \textit{state} and the \textit{control} spaces respectively. For all non-negative integers $i,j$, we define $\History_i^j := (X_j,a_j,\dots,X_i,a_i)$ and $\history_i^j := (x_j,l_j,\dots,x_i,l_i)$ and note that $\history_i^j$ is an element of $(\chi\times\Ibb)^{j-i+1}$. The $\sigma$-field generated by $\History_i^j$ shall be $\Fcal_i^j$. Throughout the paper, we will {assume that $\chi$ and $\Ibb$ are compact}. 
When they are {not compact}, all of our theory still continues to hold on any restriction of $\density$ on a compact subset $A\subset \chi\times\Ibb\times\chi$, given by $\density\indicator_A$.  {Observe that $\density\indicator_A$ is not necessarily a conditional density, in the sense that it may not integrate upto $1$.}

Let $\expec[X]$ be the expectation and $\sigma(X)$ the $\sigma$-algebra induced by $X$. We endow $\chi$ and $\Ibb$ with integrating measures $\mu_\chi$ and $\mu_\Ibb$ respectively. One can assume $\mu$'s to be Lebesgue when $\chi$ and $\Ibb$ are continuous, or count when $\chi$ and $\Ibb$ are discrete. By $\mathrm{Vol}(\Scal)$ we denote the volume of the set $\Scal$ with respect to its natural measure. As an example, if $\Scal\subset \chi$, then $\Vol{\Scal}=\mu_{\chi}(\Scal)$; if $\Scal\subset \Ibb$, then $\Vol{\Scal}=\mu_{\Ibb}(\Scal)$, etc.  $\Constant$ and $\constant$ are always used to denote universal constants whose values can change from line to line.
 We call $m = \lc k:k\subseteq \chi\times\Ibb\times\chi \rc$ to be a \textit{partition} of $\chi\times\Ibb\times\chi$ if $\bigcup_{k\in m}k = \chi\times\Ibb\times\chi$ and $k\bigcap k'=\text{\O}$ for all distinct $ k,k'\in m$. Finally, to avoid trivialities, we assume throughout the paper that the number of samples, denoted by $n$ is at least $3$.

\section{Risk Bounds With Respect to Empirical Hellinger Loss}\label{sec:rand-loss}

\paragraph{Definitions.} 
For an arbitrary process $a_i$ adapted to the filtration $\Fcal_i$, a stochastic process $\{(X_i,a_i)\}$ is said to be a \textbf{controlled Markov chain (CMC)} with \textbf{transition function} $s(\cdot,\cdot,\cdot):\chi\times\Ibb\times\chi\rightarrow \Rbb$ if the conditional probability density (defined as in \cite[Chapter 5]{ash_probability_2000}) satisfies 
\[
\prob\lp X_{i+1}\in dy| \History_0^i = \history_0^i \rp = \prob\lp X_{i+1}\in dy| (X_i,a_i) = x_i,l_i \rp = s(x_i,l_i,y)\mu_\chi (dy),
\]

For any partition $m$, and a sample $\lc (X_i,a_i)\rc_{i=0}^n$ of length $n+1$, the \textbf{histogram estimator} $\hat \density_m(\cdot,\cdot,\cdot)$  of  $\density$ (we will just use the term \textbf{estimator}) is defined as
\[
\hat \density_m(\cdot,\cdot,\cdot) := \sum_{k\in m} \frac{\sum_{i=0}^{n-1}{\indicator_k(X_i,a_i,X_{i+1})}}{\sum_{i=0}^{n-1}\int_{\chi}{\indicator_k(X_i,a_i,y)}d\mu_\chi(y)}\indicator_k(\cdot,\cdot,\cdot). \numberthis \label{def:sm_est}
\]

For any two bounded positive functions $f_1$ and $f_2$ (not necessarily densities) 
define the square of the \textbf{empirical Hellinger distance} $\Hcal^2$ as

\[
\Hcal^2(f_1,f_2) := \frac{1}{2n} \sum_{i=0}^{n-1} \int_\chi \lp \sqrt{f_1(X_i,a_i,y)}-\sqrt{f_2(X_i,a_i,y)}  \rp^2d\mu_\chi(y).\tag{Empirical Hellinger}\label{def:heL_pist}
\]

\begin{remark}~\label{remark:lambda_n}
Observe that $\Hcal(f_1,f_2)$ follows from the standard Hellinger distance between $f_1$ and $f_2$ (see Section 3.3, Page 61 \cite{pollard_user\density_2001}), by setting the integrating measure on $\chi\times\Ibb\times\chi$ to be the empirical measure $ \lambda_n := n^{-1}\sum_{i=0}^{n-1} \delta_{X_i,a_i} \otimes \muc$. It follows that $\Hcal$ is a nonnegative random variable adapted to $\Fcal_0^n$. 
\end{remark}

Let $V_m := \lc \sum_{k\in m} a_k\indicator_k : a_k\geq 0 \ \forall\ k\in m\rc$ be the set of all piecewise constant functions (not necessarily histograms) on partition $m$. 
The following proposition shows that $\hat{\density}_m$ is ``almost'' as good as the best approximation of $\density$ in $V_m$. 
For a set of integrable functions $\Lcal$ and a function $f_1$, define $\Hcal^2(f_1,\Lcal):= \min_{f_2\in\Lcal} \Hcal^2(f_1,f_2)$. The following proposition is a standard first step (see Proposition 2.1 \citep{sart_estimation_2014}, Proof of Theorem 6 \citep{baraud_estimating_2009} etc) that illustrates how $\Hcal$ can be used to choose a good estimator. 

\begin{proposition}~\label{prop:Loss-bound}  For a given transition function $s$, for any partition $m$, the associated estimator $\hat s_m$ satisfies
\[
\expec\lb \Hcal^2(\density,\hat \density_m) \rb \leq 2\expec\lb \Hcal^2(\density,V_m) \rb +\frac{1.5+\log n}{n} |m|. 
\]
\end{proposition}
\begin{remark}~\label{remark:penalty}
Let $L\geq 64$ be a given constant. For convenience of notation, we denote the `penalty' term as \[pen(m) := L (1.5+\log n)|m|/n\numberthis\label{eq:penalty}.\] Because $L$ is known, we have suppressed its dependence from the notation $pen(m)$.     
\end{remark}
    The proof of the previous proposition can be found in Section \ref{sec:prf-lossbd}, and involves showing that $\hat \density_m$ is the approximate projection of $\density$ on the space of all piecewise constant functions $V_m$ with respect to the randomized Hellinger loss function $\mathcal{H}$.

Now we extend Proposition \ref{prop:Loss-bound} to the class of all dyadic partitions on $\chi\times\Ibb\times\chi$. 
To that end, we first recursively define $\Mcal_\lcal$, the set of dyadic partitions of $\chi\times\Ibb\times\chi$  {upto} depth $\lcal$ 
as follows \citep{devore_degree_1990}:
\begin{definition}~\label{def:dyadic-cuts}
    Define $\Mcal_0:=\{\chi\times\Ibb\times\chi\}$. For any $\lcal$, let $m\in\Mcal_\lcal$ and $k \in m$. Thus $k$ is an element of a partition of $\chi\times\Ibb\times\chi$, so that $k\subseteq \Rbb^{d_2+2d_1}$. Let $k_1,k_2,\dots,k_{2^{d_2+2d_1}}$ be the $2^{d_2+2d_1}$ sets obtained by equally dividing $k$ along each axis. Let $\Scal(m, k) := m \bigcup \{k_1,k_2,\dots,k_{2^{d_2+2d_1}}\}\backslash k$. Then 
    \[
    \Mcal_{\lcal+1}:=\lc\bigcup_{m\in \Mcal_\lcal}\bigcup_{k\in m}\Scal(m, k)\rc\bigcup \Mcal_\lcal.
    \]
\end{definition}

To formally write the contrast, we introduce some notation. For any two functions $f_1,f_2:\chi\times\Ibb\times\chi \rightarrow \Rbb$ define $T(f_1,f_2)$ as,
\begin{small}
\begin{align*}
    T(f_1,f_2) &:= \frac{1}{n} \sum_{i=0}^{n-1} \frac{1}{\sqrt{2}} \frac{\sqrt{f_2(X_i,a_i,X_{i+1})} - \sqrt{f_1(X_i,a_i,X_{i+1})}}{\sqrt{f_2(X_i,a_i,X_{i+1}) + f_1(X_i,a_i,X_{i+1})}} \\
    &\qquad \qquad + \int \sqrt{\frac{f_1+f_2}{2}} \cdot (\sqrt{f_2} - \sqrt{f_1}) \, d\lambda_n
    + \int (f_1 - f_2) \, d\lambda_n.\numberthis \label{eq:Tn}
\end{align*}   
\end{small}

\noindent Following similar literature \citep{baraud_estimator_2011,baraud_estimating_2009,sart_density_2023,sart_estimation_2014} we measure the ``goodness" of a partition $m \in\Mcal_\lcal$ compared to all others in $\Mcal_\lcal$ through $\gamma(m)$, defined as 
\begin{small}
    \begin{align*}
        & \gamma (m) := \sum_{K\in m} \sup_{m'\in \Mcal_\lcal}  \lb \frac{3}{4}\lp1-\frac{1}{\sqrt{2}}\rp\Hcal^2(\hat \density_m\indicator_K,\hat \density_{m'}\indicator_K) +\Test(\hat \density_m\indicator_K,\hat \density_{m'}\indicator_K) -pen(m'\vee K)\rb +2\,pen (m)\numberthis\label{def:gamma}
    \end{align*}
\end{small}
\text{ where }
 \[
 m'\vee K := \lc K'\cap K: K'\in m',K'\cap K\neq \text{\O}  \rc.\numberthis\label{eq:vee-def2}
 \]
Since a partition uniquely defines a histogram, the selection procedure we enact requires us to choose a particular partition. Therefore, it is sufficient to use $\gamma$ to select a partition $\hat m$.  {For any given $(\lcal,L)$, }we select the $\hat m$ such that
\begin{align*}
    \gamma(\hat m) \leq \min_{m \in \Mcal_\lcal} \gamma(m) +\frac{1}{n}\tag{Constrast}\label{eq:model}.
\end{align*}
\begin{remark}\label{remark:computation}
    The time complexity of finding $\hat m$ is $\Ocal\lp n\lcal(d_1+d_2)+\lcal2^{(\lcal+1)(d_1+d_2)}\rp$. See \cite[Proposition A.1]{sart_estimation_2014} or \cite[Section 3.2.4]{baraud_estimating_2009} for details.
\end{remark}

{  Observe that $\hat m$ depends \emph{solely} on $\indexeddata$, $\lcal$, and $L$. We define the estimator $\hat \density := \hat \density_{\hat m}$ and highlight its dependence on $\lcal$ and $L$, although we omit these details in the notation for brevity. 

Theorem \ref{thm:main-riskbd} demonstrates that the above estimator $\hat \density$ achieves an oracle risk bound with respect to $\Hcal$. In Section \ref{sec:det-hell} we demonstrate that $\hat \density$ is also optimal under the usual (deterministic) Hellinger loss function.

\begin{theorem}~\label{thm:main-riskbd}
    There exist universal constants $L_0$ and $\Constant$ such that for all $L\geq L_0$ and $\lcal\geq 1$, the estimator $\hat \density$ satisfies 
    \begin{align*}
        \Constant\expec\lb \Hcal^2\lp \density ,\hat \density \rp \rb\leq \inf_{m\in \Mcal_\lcal} \lc \expec\lb \Hcal^2\lp \density ,V_m \rp \rb+pen(m) \rc.
    \end{align*}
\end{theorem}

Observe that Theorem \ref{thm:main-riskbd} does not require any recurrence or mixing assumptions on the controlled Markov chain, indicating that $\hat \density_m$ is  {the best piecewise constant estimator of $\density$ with respect to the loss function $\Hcal$ for the given sample $\{(X_i,a_i)\}$. It is instance-dependent since our choice of empirical Hellinger loss function itself depends upon the sample path. And, by satisfying the oracle risk bound presented in Theorem \ref{thm:main-riskbd}, it becomes the best piecewise constant estimator.} Because the controls $a_i$ may be non-stationary and non-ergodic, this property is even more significant for controlled Markov chains than for stationary ergodic processes such as i.i.d.\ data or Markov chains. To the best of our knowledge, Theorem \ref{thm:main-riskbd} is the only result that provides a risk bound for \textbf{arbitrary} controlled Markov chains. We now turn to prove Theorem \ref{thm:main-riskbd}.

\subsection{Proof of Theorem \ref{thm:main-riskbd}}\label{sec:prf-thmmain}
\begin{proof}
{

For the case $\lcal > n$, we leverage Proposition \ref{prop:Loss-bound} and a union bound to obtain a risk bound over $\Mcal_\lcal$, as demonstrated in equations (\ref{eq:projection_bound}) and (\ref{eq:union-bound}), respectively. }

\textbf{Case I ($\lcal \leq n$):}    
We write the following proposition, whose proof is provided in~\cref{sec:prf-mainconc}:  
    \begin{proposition}~\label{prop:main-concentration}
    For any $\zeta>0$, and for all $L\geq64$ and $1\leq \lcal\leq n$, and a large enough constant $\Constant$, the estimator $\hat \density$ satisfies for any $\density$, 
    \begin{align*}
    &    \prob\lp \Constant\Hcal^2(\density , \hat \density) \geq \inf_{m\in \Mcal_\lcal} \lc \Hcal^2(\density ,   \hat \density_m)+pen(m)\rc+\zeta\rp\leq 6e^{-n\zeta}.\numberthis\label{eq:prop:main-conc}
    \end{align*}
\end{proposition}
 {Recall that for any random variable $X$, $\int_{t>0} P(X>t)dt = \expec[X^+]\geq \expec[X]$, where $X^+=\max (X,0)$. Using this fact and integrating both sides of \cref{eq:prop:main-conc} over $\zeta$, we have 
\begin{align*}
    \expec\lb\Constant\Hcal^2(\density , \hat \density) -\inf_{m\in \Mcal_\lcal} \lc \Hcal^2(\density ,   \hat \density_m)+pen(m)\rc\rb\leq \frac6n.
\end{align*}
The main result now follows by trivially upper bounding $6/n$ by $L(1.5+\log n)|m|/n$ for all non-empty partitions $m$. We move to Case II.}

 { \textbf{Case II ($\lcal \geq n+1$)}  We will show that, when $\lcal\geq n+1$, we the optimal histogram is created by some partition $m^\dagger$ such that $m^\dagger\in \Mcal_n$. The proof will then proceed similarly to \textbf{Case I}. We begin with the following proposition, whose proof can be found in Section \ref{sec:prf-suffdepth}.
\begin{proposition}\label{prop:suffdepth}
    For all $\lcal\geq n+1$, 
    \[
        \inf_{\lcal\in \Mcal_\lcal}\gamma(m)=\inf_{m\in\Mcal_n}\gamma(m)\numberthis\label{eq:gamma_redundancy}.
    \]
\end{proposition} 
    }
 
    Next,  {for any $\lcal\geq n+1$} let 
    \[
    m^\dagger \in \argmin_{m\in \Mcal_\lcal}\lc\expec\lb \Hcal^2\lp \density ,V_m \rp \rb+pen(m)\rc.
    \]
    To complete the proof we need to show $m^\dagger\in\Mcal_n$. Let $\varnothing$ be the trivial partition of $\chi\times\Ibb\times\chi$ and $0_\varnothing\equiv0$ be the trivial piecewise constant function associated with it. We now observe that
    \begin{align*}
        pen(m^\dagger) & \leq \expec\lb\Hcal^2(s ,V_{m^\dagger})\rb+pen(m^\dagger)\\
        & \leq \expec\lb\Hcal^2(s ,V_{\varnothing})\rb+pen(\varnothing)\\
        & \leq \expec\lb\Hcal^2(s ,0_\varnothing)\rb+pen(\varnothing)\numberthis\label{eq:projection_bound}\\
        & = \frac{1}{2}+L\frac{\log n}{n}.
    \end{align*}
        The first inequality follows trivially from the fact that $\Hcal^2(\cdot,\cdot)\geq0$. The second inequality follows from the definition of $m^\dagger$. The third inequality follows from the definition of $\Hcal^2(s ,V_m)$ in Proposition \ref{prop:Loss-bound}. The final equality follows by observing that $\Hcal^2(s ,0_\varnothing)=1/2$ and by substituting the value of $pen(\varnothing)$. Substituting the value of $pen(m^\dagger)$ from \cref{eq:penalty} we now get
    $|m^\dagger|\leq 2+n/(L\log n)$ 
    
     {Recall from Section \ref{sec:introduction} that $n\geq 3$ and from the hypothesis of the Theorem that $L\geq 64$. Therefore,} 
    $2+n/(L\log n)$ is trivially upper bounded by $n$.  {Therefore $|m^\dagger|\leq n$ which in turn implies that $m^\dagger\in\Mcal_n$. The rest of the proof now follows similarly to \textbf{Case I}.} 
\end{proof}

Proposition \ref{prop:main-concentration} is established by verifying that standard results in adaptive estimation of i.i.d (theorem 1 \cite{baraud_estimator_2011}, see also theorem 8 \citep{baraud_estimating_2009}) or Markov chain (theorem B.1 \cite{sart_estimation_2014}) densities canonically extend to the realm of controlled Markov chains. A sketch of the proof is included for the convenience of the reader in Appendix \ref{sec:sketch-prfmainconc}. The complete proof can be found in Section \ref{sec:prf-mainconc}.

}

\section{The Risk Bound for the deterministic Hellinger Loss}\label{sec:det-hell}

{ 

 {As mentioned previously, the empirical Hellinger risk, which was the main focus of the previous section, can be thought of as a risk bound tailored to the given sample $\{(X_i,a_i)\}$ and was therefore, assumption free. In this section, we move on to the deterministic version of the Hellinger loss, which is averaged over all possible sample paths.} 
This brings the two additional challenges that were described in the \textbf{Technical Contributions} paragraph of Section \ref{sec:introduction}. We address these first; beginning with mixing.

\paragraph{Mixing:} In this section, we assume the controlled Markov chain $\{(X_i,a_i)\}$ is \emph{geometrically strongly mixing} \cite{bradley_basic_2005}. The strong mixing coefficient (also referred to as $\alpha$-mixing coefficients) $\alpha_{i,j}$ is defined by
\begin{align*}
    \alpha_{i,j} := \sup_{A,B}\lv\prob\lp \History_{0}^{i}\in A\bigcap\History_{j}^{\infty}\in B \rp-\prob\lp \History_0^i\in A\rp\prob\lp \History_j^\infty\in B \rp\rv, \tag{Strong Mixing Coeff.}\label{def:strong-mixing}
\end{align*}
where $A$ and $B$ are Borel-measureable sets in the $\sigma$-algebras generated by $\History_0^i$ and $\History_j^\infty$ respectively. We refer the readers to \cite{bradley_basic_2005} for a comprehensive treatment of strong mixing coefficients (see also \cite{bhattacharya_explicit_2023} for results on finding explicit constants). We assume the following in the ensuing developments.

\begin{assumption}~\label{assume:alpha-mix}
    There exists a constant $\constant_p$ such that $\alpha_{i,j}\leq e^{-\constant_p(j-i)}$. 
    Observe that under this assumption, $\sup_i\sum_{j\geq i}\sqrt{\alpha_{i,j}}<\infty$ . We define 
     $   \Constant_\Delta:=\sup_i( 1+\sum_{j\geq i}\sqrt{\alpha_{i,j}}) $
    and note that $\Constant_\Delta$ is a positive constant.
\end{assumption}
\begin{remark}
    The term ``exponentially mixing" is commonly used in the literature to describe sequences of random variables whose strong mixing coefficients decay exponentially.
\end{remark}

Our primary motivation for assuming exponential mixing conditions is to utilize the sharp concentration inequalities in \cite{merlevede_bernstein_2009}, which also require exponentially decaying strong mixing coefficients. To the best of our knowledge, there exists no equivalent results which relaxes the assumptions to accommodate polynomially decaying strong mixing coefficients. Any such relaxations would immediately apply to our own results. 

\paragraph{Non-stationarity:} Recall that the sequence $(X_i,a_i)$ can be non-stationary and non-ergodic. In contrast to the Empirical Hellinger defined in eq. (\ref{def:heL_pist}), there is no canonical notion of a deterministic Hellinger loss for such sequences. Consequently, we consider two separate cases: one in which an ergodic occupation measure (Definition \ref{def:erg-occmeas} below) exists (Theorem \ref{thm:detlos-1}), and one in which it does not (Theorem \ref{thm:detlos-2}). The former can be viewed as a generalization of stationarity, while the latter dispenses with stationarity altogether. 
Proposition \ref{prop:1-better-than-2} provides a simple example showing that a sharper bound can be derived by incorporating the ergodic occupation measure than by ignoring it. 

}
{ 
\subsection{Ergodic Occupation Measure Exists}\label{sec:erg-occ-meas-exists}

The ergodic occupation measure was introduced informally in Section \ref{sec:introduction}. We now formalize it by adapting equation 1.3 of \cite{bhatt_occupation_1996} to the discrete time setting.
\begin{definition}\label{def:erg-occmeas}[Ergodic Occupation Measure]
Define the ergodic occupation measure $\nu:\Bcal(\chi\times\Ibb)\rightarrow \Rbb$ as
\begin{align*}
    \nu(\Acal) := \lim_{t\to\infty} \frac{1}{t}\sum_{i=1}^t \prob\lp \lp X_i,a_i\rp\in \Acal \rp.
\end{align*}
\end{definition}

Observe that if $\{(X_i,a_i)\}$ is a strictly stationary sequence, then the ergodic occupation measure exists (i.e., the limit is well-defined) and is given by the marginal distribution of $(X_0,a_0)$. More precisely, 
\begin{align*}
    \lim_{t\to\infty}\frac{1}{t}\sum_{i=1}^t \prob\lp \lp X_i,a_i\rp\in \Acal \rp =  \prob\lp \lp X_1,a_1\rp\in \Acal \rp =\frac{1}{n}\sum_{i=1}^n \prob\lp \lp X_i,a_i\rp\in \Acal \rp.\numberthis\label{eq:occupation-limit}
\end{align*}

\begin{definition}~\label{assume:conv-gap}
       Let $\nu_n(\Acal) := n^{-1}\sum_{i=1}^n \prob\lp \lp X_i,a_i\rp\in \Acal \rp$. We define $r_n:=\lV \nu_n-\nu\rV_{TV}$.
\end{definition}

\begin{remark}\label{remark:stationaryergodic}
    For stationary sequences, $r_n=0$.   
It can also be verified that $r_n\leq \Ocal(1/n)$ holds under more general notions of stationarity, such as $N^\text{th}$-order or semi-stationarity \cite{serfozo_semi-stationary_1972}.
\end{remark}

The following deterministic Hellinger distance is derived from \cref{def:heL_pist} by replacing the empirical measure with the ergodic occupation measure. Formally we define the Hellinger distance $h^2$ as follows:
\[
h^2(f_1,f_2) := \frac{1}{2}\int_{\chi\times\Ibb\times\chi} \lp \sqrt{f_1(x,l,y)}-\sqrt{f_2(x,l,y)}  \rp^2\mu_\chi(dy)\nu(dx,dl).
\]
Let $\hat \density$ be as defined in Section \ref{sec:rand-loss}. We establish the following risk bound, whose proof is in Section \ref{sec:prf-detls}.


\begin{theorem}~\label{thm:detlos-1} Let $m_{ref}\pow 2$ be the partition of $A$ into cubes of edge length $2^{-\lcal}$. Assume $\lc (X_i,a_i)\rc_{i=0}^n$ is a sequence from a controlled Markov chain satisfying Assumption \ref{assume:alpha-mix}. Then, the histogram estimator $\hat \density$ satisfies 
     \begin{align*}
        \Constant\expec\lb h^2\lp \density ,\hat \density \rp \rb\leq \inf_{m\in \Mcal_\lcal} \lc  h^2\lp \density ,V_m \rp +pen(m) \rc+\Rcal(n).
    \end{align*}
    where $\Rcal(n)$ is the following remainder term
    \begin{small}
         \begin{align*}
            \Rcal(n) = 2^{\lcal(d_1+d_2)}\max_{\Scal_{r}\in m_{ref}\pow 2} \exp\lp- \frac{\Constant_pn\nu^2(\Scal_{r})-2n\Constant_pr_n}{4\Constant_\Delta\rho_\star(\Scal_r) +4n^{-1}+2\nu(\Scal_{r})(\log n)^2+2r_n(\log n)^2}\rp+r_n
        \end{align*}   
    \end{small}
    and $\Constant_p$ only depends upon $\constant_p$ in Assumption \ref{assume:alpha-mix}, and
    \[
\rho_\star(\Scal_r) := \sup_i\max\lc \prob((X_i,a_i)\in \Scal_r), \sup_{j>i}\sqrt{\prob\lp (X_i,a_i)\in \Scal_r,(X_j,a_j)\in \Scal_r\rp}\rc.
\]
\end{theorem}
We highlight two key aspects of the previous theorem. First, since $h^2(\cdot,\cdot)\leq 1/2$, Theorem~\ref{thm:detlos-1} is only meaningful if $\Rcal(n)<1/2$. We show that this condition is satisfied whenever $\nu$ admits a density on $A$ that is bounded below by  {a positive constant} $\kcal_0$  (see Corollary \ref{corollary:markov_recovery} below). If $(X_i,a_i)$ is a Markov chain,  {this effectively means that its stationary density is bounded below by $\kcal_0$ on the compact set $A$. In other words, we require that the chain is recurrent on $A$, which is not a stringent requirement.} Second, although the $\rho_\star$ term is slightly unconventional, it is important for preserving the sharpness of the bound. See Remark \ref{remark:important-remark} below for more discussion.

We now show how deterministic risk bounds for i.i.d.\ data (Corollary~2 of \cite{baraud_estimator_2011}) or for stationary Markov chains (Theorem~2.2 of \cite{sart_estimation_2014}) can be recovered as special cases of Theorem~\ref{thm:detlos-1}. For concreteness, we restrict our attention to stationary Markov chains.
\begin{corollary}\label{corollary:markov_recovery}
    Let $\{(X_i,a_i)\}$ be a geometrically strong mixing stationary Markov chain with invariant distribution $\nu$, which is bounded below by $\kcal_0$. Then, for large enough $n$
    \begin{align*}
        \Rcal(n)\leq 2^{\lcal(d_1+d_2)}\exp\lp - \frac{\Constant_p\kcal_0n}{\Constant_\Delta2^{\lcal(d_1+d_2)+3}(\log n)^2}\rp.
    \end{align*}
\end{corollary}

A direct comparison of Corollary \ref{corollary:markov_recovery} with Theorem 2.2 in \cite{sart_estimation_2014} reveals that we recover a sharper bound for $R(n)$ due to our use of the Bernstein's inequality (see Section \ref{sec:prf-markov_recovery} for details). In particular, when $d_1=d_2$, we show that 
\begin{align*}
R(n)\leq\Ocal\lp 2^{2\lcal d}\exp\lp-\frac{\Constant_p\kcal_0n}{\Constant_\Delta2^{2\lcal d+3}(\log n)^2}\rp\rp,    
\end{align*}
whereas \cite{sart_estimation_2014} obtains the bound
\begin{align*}
    \Ocal\lp n^22^{3\lcal d+1} \exp\lp-\sqrt{\frac{n\kcal_0}{(40\times 2^{\lcal d})}}\rp\rp
\end{align*}
which is larger for sufficiently large $n$. We now turn to proving Corollary \ref{corollary:markov_recovery}.
\subsection{Proof of Corollary \ref{corollary:markov_recovery}}\label{sec:prf-markov_recovery}
\begin{proof}
    $(X_i,a_i)$ is stationary. Therefore, as mentioned in Remark \ref{remark:stationaryergodic}, $r_n=0$. Consequently,
    \begin{align*}
        \Rcal(n) = 2^{\lcal(d_1+d_2)}\max_{\Scal_{r}\in m_{ref}\pow 2} \exp\lp- \frac{\Constant_pn\nu^2(\Scal_{r})}{4\Constant_\Delta\rho_\star(\Scal_r) +4n^{-1}+2\nu(\Scal_{r})(\log n)^2}\rp
    \end{align*}
    
    Next, fix a set $\Scal_r\in m_{ref}\pow 2$. We note by stationarity that $\prob((X_i,a_i)\in\Scal_r)=\nu(\Scal_r)$. We first consider the case when $\prob((X_i,a_i)\in\Scal_r)\geq\sup_{j>i}\sqrt{\prob((X_i,a_i)\in\Scal_r,(X_j,a_j)\in\Scal_r)}$, so that $\rho_\star(\Scal_r)$ becomes 
    \begin{align*}
        \rho_\star(\Scal_r) = \nu(\Scal_r).     \end{align*}
    The other case is handled similarly with more careful book-keeping.
    This implies,
    \begin{align*}
        \exp\lp- \frac{\Constant_pn\nu^2(\Scal_{r})}{4\Constant_\Delta\rho_\star(\Scal_r) +4n^{-1}+2\nu(\Scal_{r})(\log n)^2}\rp < \exp\lp- \frac{\Constant_pn\nu^2(\Scal_{r})}{4\Constant_\Delta\nu(\Scal_r) +4n^{-1}+2\nu(\Scal_{r})(\log n)^2}\rp.\numberthis\label{eq:prf-cormarreceq1}
    \end{align*}

    Recall from Assumption \ref{assume:alpha-mix} that $\Constant_\Delta$ is a positive number greater than $1$. Therefore,
    \begin{align*}
        4\Constant_\Delta\nu(\Scal_r) +4n^{-1}+2\nu(\Scal_{r})(\log n)^2\leq 4\Constant_\Delta\nu(\Scal_r)(\log n)^2 +4n^{-1}.
    \end{align*}
    Now, allowing $n$ to be large enough such that $4\Constant_\Delta\nu(\Scal_r)(\log n)^2 \geq 4n^{-1}$ we get
    \begin{align*}
        4\Constant_\Delta\nu(\Scal_r)(\log n)^2 +4n^{-1}\leq 8\Constant_\Delta\nu(\Scal_r)(\log n)^2.
    \end{align*}
    Substituting this upper bound on the right hand side of \cref{eq:prf-cormarreceq1} we get,
    \begin{align*}
        \exp\lp- \frac{\Constant_pn\nu^2(\Scal_{r})}{4\Constant_\Delta\nu(\Scal_r) +4n^{-1}+2\nu(\Scal_{r})(\log n)^2}\rp & \leq \exp\lp- \frac{\Constant_pn\nu^2(\Scal_{r})}{8\Constant_\Delta\nu(\Scal_r)(\log n)^2}\rp\\ & = \exp\lp- \frac{\Constant_pn\nu(\Scal_{r})}{8\Constant_\Delta(\log n)^2}\rp.
    \end{align*}
    $\Scal_r$ is a cube of side length $2^{-\lcal}$ and $\nu$ admits a density lower bounded by $\kcal_0$. Therefore, $\nu(\Scal_r)\geq \kcal_0/2^{\lcal(d_1+d_2)}$. The rest of the proof now follows.
\end{proof}

}
\subsection{Ergodic Occupation Measure Does Not Exist}
{ 

If the limit on the left hand side of \cref{eq:occupation-limit} fails to exist, then the ergodic occupation measure is undefined. This situation arises for non-stationary, non-ergodic processes. To endow such a process with a notion of recurrence, we define the `time to return' as follows
\begin{definition}~\label{def:return-time}
The first \emph{hitting time} $\Scal$ is defined as
\[
\tau_{\Scal}\pow{1}:=\min\lc n: (X_n,a_n)\in \Scal,(X_j,a_j)\notin \Scal \  \forall \ 0\leq j<n \rc.
\]
When $i\geq2$ the $i$-th {\emph{time to return}} (or {\emph{return time}}) of the state-control pair $(x,l)$ is recursively defined as 
\[
\tau^{(i)}_{\Scal}:= \min\lc n:\lp X_{\sum_{k=1}^{i-1}\tau^{(k)}_{x,l}+n},a_{\sum_{k=1}^{i-1}\tau^{(k)}_{x,l}+n}\rp \in \Scal,\lp X_{j},a_{j}\rp\notin \Scal\enspace\forall \enspace \sum_{k=1}^{i-1}\tau^{(k)}_{\Scal}<j<\sum_{k=1}^{i-1}\tau^{(k)}_{\Scal}+n\rc.
\]
\end{definition}
If \(a_i\) depends only on \(X_i\), then \(\{(X_i,a_i)\}\) forms a Markov chain, and \(\{\tau_\Scal\pow{i}\}\) becomes a renewal process \cite{ross_stochastic_1983}. We use this idea to prove a renewal-type result (Lemma~\ref{lemma:KAC-lower}) that counts the number of occurrences of \(\Scal\). In contrast to Harris recurrent processes, we do not assume independent renewals \citep{glynn_wide-sense_2011,glynn_new_2023}, making our results applicable in a broader setting.

We now introduce some notation.  {We define the maximum expected return time to $\Scal$ as $T(\Scal)$ and recall the definition of $\nu_n(\Scal)$ from the introduction.} Formally, 
\begin{align*}
    T(\Scal):=\sup_{i}\E[\tau^{(i)}_{\Scal}|{\Fcal_{\sum_{p=0}^{i-1}\tau_{x,l}\pow{p}}}], \text{ and } \nu_n(\Scal) = \frac{1}{n}\sum_{i=1}^n \prob\lp \lp X_i,a_i\rp\in \Scal \rp, \text{ respectively.}\numberthis\label{eq:return_time_def}
\end{align*}

Lemma \ref{lemma:erg-vs-recurring} (proved in Section \ref{sec:prf-ergvsrec}) establishes that having $T(\Scal_\star)<\infty$ does not, by itself, imply that $\lim_{n\rightarrow\infty} \nu_n(\Scal_\star)$ is well defined.

\begin{lemma}\label{lemma:erg-vs-recurring}
    There exist controlled Markov chains for which $T(\Scal)<\infty$ and $\nu(\Scal)$ does not exist for any $\Scal\subset\chi\times\Ibb$. 
\end{lemma}
We prove Lemma \ref{lemma:erg-vs-recurring}  by producing an i.n.i.d sequence. Thus, the counterexample is both recurrent and mixing without being ergodic.
Next, we define the Hellinger distance with respect to $\nu_n$ as
\[
h_n^2(f_1,f_2) := \frac{1}{2}\int_{\chi\times\Ibb\times \chi} \lp \sqrt{f_1(x,l,y)}-\sqrt{f_2(x,l,y)}  \rp^2\mu_\chi(dy)\nu_n(dx,dl).
\]

Choose a depth $\lcal\leq n$ and let $m_{ref}\pow 2$ be the partition of $\chi\times\Ibb$ into uniform cubes of edge length $2^{-\lcal}$. To avoid trivialities, we implicitly assume throughout the rest of this section that $T(\Scal)<\infty$ for any $\Scal\in m_{ref}\pow 2$. We interpret this condition to mean that the controlled Markov chain $\{(X_i,a_i)\}$ is recurrent on open subsets of $\chi\times\Ibb$. This enforces a notion of recurrence even for non-stationary processes and allows us to establish the non-ergodic analogue of Theorem \ref{thm:detlos-1} in Theorem \ref{thm:detlos-2} next; the proof is relegated to Section \ref{sec:prf-detls2}.

}

\begin{theorem}~\label{thm:detlos-2} Let $m_{ref}\pow 2$ be the partition of $\chi\times\Ibb$ into uniform cubes of edge length $2^{-\lcal}$. Define $\Scal_\star$ as 
\[
\Scal_\star:= \argmax_{\Scal_r\in m_{ref}\pow 2} \exp\lp- \frac{ \frac{\Constant_pn}{4T(\Scal_r)^2}}{4\Constant_\Delta\rho_\star(\Scal_r) +4n^{-1}+\frac{(\log n)^2}{2T(\Scal_r)}}\rp,
\]
where $\Constant_\Delta$ is as in Assumption \ref{assume:alpha-mix}, $\Constant_p$ only depends upon $\constant_p$ in Assumption \ref{assume:alpha-mix}, and
\[
\rho_\star(\Scal_r) := \sup_i\max\lc \prob((X_i,a_i)\in \Scal_r), \sup_{j>i}\sqrt{\prob\lp (X_i,a_i)\in \Scal_r,(X_j,a_j)\in \Scal_r\rp}\rc.
\]
With, $n\geq 2T(\Scal_\star)$, assume that $\lc (X_i,a_i)\rc_{i=0}^n$ is a sequence from a controlled Markov chain satisfying Assumption \ref{assume:alpha-mix}. Then, the histogram estimator $\hat s$ satisfies the following risk bound
     \begin{align*}
        \Constant\expec\lb h_n^2\lp \density ,\hat \density \rp \rb\leq \inf_{m\in \Mcal_\lcal} \lc  h_n^2\lp \density ,V_m \rp +pen(m)\rc+\Rcal(n) .
    \end{align*}
    where the remainder term satisfies
    \begin{small}
       $ 
            \Rcal(n) = 2^{\lcal(d_1+d_2)} \exp\lp- \frac{ \frac{\Constant_pn}{4T(\Scal_\star)^2}}{4\Constant_\Delta\rho_\star(\Scal_\star) +\frac{4+(\log n)^2}{2T(\Scal_\star)}}\rp.
   $ 
    \end{small}
\end{theorem}

{ 
\begin{remark}\label{remark:important-remark}
    We remark on two important aspects of the previous theorem, both of which are related to the remainder term $\Rcal(n)$. On the one hand, as noted earlier, the risk bound is only meaningful if $\Rcal(n)<1/2$ which requires $T(\Scal_\star)<\infty$. 

Second, although the term $\rho_\star$ may initially appear unusual, it is 
instrumental in proving Corollary~\ref{corollary:markov_recovery} and establishing the lower bound in Theorem~\ref{thm:detls2-lb}. $\rho_\star$ arises in the proof of Theorem~\ref{thm:detlos-2} when {we use a Bernstein inequality coupled with a covariance bound for strongly mixing random variables (Lemma~\ref{lemma:alpha-covbound})} to bound a covariance term (\cref{eq:cov-bound}).

If one were to trivially set $\rho_\star = 1$ or rely on weaker Hoeffding-type inequalities for non-stationary processes (e.g., theorem~1.2 of \cite{kontorovich_concentration_2008}), the lower bound would degrade to the point of losing its minimax sharpness. Such connections between concentration inequalities and the precision of resulting bounds are well-established in the literature; see section~1.2 of \citep{massart_concentration_2007} for a detailed discussion.
\end{remark}

{A natural question concerns the optimality of the previous bound. The following theorem addresses this issue by demonstrating the minimax-optimality (described below in \cref{eq:minimax}) of the estimator up to poly-$\log$ order terms.}

\begin{theorem}~\label{thm:detls2-lb}
Assume the conditions of Theorem \ref{thm:detlos-2}, and define $\tilde\Scal_\star:=\argmax_{\Scal\in m_{ref}\pow 2}T(\Scal)$. 
\begin{enumerate}
    \item If 
    \[
    \frac{n}{(\log n)^3}\geq \constant\Constant_p^{-1}T(\Scal_\star)^2\lp \Constant_\Delta\rho_\star(\Scal_\star)+\frac{1}{T(\Scal_\star)}\rp\log\lp T\lp\tilde \Scal_\star\rp\rp\numberthis\label{eq:thmdetls2-eq1}
    ,\]
    then $\Rcal(n)\leq 4/n$.
    \item If  
    $
     n\leq \Constant_p^{-1}T(\Scal_\star)^2\lp \Constant_\Delta\rho_\star(\Scal_\star)+\frac{1}{T(\Scal_\star)}\rp,
    $
    then $\Rcal(n)>1/2$, and 
    the minimax risk satisfies
    \[    \inf_{\hat\density}\sup_\density\expec[h_n^2(\density ,\hat \density)]\leq \frac{1}{2(1+\pi^2)}\numberthis\label{eq:minimax}
    \]
    where the infimum is over the class of all possible estimators and the supremum is over the class of all possible controlled Markov chains satisfying our assumptions.
    \end{enumerate}
\end{theorem}

\begin{proof}
The proof is divided in two cases. When $\lcal\leq n$, the proof follows from Proposition \ref{prop:detls2} in Section \ref{sec:prf-detls2lbub}. 

Next, when $l\geq n+1$ it follows similarly to the proof of \textbf{Case II}, Theorem \ref{thm:main-riskbd} that the optimal histogram is created by some partition $m^\dagger$ such that $m^\dagger\in \Mcal_n$.  This completes the proof.

\end{proof}

}

{

A final question concerns whether the utility of considering the ergodic occupation measure described in Section \ref{sec:erg-occ-meas-exists} when Theorem \ref{thm:detlos-2} proves a risk bound under a more general setting. The benefit is in the inherent tightness that an average case object like the ergodic occupation measure provides over a worst case statistic like the maximum expected return time. In this situation, $\nu$ is smaller than $T$ and Theorem \ref{thm:detlos-1} provides a tighter bound than \ref{thm:detlos-2}. We make this concrete with the following Proposition.

\begin{proposition}\label{prop:1-better-than-2}
   Let $\Rcal\pow 1(n)$ be the remainder term obtained from Theorem \ref{thm:detlos-1} and $\Rcal\pow 2(n)$ be the remainder term obtained from Theorem \ref{thm:detlos-2}. Then there exists a controlled Markov chain such that $\Rcal\pow 2(n)=\Ocal(pen(m^\star))$ and $\Rcal \pow 1(n)=\ocal(pen(m^\star))$, where $m^\star$ is the partition minimising the oracle risk.
\end{proposition}

The broad idea behind the proof is to compare remainder terms of a time-inhomogenous Markov chain with carefully chosen piecewise-constant densities. It demonstrates that under appropriate choices, the first risk term is negligible compared to the second. See Section \ref{sec:prf-1betthan2} for full details.
}

\section{Applications}

In this section we show the applicability of Theorem \ref{thm:main-riskbd} by deriving uniform risk bounds when $\density$ lies in a given smooth functional class.  We also demonstrate the applicability of Theorem \ref{thm:detlos-2} for controlled Markov chains by showing that its conditions hold with mild and practical assumptions. We start with the former.

\subsection{Uniform Risk Bounds over Functional Classes}\label{sec:empiricalriskbounds}
Here we show that the empirical Hellinger loss recovers optimal rates of convergence over classes of H\"older smooth functions \cite[Chapter 6]{bergh_interpolation_1976} functions. For the purpose of illustration, we assume that $d_1=d_2=d$.
\begin{definition}~\label{def:holder-space}
    We call a function $f:A\rightarrow\Rbb$ to belong to the H\"older space $\Hbb_\sigma(A)$ with parameter $\sigma\in(0,1]$ and finite norm $\|f\|_\sigma>0$ if $|f(x)-f(y)|\leq \|f\|_\sigma \|x-y\|^\sigma\ \forall x,y\in A$.
\end{definition}
Any $f\in\Hbb^\sigma(A)$ is called H\"older smooth. Recall that $\Hbb^1(A)$ is the space of all Lipschitz smooth functions, and that elements of $\Hbb^\sigma(A)$ are constant functions when $\sigma>1$. In particular, all non-constant continuously differentiable functions belong to $\Hbb^\sigma(A)$ for some $\sigma\in(0,1]$. When $\sqrt{\density }$ (where $s$ is the transition kernel corresponding to the controlled Markov chain) belongs to $\Hbb^\sigma(A)$, we have the following corollary to Theorem \ref{thm:main-riskbd}.
\begin{corollary}~\label{cor:holder}
   For all $\sigma\in(0,1]$, and $\sqrt{\density }\in \Hbb^\sigma(A)$, the estimator $\hat \density$ satisfies with an universal constant $\Constant>0$,
   \[
    \Constant\expec\lb \Hcal^2\lp \density ,\hat \density \rp \rb\leq (d\|\sqrt{\density }\|_\sigma)^{2d/(d+\sigma)}\lp \frac{\log n}{n} \rp^{\sigma/(d+\sigma)}+\frac{\log n}{n}.    
   \]
\end{corollary}

Next, we derive a risk bound for functions belonging to isotropic Besov spaces.
\begin{definition}~\label{def:besov-space}
   Given a function $f\in L_p(\Omega),0<p\leq\infty$, and any integer $r$, we defne its
modulus of smoothness of order $r$ as
$$\omega_r(f,t)_p:=\sup_{0<|h|\leq t}\|\Delta_h^r(f,\cdot)\|_{L_p(\Omega)},\quad t>0,$$
where $h\in\mathbb{R}^d$ and $|h|$ is it Euclidean norm. Here, $\Delta_h^r$, is the $r$-th difference
operator,  defined by
$$\Delta_h^r(f,x):=\sum_{k=0}^r(-1)^{r-k}\binom{r}{k}f(x+kh),\quad x\in\Omega\subset\mathbb{R}^d,$$
where this difference is set to zero whenever one of the points $x+kh$ is not in the support of $f$. It is easy to see that for any $f\in L_p(\Omega)$, we have $\omega_{\boldsymbol{r}}(f,t)_p\to0$, Then, Besov space $\Bbba_q^\sigma(L_p(A))$ consists of all $f$ such functions such that
    \begin{align*}
        |f|_{\Bbba_q^\sigma(L_p(A))} := \begin{cases}
            \int_{t>0} t^{q\sigma-1}(\omega_r(f,t)_p)^q dt \quad 0<q<\infty\\
            \sup_{t\geq 0} t^{q\sigma-1}(\omega_r(f,t)_p)^q \quad q=\infty
        \end{cases}
    \end{align*}
    is finite.
    Then, we define $\Bbba^\sigma(L_p(A))$ as
    \begin{align*}
    \Bbba^\sigma(L_p(A)):= 
        \begin{cases}
            \Bbba_p^\sigma(L_p(A)), \qquad p\in(1,2)\\
            \Bbba_\infty^\sigma(L_p(A)), \qquad p \geq 2
        \end{cases}
    \end{align*}
    with the attached norm $\lV\cdot\rV_{\sigma,p}$.
\end{definition}

\begin{assumption}\label{assume:besov} 
We make the following assumptions:
    \begin{enumerate}
        \item Let $p\in (2d/(d+1),\infty)$, $\sigma\in(2d(1/p-1/2)_+,1)$, and $\sqrt{\density }\in  \Bbba^\sigma(L_p(A))$.
        \item For each $i$, $(X_i,a_i)$ admits the density $\Phi_i$ such that $\Phi_i(x,l)\leq \Gamma$ for all $(x,l)\in\chi\times\Ibb$.
    \end{enumerate}
\end{assumption}
{Recall from the Section \ref{sec:introduction} the definition of $\Vol{\cdot}$. Then we have the following corollary.}
\begin{corollary}~\label{cor:besov} 
    Under Assumption \ref{assume:besov}, the estimator $\hat \density=\hat \density(L_0,\infty)$ satisfies
    \[
    \Constant' \expec\lb \Hcal^2(\density ,\hat \density) \rb\leq \Gamma\Vol{A}\| \sqrt{\density } \|_{p,\sigma}^{2d/(d+\sigma)}\lp \frac{\log n}{n} \rp^{\sigma/(\sigma+d)}+\frac{\log n}{n},
    \]
    where $\Constant'>0$ depends only on $\Gamma, \sigma, d,p$ and $\Vol A$ where $\Vol A$ is the volume of the set $A$.
\end{corollary}
The proofs of Corollaries \ref{cor:holder}, and \ref{cor:besov} follow similarly to the proof of \cite[Proposition 3]{baraud_estimating_2009} and we provide a brief sketch in Section \ref{sec:prf-corbesov} 

\subsection{Estimating the Transition Density of Fully Connected Markovian CMC's}\label{sec:time-inhomogenous}
{ 
In this section, we focus on {\bf fully connected} CMC's. A CMC \(\{(X_i,a_i)\}\) is {fully connected} if there exists some \(\epsilon_0>0\) such that for all \(x,l,y \in \chi\times\Ibb\times\chi\),
\begin{align*}
    &\epsilon_0 \leq \density(x,l,y)\leq 1/\epsilon_0, 
    \tag{Fully Connected}\label{eq:uniell-1}
\end{align*}
which endows \(\{(X_i,a_i)\}\) with recurrence and mixing. Our notion of fully connected generalizes the class of inhomogeneous Markov chains first introduced in 
\cite{dobrushin_central_1956-1,dobrushin_central_1956}---and subsequently expanded in \cite{merlevede_local_2021,merlevede_local_2022}---to the setting of controlled Markov chains. 

A CMC is said to have {`Markov controls'} if for any $\Scal_\Ibb\in\Ibb$
\begin{align*}
    \prob\lp a_i\in \Scal_\Ibb|X_i=x, \History_0^{i-1}=\history_0^{i-1} \rp = \prob\lp a_i\in \Scal_\Ibb|X_i=x\rp.
\end{align*}
\begin{remark}
    The dependence of $a_i$ on $X_i$ and $i$ can be non-trivial. If there is no dependence on $i$, then $\{(X_i,a_i)\}$ is a regular Markov chain. If there is no dependence on $X_i$, then $\{X_i\}$ is a regular time-inhomogenous Markov chain. 
\end{remark}

{Our objective in this section will be to show the recurrence and mixing properties of a fully connected Markovian CMC. In particular, we will show that a fully connected Markovian CMC satisfies Assumption \ref{assume:alpha-mix}, and derive the rate constant. Then we will derive an expression for $T(S)$. We first address} mixing by presenting a more general lemma, from which the mixing properties of fully connected CMCs follow as an immediate corollary.

\begin{lemma}\label{lemma:mixing-lemma}
Let \(\indexeddata\) be a CMC with transition density \(\density\) and Markov controls. If there exist \(\chi_0 \subseteq \chi\) and \(\kappa\) such that
\[
\inf_{x\in\chi,\,l\in\Ibb} \density(x,l,y)\;\ge\;\kappa
\quad \text{for all }y\in\chi_0,
\]
then
\[
\alpha_{i,j} \;\le\; \bigl(1-\Vol{\chi_0}\,\kappa\bigr)^{\,j-i-1},
\]
where \(\Vol{\chi_0}\) denotes the ``volume'' of the set \(\chi_0\). Consequently, this CMC satisfies Assumption~\ref{assume:alpha-mix} with \(\Constant_\Delta=1/(\Vol{\chi_0}\,\kappa)\).
\end{lemma}

Applying Lemma~\ref{lemma:mixing-lemma} with \(\chi_0=\chi\) and \(\kappa=\epsilon_0\) immediately shows that a fully connected controlled Markov chain satisfies Assumption~\ref{assume:alpha-mix} with \(\Constant_\Delta = (\epsilon_0\Vol{\chi})^{-1}\).
Moreover, we note that the proof of Lemma~\ref{lemma:mixing-lemma} actually shows something stronger: a fully connected CMC is $\phi$-mixing \citep{bradley_basic_2005}. The full proof, found in Section~\ref{sec:prf-mixinglemma}, generalizes a classical result by Wolfowitz \citep{wolfowitz_products_1963} on products of matrices.

Turning to recurrence, we introduce some notation for the sake of exposition. Let \(\Scal\subseteq \chi\times\Ibb\), and $\Scal_\chi$ and $\Scal_\Ibb$ be such that $\Scal_\chi=\{x\in\chi:(x,l)\in \Scal \text{ for some } l\in\Ibb\}$ and $\Scal_\Ibb = \{l\in\Ibb:(x,l)\in\Scal\text{ for some } x\in\chi\}$.

\begin{definition}
Define \(\tau_{\Scal}\pow{i,\star,j}\) to be the time between the \((j-1)\)- and \(j\)-th visits to \(\Scal_\Ibb\) after the \(i\)-th visit to the state-control pair \(\Scal\). For convenience, let
\[
    \tau_\star 
    \;=\;
    \sum_{k=1}^i 
        \tau_{\Scal}\pow{k}
    \;+\;
    \sum_{k=1}^{j-1}
        \tau^{(i,\star,k)}_{\Scal}.
\]
Then \(\tau_{\Scal}\pow{i,\star,j}\) is recursively defined as
\begin{align*}
\tau_{\Scal}\pow{i,\star,j}:= \min\{n:(a_{\tau_\star+n}\in \Scal_\Ibb),a_{j}\notin \Scal_\Ibb\enspace\forall \enspace \tau_\star<j<\tau_\star+n\}.
\end{align*}
Further, define
\[
T\pow\star(\Scal):= \sup_{i,{j}\geq 0} \expec\bigl[\tau_{\Scal}\pow{i,\star,j}\,\bigm|\,
    \Fcal_{\sum_{p=1}^{i-1} \tau_{\Scal}\pow{p}+\sum_{p=1}^{j-1}\tau_{\Scal}\pow{i,\star,p}}
\bigr].
\]
\end{definition}
The following proposition establishes the return-time properties of fully connected CMCs. Its proof is in Section~\ref{sec:prf-rtm}.

\begin{proposition}\label{prop:return-time-markov}
For all \(\,(i,\Scal)\in\naturalset\times\chi\times\Ibb\), it holds {$\prob$-almost everywhere} that
\begin{small}
    \begin{align*}
        \expec\left[\tau_{\Scal}^{(i)} \,\bigm|\,
            \Fcal_{\sum_{p=1}^{i-1} \tau_{\Scal}\pow{p}}
        \right]
        <\frac{T\pow\star(\Scal)}{\epsilon_0^3\Vol{\Scal_\chi}}.
        \numberthis\label{eq:markov-eq3}         
    \end{align*}
\end{small}
\end{proposition}
\begin{remark}
    The bound in \cref{eq:markov-eq3} can be improved by a more careful (but considerably more tedious) bookkeeping, but this is sufficient for the purposes of illustration.
\end{remark}

}

\subsection{Estimating the Transition Density of Fully Connected non-Markovian CMC's}

The previous Section addressed fully connected Markov chains with Markovian controls, which sufficed to ensure mixing. Here, we remove the Markovianity assumption on the controls and instead consider general sequences of minorized \(\alpha\)-mixing controls. 

To clarify the setup, we introduce additional notation. We call the sequence of controls $a_i$ minorized by $\Vcal$ if there exists a positive measure $\Vcal$ on $\Ibb$ such that $\Vcal(\Ibb)\leq 1$, and
\begin{align*}
   \inf_{\substack{A\in \Fcal_{0}^{p-1},\\C\subseteq\chi, D\subseteq \Ibb}} \prob\lp a_p\in D|X_p\in C, A \rp \geq \Vcal(D). \tag{Minorisation}\label{eq:minorisation}
\end{align*}

If \(\{a_i\}\) itself forms a Markov chain, then taking \(C\times A\) as a ``small set'' recovers the usual notion of minorization for Markov chains; see \cite{meyn_markov_2012} for details. It remains unclear whether an analogous concept of small sets exists for controlled Markov chains, but the presence of such sets would immediately generalize the condition in eq. (\ref{eq:minorisation}) above.
To make a non-Markovian controlled Markov chain tractable for analysis, we impose the following:

\begin{assumption}\label{assume:non-markov}
    The controlled Markov chain \(\{(X_i,a_i)\}\) is geometrically $\alpha$-mixing, fully connected and satisfies the condition in eq. \eqref{eq:minorisation} with a measure \(\Vcal\) whose Radon–Nikod\'ym derivative with respect to $\mu_\chi\otimes\mu_\Ibb$ is bounded below by \(\epsilon_1>0\).
\end{assumption}
This leads us to the following Proposition.
\begin{proposition}\label{prop:non-markov}
    Let $\{(X_i,a_i)\}$ be a controlled Markov chain satisfying Assumption~\ref{assume:non-markov}. Then it is geometrically fast \(\alpha\)-mixing and satisfies the following bound on expected return times:
    \begin{align*}
        T(\Scal)\leq \frac{\epsilon_0\epsilon_1\Vol{\Scal}}{ 1- \epsilon_0\epsilon_1\Vol{\Scal}}+1.
    \end{align*}
\end{proposition}

Our strategy to prove this result will be to dominate the tail probability $\{\tau_{\Scal}\pow i>p,p\in\Nbb\}$ with the tail probability of a geometric distribution whose probability of success is $\epsilon_0\epsilon_1\Vol{\Scal}$. See Section \ref{sec:prf-nonmarkov} for complete details.

The main point of this section is not merely Proposition~\ref{prop:non-markov}, but rather that condition in eq. \eqref{eq:minorisation} alone is insufficient to guarantee both recurrence and mixing in the controlled Markov chain. Lemma~\ref{lemma:mino-but-not-mixing} establishes this formally, and its proof (deferred to Section~\ref{sec:prf-minbnmix}) provides a concrete counterexample.

\begin{lemma}\label{lemma:mino-but-not-mixing}
    There exists a controlled Markov chain that satisfies the condition in eq. \eqref{eq:minorisation} but whose \(\alpha\)-mixing coefficients remain uniformly bounded away from zero.
\end{lemma}

Lemma~\ref{lemma:mino-but-not-mixing} does not imply that deterministic risk bounds \emph{cannot} be derived for chains failing Assumption~\ref{assume:non-markov}; it merely shows our two assumptions are not redundant. However, if \(\{a_i\}\) is a Markov chain, then the condition in eq. \eqref{eq:minorisation} allows a Nummelin split \cite[Chapter 5]{meyn_markov_2012} which opens up a plethora of tools to derive its mixing properties.

\section{Conclusions}
In this paper, we provide two flavors of risk bounds for estimation of the transition functions of controlled Markov chains with continuous states and controls. The first (Theorem \ref{thm:main-riskbd}) is tailored to the particular observed sample $\{(X_i,a_i)\}$ and is assumption free, while the second (Theorems \ref{thm:detlos-1} and \ref{thm:detlos-2}) are oracle risk bounds under assumptions on the recurrence and mixing conditions. 
This address several open problems posed in previous work \citep{banerjee_off-line_2025}, like data-dependent risk bounds, and risk bounds for controlled Markov chains with continuous state-control spaces.

To conclude, we list a few directions for future research.
 Our deterministic guarantees rely on geometric $\alpha$-mixing; existing concentration technology does not yet deliver comparably sharp bounds under summable mixing conditions. Relaxing this requirement without sacrificing tightness is an open problem. Doing this requires developing Bernstein-type inequalities for processes whose strong-mixing coefficients decay only polynomially.  Moreover, while histograms confer interpretability and computational tractability, they may suffer in very high dimensions, suggesting that wavelets or spline based methods could yield further computational gains \citep{loffler_spectral_2021}. Integrating adaptive partitioning schemes with dimension-reduction (like PCA or its variants \citep{datta_consistency_2023}) or representation-learning techniques promises to scale the methodology to higher-dimensional state–control spaces. 
 
 Looking forward, the important question of developing hypothesis tests and resampling techniques \citep{banerjee_clt_2025} for transition probabilities remains unsolved, and we plan to address this question in a future work. Broadly, the risk bounds obtained in this paper lay a principled foundation for offline reinforcement learning \citep{sutton_reinforcement_2018}---like estimating the value-, Q-, and advantage- functions for offline MDP's--- and online control problems, like optimal controls for Guassian \citep{krishnamurthy_partially_2016}, and non-Gaussian \citep{goggin_convergence_1992,banerjee_goggins_2025} POMDP's.

\sloppy
\printbibliography

@article{athreya_kernel_1998,
	title = {Kernel {Estimation} for {Real}-{Valued} {Markov} {Chains}},
	volume = {60},
	issn = {0581-572X},
	url = {https://www.jstor.org/stable/25051178},
	number = {1},
	urldate = {2022-12-01},
	journal = {Sankhyā: The Indian Journal of Statistics, Series A (1961-2002)},
	author = {Athreya, Krishna B. and Atuncar, Gregorio S.},
	year = {1998},
	note = {Publisher: Springer},
	pages = {1--17},
}

@misc{loffler_spectral_2021,
	title = {Spectral thresholding for the estimation of {Markov} chain transition operators},
	url = {http://arxiv.org/abs/1808.08153},
	doi = {10.48550/arXiv.1808.08153},
	urldate = {2022-11-15},
	publisher = {arXiv},
	author = {Löffler, Matthias and Picard, Antoine},
	month = oct,
	year = {2021},
	note = {arXiv:1808.08153 [math, stat]},
}

@inproceedings{hajnal_weak_1958,
	title = {Weak ergodicity in non-homogeneous {Markov} chains},
	volume = {54},
	booktitle = {Mathematical {Proceedings} of the {Cambridge} {Philosophical} {Society}},
	publisher = {Cambridge University Press},
	author = {Hajnal, John and Bartlett, Maurice S},
	year = {1958},
	pages = {233--246},
}

@article{billingsley_statistical_1961,
	title = {Statistical methods in {Markov} chains},
	journal = {The Annals of Mathematical Statistics},
	author = {Billingsley, Patrick},
	year = {1961},
	note = {Publisher: JSTOR},
	pages = {12--40},
}

@article{hernandez-lerma_recurrence_1991,
	title = {Recurrence conditions for {Markov} decision processes with {Borel} state space: a survey},
	volume = {28},
	number = {1},
	journal = {Annals of Operations Research},
	author = {Hernández-Lerma, Onésimo and Montes-de-Oca, Raúl and Cavazos-Cadena, Rolando},
	year = {1991},
	note = {Publisher: Springer},
	pages = {29--46},
}

@book{sutton_reinforcement_2018,
	title = {Reinforcement learning: {An} introduction},
	publisher = {MIT press},
	author = {Sutton, Richard S and Barto, Andrew G},
	year = {2018},
}

@article{wolfowitz_products_1963,
	title = {Products of indecomposable, aperiodic, stochastic matrices},
	volume = {14},
	number = {5},
	journal = {Proceedings of the American Mathematical Society},
	author = {Wolfowitz, Jacob},
	year = {1963},
	note = {Publisher: JSTOR},
	pages = {733--737},
}

@book{gut_probability_2005,
	title = {Probability: a graduate course},
	volume = {5},
	publisher = {Springer},
	author = {Gut, Allan and Gut, Allan},
	year = {2005},
}

@article{ghosh_probability_2002,
	title = {Probability inequalities related to {Markov}'s theorem},
	volume = {56},
	number = {3},
	journal = {The American Statistician},
	author = {Ghosh, BK},
	year = {2002},
	note = {Publisher: Taylor \& Francis},
	pages = {186--190},
}

@inproceedings{birge_model_2006,
	title = {Model selection via testing: an alternative to (penalized) maximum likelihood estimators},
	volume = {42},
	booktitle = {Annales de l'{IHP} {Probabilités} et statistiques},
	author = {Birgé, Lucien},
	year = {2006},
	note = {Issue: 3},
	pages = {273--325},
}

@book{meyn_markov_2012,
	title = {Markov chains and stochastic stability},
	publisher = {Springer Science \& Business Media},
	author = {Meyn, Sean P and Tweedie, Richard L},
	year = {2012},
}

@book{tsybakov_introduction_2009,
	title = {Introduction to {Nonparametric} {Estimation}.},
	publisher = {Springer},
	author = {Tsybakov, Alexandre B},
	year = {2009},
	note = {Publication Title: Springer series in statistics},
}

@article{kontorovich_concentration_2008,
	title = {Concentration inequalities for dependent random variables via the martingale method},
	volume = {36},
	number = {6},
	journal = {Annals of Probability},
	author = {Kontorovich, Leonid Aryeh and Ramanan, Kavita and {others}},
	year = {2008},
	note = {Publisher: Institute of Mathematical Statistics},
	pages = {2126--2158},
}

@article{dobrushin_central_1956,
	title = {Central limit theorem for nonstationary {Markov} chains. {II}},
	volume = {1},
	number = {4},
	journal = {Theory of Probability \& Its Applications},
	author = {Dobrushin, Roland L’vovich},
	year = {1956},
	note = {Publisher: SIAM},
	pages = {329--383},
}

@article{dobrushin_central_1956-1,
	title = {Central limit theorem for nonstationary {Markov} chains. {I}},
	volume = {1},
	number = {1},
	journal = {Theory of Probability \& Its Applications},
	author = {Dobrushin, Roland L},
	year = {1956},
	note = {Publisher: SIAM},
	pages = {65--80},
}

@article{merlevede_bernstein_2009,
	title = {Bernstein inequality and moderate deviations under strong mixing conditions},
	volume = {5},
	journal = {High dimensional probability V: the Luminy volume},
	author = {Merlevède, Florence and Peligrad, Magda and Rio, Emmanuel and {others}},
	year = {2009},
	pages = {273--292},
}

@article{bradley_basic_2005,
	title = {Basic {Properties} of {Strong} {Mixing} {Conditions}. {A} {Survey} and {Some} {Open} {Questions}},
	volume = {2},
	url = {"https://doi.org/10.1214/154957805100000104"},
	doi = {10.1214/154957805100000104},
	journal = {Probability Surveys},
	author = {Bradley, Richard C.},
	year = {2005},
	note = {Publisher: "The Institute of Mathematical Statistics and the Bernoulli Society"},
	pages = {107--144},
}

@article{apostol_elementary_1999,
	title = {An elementary view of {Euler}'s summation formula},
	volume = {106},
	number = {5},
	journal = {The American Mathematical Monthly},
	author = {Apostol, Tom M},
	year = {1999},
	note = {Publisher: Taylor \& Francis},
	pages = {409--418},
}

@article{mania_active_2020,
	title = {Active learning for nonlinear system identification with guarantees},
	journal = {arXiv preprint arXiv:2006.10277},
	author = {Mania, Horia and Jordan, Michael I and Recht, Benjamin},
	year = {2020},
}

@article{levine_offline_2020,
	title = {Offline reinforcement learning: {Tutorial}, review, and perspectives on open problems},
	journal = {arXiv preprint arXiv:2005.01643},
	author = {Levine, Sergey and Kumar, Aviral and Tucker, George and Fu, Justin},
	year = {2020},
}

@article{borkar_topics_1991,
	title = {Topics in controlled {Markov} chains},
	author = {Borkar, Vivek S and Borkar, Vivek S},
	year = {1991},
	note = {Publisher: Longman Scientific \& Technical Harlow},
}

@inproceedings{sart_estimation_2014,
	title = {Estimation of the transition density of a {Markov} chain},
	volume = {50},
	booktitle = {Annales de l'{IHP} {Probabilités} et statistiques},
	author = {Sart, Mathieu},
	year = {2014},
	note = {Issue: 3},
	pages = {1028--1068},
}

@article{mutti_importance_2022,
	title = {The {Importance} of {Non}-{Markovianity} in {Maximum} {State} {Entropy} {Exploration}},
	journal = {arXiv preprint arXiv:2202.03060},
	author = {Mutti, Mirco and De Santi, Riccardo and Restelli, Marcello},
	year = {2022},
}

@article{glynn_new_2023,
    title = {On a {New} {Characterization} of {Harris} {Recurrence} for {Markov} {Chains} and {Processes}},
    volume = {11},
    copyright = {http://creativecommons.org/licenses/by/3.0/},
    issn = {2227-7390},
    url = {https://www.mdpi.com/2227-7390/11/9/2165},
    doi = {10.3390/math11092165},
    language = {en},
    number = {9},
    urldate = {2025-03-01},
    journal = {Mathematics},
    author = {Glynn, Peter and Qu, Yanlin},
    month = jan,
    year = {2023},
    note = {Number: 9
Publisher: Multidisciplinary Digital Publishing Institute},
    pages = {2165},
}

@article{baraud_estimating_2009,
    title = {Estimating the intensity of a random measure by histogram type estimators},
    volume = {143},
    issn = {1432-2064},
    url = {https://doi.org/10.1007/s00440-007-0126-6},
    doi = {10.1007/s00440-007-0126-6},
    language = {en},
    number = {1},
    urldate = {2023-04-12},
    journal = {Probability Theory and Related Fields},
    author = {Baraud, Yannick and Birgé, Lucien},
    month = jan,
    year = {2009},
    pages = {239--284},
}

@inproceedings{yu_online_2023,
    title = {Online {Adversarial} {Stabilization} of {Unknown} {Linear} {Time}-{Varying} {Systems}},
    url = {https://ieeexplore.ieee.org/document/10383849},
    doi = {10.1109/CDC49753.2023.10383849},
    urldate = {2025-03-14},
    booktitle = {2023 62nd {IEEE} {Conference} on {Decision} and {Control} ({CDC})},
    author = {Yu, Jing and Gupta, Varun and Wierman, Adam},
    month = dec,
    year = {2023},
    note = {ISSN: 2576-2370},
    pages = {8320--8327},
}

@article{sart_density_2023,
    title = {Density estimation under local differential privacy and {Hellinger} loss},
    volume = {29},
    issn = {1350-7265},
    url = {https://projecteuclid.org/journals/bernoulli/volume-29/issue-3/Density-estimation-under-local-differential-privacy-and-Hellinger-loss/10.3150/22-BEJ1543.full},
    doi = {10.3150/22-BEJ1543},
    number = {3},
    urldate = {2025-04-12},
    journal = {Bernoulli},
    author = {Sart, Mathieu},
    month = aug,
    year = {2023},
    note = {Publisher: Bernoulli Society for Mathematical Statistics and Probability},
    pages = {2318--2341},
}

@misc{banerjee_goggins_2025,
    title = {Goggin's corrected {Kalman} {Filter}: {Guarantees} and {Filtering} {Regimes}},
    shorttitle = {Goggin's corrected {Kalman} {Filter}},
    url = {http://arxiv.org/abs/2502.14053},
    doi = {10.48550/arXiv.2502.14053},
    urldate = {2025-05-15},
    publisher = {arXiv},
    author = {Banerjee, Imon and Gurvich, Itai},
    month = feb,
    year = {2025},
    note = {arXiv:2502.14053 [cs]},
}

@article{datta_consistency_2023,
	title = {On the {Consistency} of {Maximum} {Likelihood} {Estimation} of {Probabilistic} {Principal} {Component} {Analysis}},
	volume = {36},
	url = {https://proceedings.neurips.cc/paper_files/paper/2023/hash/5b0c0b2c2efdd736a53688ebfdc3bcdb-Abstract-Conference.html},
	language = {en},
	urldate = {2025-05-20},
	journal = {Advances in Neural Information Processing Systems},
	author = {Datta, Arghya and Chakrabarty, Sayak},
	month = dec,
	year = {2023},
	pages = {28648--28662},
}

@misc{banerjee_clt_2025,
	title = {{CLT} and {Edgeworth} {Expansion} for m-out-of-n {Bootstrap} {Estimators} of {The} {Studentized} {Median}},
	url = {http://arxiv.org/abs/2505.11725},
	doi = {10.48550/arXiv.2505.11725},
	urldate = {2025-05-20},
	publisher = {arXiv},
	author = {Banerjee, Imon and Chakrabarty, Sayak},
	month = may,
	year = {2025},
	note = {arXiv:2505.11725 [cs]},
}

@book{dolgopyat_local_2023,
	address = {Cham},
	series = {Lecture {Notes} in {Mathematics}},
	title = {Local {Limit} {Theorems} for {Inhomogeneous} {Markov} {Chains}},
	volume = {2331},
	copyright = {https://www.springernature.com/gp/researchers/text-and-data-mining},
	isbn = {978-3-031-32600-4 978-3-031-32601-1},
	url = {https://link.springer.com/10.1007/978-3-031-32601-1},
	language = {en},
	urldate = {2025-03-03},
	publisher = {Springer International Publishing},
	author = {Dolgopyat, Dmitry and Sarig, Omri M.},
	year = {2023},
	doi = {10.1007/978-3-031-32601-1},
}

@book{bergh_interpolation_1976,
	address = {Berlin, Heidelberg},
	series = {Grundlehren der mathematischen {Wissenschaften}},
	title = {Interpolation {Spaces}: {An} {Introduction}},
	volume = {223},
	isbn = {978-3-642-66453-3 978-3-642-66451-9},
	shorttitle = {Interpolation {Spaces}},
	url = {http://link.springer.com/10.1007/978-3-642-66451-9},
	language = {en},
	urldate = {2023-08-07},
	publisher = {Springer},
	author = {Bergh, Jöran and Löfström, Jörgen},
	editor = {Chern, S. S. and Doob, J. L. and Douglas, J. and Grothendieck, A. and Heinz, E. and Hirzebruch, F. and Hopf, E. and Mac Lane, S. and Magnus, W. and Postnikov, M. M. and Schmidt, F. K. and Schmidt, W. and Scott, D. S. and Stein, K. and Tits, J. and Van Der Waerden, B. L. and Eckmann, B. and Moser, J. K.},
	year = {1976},
	doi = {10.1007/978-3-642-66451-9},
}

@article{akakpo_inhomogeneous_2011,
	title = {Inhomogeneous and anisotropic conditional density estimation from dependent data},
	volume = {5},
	issn = {1935-7524, 1935-7524},
	url = {https://projecteuclid.org/journals/electronic-journal-of-statistics/volume-5/issue-none/Inhomogeneous-and-anisotropic-conditional-density-estimation-from-dependent-data/10.1214/11-EJS653.full},
	doi = {10.1214/11-EJS653},
	number = {none},
	urldate = {2023-04-05},
	journal = {Electronic Journal of Statistics},
	author = {Akakpo, Nathalie and Lacour, Claire},
	month = jan,
	year = {2011},
	note = {Publisher: Institute of Mathematical Statistics and Bernoulli Society},
	pages = {1618--1653},
}

@article{baraud_estimator_2011,
	title = {Estimator selection with respect to {Hellinger}-type risks},
	volume = {151},
	issn = {1432-2064},
	url = {https://doi.org/10.1007/s00440-010-0302-y},
	doi = {10.1007/s00440-010-0302-y},
	language = {en},
	number = {1},
	urldate = {2023-05-17},
	journal = {Probability Theory and Related Fields},
	author = {Baraud, Yannick},
	month = oct,
	year = {2011},
	pages = {353--401},
}

@article{devore_degree_1990,
	title = {Degree of {Adaptive} {Approximation}},
	volume = {55},
	issn = {0025-5718},
	url = {https://www.jstor.org/stable/2008437},
	doi = {10.2307/2008437},
	number = {192},
	urldate = {2023-07-15},
	journal = {Mathematics of Computation},
	author = {DeVore, Ronald A. and Yu, Xiang Ming},
	year = {1990},
	note = {Publisher: American Mathematical Society},
	pages = {625--635},
}

@article{goggin_convergence_1992,
	title = {Convergence of filters with applications to the {Kalman}-{Bucy} case},
	volume = {38},
	issn = {1557-9654},
	url = {https://ieeexplore.ieee.org/abstract/document/135648},
	doi = {10.1109/18.135648},
	number = {3},
	urldate = {2024-06-26},
	journal = {IEEE Transactions on Information Theory},
	author = {Goggin, E.M.},
	month = may,
	year = {1992},
	note = {Conference Name: IEEE Transactions on Information Theory},
	pages = {1091--1100},
}

@book{massart_concentration_2007,
	address = {Berlin, Heidelberg},
	series = {Lecture {Notes} in {Mathematics}},
	title = {Concentration {Inequalities} and {Model} {Selection}},
	volume = {1896},
	isbn = {978-3-540-48497-4},
	url = {http://link.springer.com/10.1007/978-3-540-48503-2},
	language = {en},
	urldate = {2023-11-07},
	publisher = {Springer},
	author = {Massart, Pascal},
	editor = {Picard, Jean},
	year = {2007},
	doi = {10.1007/978-3-540-48503-2},
}

@book{rio_asymptotic_2017,
	address = {Berlin, Heidelberg},
	series = {Probability {Theory} and {Stochastic} {Modelling}},
	title = {Asymptotic {Theory} of {Weakly} {Dependent} {Random} {Processes}},
	volume = {80},
	copyright = {http://www.springer.com/tdm},
	isbn = {978-3-662-54322-1 978-3-662-54323-8},
	url = {http://link.springer.com/10.1007/978-3-662-54323-8},
	language = {en},
	urldate = {2025-03-13},
	publisher = {Springer},
	author = {Rio, Emmanuel},
	year = {2017},
	doi = {10.1007/978-3-662-54323-8},
}

@article{lacour_adaptive_2007,
	title = {Adaptive estimation of the transition density of a {Markov} chain},
	volume = {43},
	issn = {02460203},
	url = {http://arxiv.org/abs/math/0611680},
	doi = {10.1016/j.anihpb.2006.09.003},
	number = {5},
	urldate = {2022-12-12},
	journal = {Annales de l'Institut Henri Poincare (B) Probability and Statistics},
	author = {Lacour, Claire},
	month = sep,
	year = {2007},
	note = {arXiv:math/0611680},
	pages = {571--597},
}

@article{baraud_new_2017,
	title = {A new method for estimation and model selection:\$\${\textbackslash}rho \$\$-estimation},
	volume = {207},
	issn = {1432-1297},
	shorttitle = {A new method for estimation and model selection},
	url = {https://doi.org/10.1007/s00222-016-0673-5},
	doi = {10.1007/s00222-016-0673-5},
	language = {en},
	number = {2},
	urldate = {2024-10-25},
	journal = {Inventiones mathematicae},
	author = {Baraud, Y. and Birgé, L. and Sart, M.},
	month = feb,
	year = {2017},
	pages = {425--517},
}

@inproceedings{icarte_using_2018,
	title = {Using {Reward} {Machines} for {High}-{Level} {Task} {Specification} and {Decomposition} in {Reinforcement} {Learning}},
	url = {https://proceedings.mlr.press/v80/icarte18a.html},
	language = {en},
	urldate = {2023-09-05},
	booktitle = {Proceedings of the 35th {International} {Conference} on {Machine} {Learning}},
	publisher = {PMLR},
	author = {Icarte, Rodrigo Toro and Klassen, Toryn and Valenzano, Richard and McIlraith, Sheila},
	month = jul,
	year = {2018},
	pages = {2107--2116},
}

@article{bhatt_occupation_1996,
	title = {Occupation measures for controlled {Markov} processes: characterization and optimality},
	volume = {24},
	issn = {0091-1798, 2168-894X},
	shorttitle = {Occupation measures for controlled {Markov} processes},
	url = {https://projecteuclid.org/journals/annals-of-probability/volume-24/issue-3/Occupation-measures-for-controlled-Markov-processes-characterization-and-optimality/10.1214/aop/1065725192.full},
	doi = {10.1214/aop/1065725192},
	number = {3},
	urldate = {2024-01-17},
	journal = {The Annals of Probability},
	author = {Bhatt, Abhay G. and Borkar, Vivek S.},
	month = jul,
	year = {1996},
	pages = {1531--1562},
}

@book{ash_probability_2000,
	title = {Probability and {Measure} {Theory}},
	isbn = {978-0-12-065202-0},
	language = {en},
	publisher = {Academic Press},
	author = {Ash, Robert B. and Doleans-Dade, Catherine A.},
	year = {2000},
}

@article{serfozo_semi-stationary_1972,
	title = {Semi-stationary processes},
	volume = {23},
	issn = {1432-2064},
	url = {https://doi.org/10.1007/BF00532855},
	doi = {10.1007/BF00532855},
	language = {en},
	number = {2},
	urldate = {2024-08-12},
	journal = {Zeitschrift für Wahrscheinlichkeitstheorie und Verwandte Gebiete},
	author = {Serfozo, Richard F.},
	month = jun,
	year = {1972},
	pages = {125--132},
}

@book{krishnamurthy_partially_2016,
	address = {Cambridge},
	title = {Partially {Observed} {Markov} {Decision} {Processes}: {From} {Filtering} to {Controlled} {Sensing}},
	isbn = {978-1-107-13460-7},
	shorttitle = {Partially {Observed} {Markov} {Decision} {Processes}},
	url = {https://www.cambridge.org/core/books/partially-observed-markov-decision-processes/505ADAE28B3F22D1594F837DEAFF1E0C},
	urldate = {2024-10-19},
	publisher = {Cambridge University Press},
	author = {Krishnamurthy, Vikram},
	year = {2016},
	doi = {10.1017/CBO9781316471104},
}

@article{barron_risk_1999,
	title = {Risk bounds for model selection via penalization},
	volume = {113},
	issn = {1432-2064},
	url = {https://doi.org/10.1007/s004400050210},
	doi = {10.1007/s004400050210},
	language = {en},
	number = {3},
	urldate = {2023-07-03},
	journal = {Probability Theory and Related Fields},
	author = {Barron, Andrew and Birgé, Lucien and Massart, Pascal},
	month = feb,
	year = {1999},
	pages = {301--413},
}

@article{baraud_rho-estimators_2018,
	title = {Rho-estimators revisited: {General} theory and applications},
	volume = {46},
	issn = {0090-5364, 2168-8966},
	shorttitle = {Rho-estimators revisited},
	url = {https://projecteuclid.org/journals/annals-of-statistics/volume-46/issue-6B/Rho-estimators-revisited-General-theory-and-applications/10.1214/17-AOS1675.full},
	doi = {10.1214/17-AOS1675},
	number = {6B},
	urldate = {2024-10-25},
	journal = {The Annals of Statistics},
	author = {Baraud, Yannick and Birgé, Lucien},
	month = dec,
	year = {2018},
	pages = {3767--3804},
}

@article{bradley_examples_1993,
	title = {Some {Examples} of {Mixing} {Random} {Fields}},
	volume = {23},
	issn = {0035-7596},
	url = {https://projecteuclid.org/journals/rocky-mountain-journal-of-mathematics/volume-23/issue-2/Some-Examples-of-Mixing-Random-Fields/10.1216/rmjm/1181072573.full},
	doi = {10.1216/rmjm/1181072573},
	number = {2},
	urldate = {2024-11-06},
	journal = {Rocky Mountain Journal of Mathematics},
	author = {Bradley, Richard C.},
	month = jun,
	year = {1993},
	pages = {495--519},
}

@misc{deb_trade-off_2024,
	title = {Trade-off {Between} {Dependence} and {Complexity} for {Nonparametric} {Learning} – an {Empirical} {Process} {Approach}},
	url = {http://arxiv.org/abs/2401.08978},
	urldate = {2024-11-06},
	publisher = {arXiv},
	author = {Deb, Nabarun and Mukherjee, Debarghya},
	month = jan,
	year = {2024},
	doi = {10.48550/arXiv.2401.08978},
}

@article{banerjee_off-line_2025,
	title = {Off-line {Estimation} of {Controlled} {Markov} {Chains}: {Minimaxity} and {Sample} {Complexity}},
	issn = {0030-364X},
	shorttitle = {Off-line {Estimation} of {Controlled} {Markov} {Chains}},
	url = {https://pubsonline.informs.org/doi/abs/10.1287/opre.2023.0046},
	doi = {10.1287/opre.2023.0046},
	urldate = {2025-02-24},
	journal = {Operations Research},
	author = {Banerjee, Imon and Honnappa, Harsha and Rao, Vinayak},
	month = feb,
	year = {2025},
}

@book{ross_stochastic_1983,
	title = {Stochastic {Processes}},
	isbn = {978-0-471-09942-0},
	language = {en},
	publisher = {Wiley},
	author = {Ross, Sheldon M.},
	year = {1983},
}

@article{glynn_wide-sense_2011,
	title = {Wide-sense regeneration for {Harris} recurrent {Markov} processes: an open problem},
	volume = {68},
	issn = {1572-9443},
	shorttitle = {Wide-sense regeneration for {Harris} recurrent {Markov} processes},
	url = {https://doi.org/10.1007/s11134-011-9238-x},
	doi = {10.1007/s11134-011-9238-x},
	language = {en},
	number = {3},
	urldate = {2025-03-01},
	journal = {Queueing Systems},
	author = {Glynn, Peter W.},
	month = aug,
	year = {2011},
	pages = {305--311},
}

@article{merlevede_local_2021,
	title = {On the local limit theorems for psi-mixing {Markov} chains},
	volume = {18},
	issn = {1980-0436},
	url = {http://alea.impa.br/articles/v18/18-45.pdf},
	doi = {10.30757/ALEA.v18-45},
	language = {en},
	number = {1},
	urldate = {2025-03-03},
	journal = {Latin American Journal of Probability and Mathematical Statistics},
	author = {Merlevède, Florence and Peligrad, Magda and Peligrad, Costel},
	year = {2021},
	pages = {1221},
}

@misc{bhattacharya_explicit_2023,
	title = {Explicit {Constraints} on the {Geometric} {Rate} of {Convergence} of {Random} {Walk} {Metropolis}-{Hastings}},
	url = {http://arxiv.org/abs/2307.11644},
	urldate = {2025-02-28},
	publisher = {arXiv},
	author = {Bhattacharya, Riddhiman and Jones, Galin L.},
	month = jul,
	year = {2023},
	doi = {10.48550/arXiv.2307.11644},
}

@article{merlevede_local_2022,
	title = {On the local limit theorems for lower psi-mixing {Markov} chains},
	volume = {19},
	issn = {1980-0436},
	url = {http://alea.impa.br/articles/v19/19-45.pdf},
	doi = {10.30757/ALEA.v19-45},
	language = {en},
	number = {1},
	urldate = {2025-03-03},
	journal = {Latin American Journal of Probability and Mathematical Statistics},
	author = {Merlevède, Florence and Peligrad, Magda and Peligrad, Costel},
	year = {2022},
	pages = {1103},
}

@book{ljung_system_1999,
	address = {NJ, USA},
	title = {System identification (2nd ed.): theory for the user},
	isbn = {978-0-13-656695-3},
	shorttitle = {System identification (2nd ed.)},
	publisher = {Prentice Hall PTR},
	author = {Ljung, Lennart},
	year = {1999},
}

@misc{wang_foundation_2024,
	title = {On the {Foundation} of {Distributionally} {Robust} {Reinforcement} {Learning}},
	url = {http://arxiv.org/abs/2311.09018},
	urldate = {2025-03-14},
	publisher = {arXiv},
	author = {Wang, Shengbo and Si, Nian and Blanchet, Jose and Zhou, Zhengyuan},
	month = jan,
	year = {2024},
	doi = {10.48550/arXiv.2311.09018},
}

@incollection{schaefer_modeling_2004,
	address = {Boston, MA},
	title = {Modeling {Medical} {Treatment} {Using} {Markov} {Decision} {Processes}},
	isbn = {978-1-4020-8066-1},
	url = {https://doi.org/10.1007/1-4020-8066-2_23},
	language = {en},
	urldate = {2025-04-09},
	booktitle = {Operations {Research} and {Health} {Care}: {A} {Handbook} of {Methods} and {Applications}},
	publisher = {Springer US},
	author = {Schaefer, Andrew J. and Bailey, Matthew D. and Shechter, Steven M. and Roberts, Mark S.},
	editor = {Brandeau, Margaret L. and Sainfort, François and Pierskalla, William P.},
	year = {2004},
	doi = {10.1007/1-4020-8066-2_23},
	pages = {593--612},
}
\fussy

\appendix

\section{Sketch of Proof of Proposition \ref{prop:main-concentration}}\label{sec:sketch-prfmainconc}

We first prove an auxiliary result comparing two different piecewise constant estimators on two different partitions $m_1$ and $m_2$. The proof of this Proposition can be found in Section \ref{sec:prf-m1m2conc}.
\begin{proposition}~\label{prop:conc-m1m2}
    Let $m_1$ and $m_2$ be two different partitions belonging to $\Mcal_\lcal$ for some $\lcal$, and $f_1$ and $f_2$ be two piecewise constant functions on the two partitions respectively. Let $ \kappa = (2+11\sqrt{2})/(2\sqrt{2}-2)$. Then, it holds with probability at most $\exp\lp -n(pen(m_1)-pen(m_2))/\kappa-n\zeta \rp$ that
    \begin{small}
    \begin{align*}
        \frac{3}{4}\lp 1-\frac{1}{\sqrt{2}}\rp \Hcal^2(\density ,f_2)+ \Test(f_1,f_2) \leq  \frac{5}{4}\lp1+\frac{1}{\sqrt{2}}\rp\Hcal^2(\density ,f_1)+pen(m_1)+pen(m_2)+\zeta 
    \end{align*}
    for any $\zeta>0$.
    \end{small}
\end{proposition}
Let $m$ be any partition. Consider the following two cases.
\paragraph{CASE I  $\bigg( T(\hat \density_m,\hat \density_{\hat m})-pen(\hat m) +pen(m) \bigg)\geq 0$ :} If $\bigg( T(\hat \density_m,\hat \density_{\hat m})-pen(\hat m) +pen(m) \bigg)\geq 0$, then the conclusion follows readily from Proposition \ref{prop:conc-m1m2} and some algebra.

\paragraph{CASE II $\bigg( T(\hat \density_m,\hat \density_{\hat m})-pen(\hat m) +pen(m) \bigg)\leq 0$ : } We first write the following proposition about dyadic partitions. Its proof follows by using the tree-like structure of dyadic cuts and can be found in Section \ref{sec:prf-partition}.

\begin{proposition}~\label{prop:partition}
Let $\Mcal_\lcal$ be the dyadic partitions of depth $\lcal$ as in Definition \ref{def:dyadic-cuts}. Then,    \begin{enumerate}
    \item~\label{assume:part1} $\Mcal_\lcal\subset \Mcal_{\lcal+1}$, for any $\lcal$. Furthermore, $\sum_{m\in\Mcal_\infty}e^{-|m|}\leq \sum_{\lcal\geq 0} {2}^{\lcal(2d_1+d_2)} e^{-{2}^{\lcal(2d_1+d_2)}}\leq 15$, and for any $m\in\Mcal_\lcal$, $|m|\leq {2}^{\lcal(2d_1+d_2)}$ where $|m|$ is the cardinality of the partition $m$.
    \item~\label{assume:part2} If $m\in \Mcal_\lcal\backslash \Mcal_{\lcal'}$, where $\lcal'<\lcal$, then $|m| > \lcal'$.
    \item~\label{assume:part3} If $K\in m\in \Mcal_\lcal$, then $\exists\ \{ K_1,K_2,\dots,K_\lcal \}\in \bigcup_{m\in \Mcal_\lcal}m$ such that $K\subset K_i, i\in \{1,\dots,\lcal\}$
    \item~\label{assume:part4}  Define $ m\vee m'$ as the set of non-empty intersections of $m'$ with the elements of $m$. To be precise, 
    \[
    m\vee m' = \bigcup_{K'\in m'}\lc m\vee K' \rc\numberthis\label{eq:vee-def}
    \]
    where $m\vee K'$ is as defined in \cref{eq:vee-def2}. Then, $|m\vee m'|\leq 2(|m|+|m'|)$.
    \end{enumerate}
\end{proposition}

The rest of the second case can now be divided into the following 3 steps. 
\paragraph{Step I:} 
Let $m\in\Mcal_\lcal$ be a partition and $K\in m$. Recall from Proposition \ref{prop:partition} item \ref{assume:part3} that there exists $K_1,\dots,K_\lcal$ such that $K\subset K_i$. Let $K_i = K_i\pow 1 K_i\pow 2 K_i\pow 3 $. We define the set $\density_m$ to be:
\[
\density_m := \lc \sum_{K\in m} f_K \indicator_K : f_K\in \bigcup_{i=0}^l\lc \frac{a}{b \mu_{\Ibb} \lp K_i\pow 2\rp\muc \lp K_i\pow 3\rp }:a\in\lc0,\dots,n\rc, b\in\lc 1,\dots,n \rc \rc \rc\numberthis\label{def:SM};
\]
observe that $\lc a\lp b \mu_{\Ibb} ( K_i\pow 2)\muc (K_i\pow 3)\rp^{-1}:i,a\in\lc0,\dots,n\rc, b\in\lc 1,\dots,n \rc \rc$ is the set of all the piecewise constant functions that can be made with $n$ sample points.
We then prove the following result which is formally stated in Section \ref{sec:prf-mainconc} as Lemma \ref{lemma:g-ublb}

\begin{align*}
    & \sup_{m'\in \Mcal_\lcal} \lb \frac{3}{4}\lp 1-\frac{1}{\sqrt{2}}\rp\Hcal^2(\hat \density_m,\hat \density_{m'})+T(\hat \density_m,\hat \density_{m'}) -pen(m')\rb+pen(m) \leq \gamma(m)  \\
    &  \gamma(m) \leq \sup_{\substack{f\in \density_{m'}\\m'\in \Mcal_\lcal}} \lb \frac{3}{4}\lp 1-\frac{1}{\sqrt{2}}\rp\Hcal^2(\hat \density_m,f)+T(\hat \density_m,f) -pen(m')\rb+2\,pen(m)
\end{align*}

\paragraph{Step II:} Using this lemma, we upper bound the probability of 
\[ \Constant\Hcal^2(\density , \hat \density) \geq \inf_{m\in \Mcal_\lcal} (\Hcal^2(\density ,   \hat \density_m)+pen(m))\] 
by the probability of 
\[  \Constant\Hcal^2(\density , \hat \density) \geq  \sup_{\substack{f\in \density_{m'}\\m'\in \Mcal_\lcal}} \lb \frac{3}{4}\lp 1-\frac{1}{\sqrt{2}}\rp\Hcal^2(\hat \density_{\hat m},f)+T(\hat \density_{\hat m},f) -pen(m')\rb+2\,pen(\hat m)\]
where $\density_{m'}$ is as defined in \cref{def:SM}.
\paragraph{Step III:} We produce an upper bound to the preceding probability using Proposition \ref{prop:conc-m1m2} and appropriate union bounds.

\section{Proofs}\label{sec:prfs}
\subsection{Proof of Proposition \ref{prop:Loss-bound}}\label{sec:prf-lossbd}
\begin{proof}
Construct the piece-wise constant estimator of $\density_m$ given by 
\[
    \bar \density_m := \sum_{k\in m} \frac{\sum_{i=0}^{n-1} \expec\lb \indicator_k(X_i,a_i,X_{i+1}) |X_i, a_i\rb }{ \sum_{i=0}^{n-1} \int_\chi \indicator_k(X_i,a_i,y) d\mu_\chi(y) }\indicator_k.
\]
Observe that by using the triangle inequality, we have
\begin{align*}
    \expec\lb \Hcal^2(\density ,\hat \density_m) \rb & \leq \expec\lb \Hcal^2(\density ,\bar \density_m) \rb+\expec\lb \Hcal^2(\bar \density_m,\hat \density_m) \rb.\numberthis
\end{align*}
We bound each term separately. For the purpose of bounding the first term, we require the following lemma. Let $f$ be an integrable function defined on a domain $\chi_\lambda$ with the range being $\real$, and let $\lambda$ be a measure on $\chi_\lambda$. We can then adapt Lemma 2 from \cite{baraud_estimating_2009} as:
\begin{lemma}~\label{lemma:bar-bir}
For any $m$, a finite partition of a subset $\Ical$ of $\chi_f$ define 
\[
    \bar f := \sum_{k\in m}\lp \int_k \frac{f d\lambda}{\lambda(k)} \rp\indicator_k. 
\]
Then, $\expec\lb \Hcal_\lambda(f,\bar f) \rb\leq\expec\lb 2\Hcal_\lambda(f,V_m) \rb $, where $\Hcal_\lambda$ is the Hellinger distance defined according to measure $\lambda$.
\end{lemma}
For the purposes of the lemma, we make explicit the dependence of the Hellinger distance $\Hcal$ is matched to the integrating measure $\lambda$ and the projection $\bar f$. For the rest of the paper, this relationship is satisfied by construction, and we suppress this dependence.

To use Lemma \ref{lemma:bar-bir}, we only need to verify that given  $\lambda = \lambda_n$ (as defined in Remark \ref{remark:lambda_n}), $f=\density$, and $\Ical = A$, we have 
\begin{align*}
    \bar f = \frac{\frac{1}{2n}\sum_{i=0}^{n-1} \expec\lb \indicator_k(X_i,a_i,X_{i+1}) |X_i, a_i\rb }{ \frac{1}{2n}\sum_{i=0}^{n-1} \int_\chi \indicator_k(X_i,a_i,y) d\mu_\chi(y) }  = \bar \density_m
\end{align*}
In other words, it is enough to show that for our given choice of $\lambda,f,\Ical$,
\[
    \int_k \frac{f\ d\lambda}{\lambda(k)} = \frac{\sum_{i=0}^{n-1} \expec\lb \indicator_k(X_i,a_i,X_{i+1}) |X_i, a_i\rb }{ \sum_{i=0}^{n-1} \int_\chi \indicator_k(X_i,a_i,y) d\mu_\chi(y) } = \frac{\frac{1}{2n}\sum_{i=0}^{n-1} \expec\lb \indicator_k(X_i,a_i,X_{i+1}) |X_i, a_i\rb }{ \frac{1}{2n}\sum_{i=0}^{n-1} \int_\chi \indicator_k(X_i,a_i,y) d\mu_\chi(y) }.
\]
We only verify the denominators are equal. The numerators follow similarly.

For any  $k\subset\chi\times\Ibb\times\chi$ such that $k\in m$,  
\begin{align*}
    \lambda(k) & =\int_{(z_1,z_2,z_3)\in k}  \lambda_n(dz_1,dz_2,dz_3) \\
    & =\int_{(z_1,z_2,z_3)\in k}  \frac{1}{2n}\sum_{i=0}^{n-1}\delta_{X_i,a_i}(dz_1,dz_2)\muc(dz_3) \\
    & = \int_{\chi\times\Ibb\times\chi}  \frac{1}{2n}\sum_{i=0}^{n-1} \indicator_k(z_1,z_2,z_3)\delta_{X_i,a_i}(dz_1,dz_2)\muc(dz_3)\\
    & = \int_{\chi}  \frac{1}{2n}\sum_{i=0}^{n-1}\int_{\chi\times\Ibb} \indicator_k(z_1,z_2,z_3)\delta_{X_i,a_i}(dz_1,dz_2)\muc(dz_3)\\
    & = \int_{\chi}  \frac{1}{2n}\sum_{i=0}^{n-1} \indicator_k\lp X_i,a_i,y\rp \mu_\chi(dy).
\end{align*}
This completes our verification. Now using Lemma \ref{lemma:bar-bir}, we get
\[
\expec\lb \Hcal^2(\density ,\bar \density ) \rb\leq \expec\lb 2\Hcal^2(\density ,V_m) \rb.
\]
Next, we produce an upper bound for the second term. Observe that we can expand the square in $\Hcal^2(\bar \density_m, \hat \density_m)$ to get
\begin{align*}
    \Hcal^2(\bar \density_m,\hat \density_m) & = \frac{1}{2n} \sum_{i=0}^{n-1} \sum_{k\in m} \int_\chi\frac{\sum_{i=0}^{n-1} \expec\lb \indicator_k(X_i,a_i,X_{i+1}) |X_i, a_i\rb }{ \sum_{i=0}^{n-1} \int_\chi \indicator_k(X_i,a_i,y) d\mu_\chi(y) }\indicator_k(X_i,a_i,y) d\muc(y)\\
    & + \frac{1}{2n} \sum_{i=0}^{n-1} \sum_{k\in m} \int_\chi\sum_{k\in m} \frac{\sum_{i=0}^{n-1}{\indicator_k(X_i,a_i,X_{i+1})}}{\sum_{i=0}^{n-1}\int_{\chi}{\indicator_k(X_i,a_i,t)}d\mu_\chi(t)}\indicator_k(X_i,a_i,t) d\muc(t) - 2\times\text{C}\\
    & = \frac{1}{2n} \sum_{k\in m} \frac{\sum_{i=0}^{n-1} \expec\lb \indicator_k(X_i,a_i,X_{i+1}) |X_i, a_i\rb }{ \sum_{i=0}^{n-1} \int_\chi \indicator_k(X_i,a_i,t) d\mu_\chi(t) } \sum_{i=0}^{n-1}\int_\chi\indicator_k(X_i,a_i,x) d\muc(x)\\
    & + \frac{1}{2n}   \sum_{k\in m} \frac{\sum_{i=0}^{n-1}{\indicator_k(X_i,a_i,X_{i+1})}}{\sum_{i=0}^{n-1}\int_{\chi}{\indicator_k(X_i,a_i,t)}d\mu_\chi(t)}\sum_{i=0}^{n-1}\int_{\chi}\indicator_k(X_i,a_i,x) d\muc(x) - 2\times\text{C}.
\end{align*}
Where `C' is the cross term made explicit in \cref{eq:lssbnd-eq6}. 
Observe that
the denominators cancel with the integral in the numerators.
So we can write,
\begin{align*}
    \Hcal^2(\bar \density_m,\hat \density_m) & = \frac{1}{2n}   \sum_{k\in m} \expec\lb \indicator_k(X_i,a_i,X_{i+1}) |X_i, a_i\rb + \frac{1}{2n} \sum_{k\in m} \indicator_k(X_i,a_i,X_{i+1}) - 2\times C.\numberthis\label{eq:lssbnd-eq7}
\end{align*}
The cross term `C' is
\begin{align*}
     \frac{1}{2n} \sum_{i=0}^{n-1} \int_\chi \sqrt{ \lp \sum_{k\in m} b_k(y) \rp\lp \sum_{k\in m} b'_k(y) \rp} d\muc(y)\numberthis\label{eq:lssbnd-eq6}
\end{align*}
where
\begin{align*}
    b_k (\cdot)& = \frac{\sum_{i=0}^{n-1} \expec\lb \indicator_k(X_i,a_i,X_{i+1}) |X_i, a_i\rb }{ \sum_{i=0}^{n-1} \int_\chi \indicator_k(X_i,a_i,t) d\mu_\chi(t) }\indicator_k(X_i,a_i,\cdot) \quad \textit{ and }\\
    b'_k(\cdot) & = \frac{\sum_{i=0}^{n-1}{\indicator_k(X_i,a_i,X_{i+1})}}{\sum_{i=0}^{n-1}\int_{\chi}{\indicator_k(X_i,a_i,t)}d\mu_\chi(t)}\indicator_k(X_i,a_i,\cdot)
\end{align*}
By using Cauchy-Schwarz inequality, we get $\sqrt{ \lp \sum b_k \rp\lp \sum b'_k \rp}\geq \sum  \sqrt{b_k b'_k}$.
This in turn implies that
\begin{align*}
    \sum_{i=0}^{n-1} \int_\chi \sqrt{ \lp \sum_{k\in m} b_k (y)\rp\lp \sum_{k\in m} b'_k(y) \rp} d\muc(y) \geq \sum_{k\in m}\int_\chi\sum_{i=0}^{n-1} \sqrt{b_k b'_k}d\muc\numberthis \label{eq:akbk-lb}
\end{align*}
It follows by substituting $b_k$ and $b_k'$ that,
\begin{small}
\begin{align*}
     &\int_\chi\sum_{i=0}^{n-1}\sqrt{b_k b'_k}d\muc\\
     &\quad = \int_\chi \sum_{i=0}^{n-1} \frac{\sqrt{\lp\sum_{i=0}^{n-1} \expec\lb \indicator_k(X_i,a_i,X_{i+1}) |X_i, a_i\rb \rp\lp\sum_{i=0}^{n-1}{\indicator_k(X_i,a_i,X_{i+1})} \rp }}{ \sum_{i=0}^{n-1} \int_\chi \indicator_k(X_i,a_i,t) d\mu_\chi(t) }\indicator_k(X_i,a_i,y)d\mu_\chi(y).
\end{align*}
\end{small}
The integral in the denominator cancels with the one in the numerator, which consequently implies that
\begin{align*}
    \frac{1}{2n} \sum_{i=0}^{n-1} \int_\chi \sum_{k\in m} \sqrt{  b_k b'_k } d\muc & = \frac{1}{2n}\sum_{k\in m}  \sqrt{\lp\sum_{i=0}^{n-1} \expec\lb \indicator_k(X_i,a_i,X_{i+1}) |X_i, a_i\rb \rp\lp\sum_{i=0}^{n-1}{\indicator_k(X_i,a_i,X_{i+1})} \rp }\\
    & = \frac{1}{2n}\sum_{k\in m}  \sqrt{c_k c'_k},
\end{align*}
where 
\begin{align*}
    c_k =\lp\sum_{i=0}^{n-1} \expec\lb \indicator_k(X_i,a_i,X_{i+1}) |X_i, a_i\rb\rp \text{ and } 
    \ c'_k = \lp\sum_{i=0}^{n-1}{\indicator_k(X_i,a_i,X_{i+1})} \rp.
\end{align*}
Substituting this into \cref{eq:akbk-lb}, we get that the lower bound of the right hand side of \cref{eq:akbk-lb} is $\sum_{k\in m}  \sqrt{c_k c'_k}/2n$.
Substituting this lower bound into \cref{eq:lssbnd-eq7} we can now observe that,
\begin{align*}
    \Hcal^2(\bar \density_m,\hat \density_m) & \leq \frac{1}{2n} \sum_{k\in m} \lp c_k+c'_k - 2\sqrt{c_k c'_k}  \rp\\
    & = \frac{1}{2n} \sum_{k\in m} \lp \sqrt{c_k}-\sqrt{c'_k}\rp^2\\
    & = \frac{1}{2n} \sum_{k\in m} \lp \sqrt{\sum_{i=0}^{n-1} \expec\lb \indicator_k(X_i,a_i,X_{i+1}) |X_i, a_i\rb } - \sqrt{\sum_{i=0}^{n-1}{\indicator_k(X_i,a_i,X_{i+1})} }\rp^2.
\end{align*}
 Taking expectations on both sides now yield, 
\begin{align*}
    \expec\lb \Hcal^2(\bar \density_m,\hat \density_m) \rb \leq \frac{1}{2n}\sum_{k\in m} \expec\lp \sqrt{\sum_{i=0}^{n-1} \expec\lb \indicator_k(X_i,a_i,X_{i+1}) |X_i, a_i\rb} - \sqrt{\sum_{i=0}^{n-1}{\indicator_k(X_i,a_i,X_{i+1})} }\rp^2.\numberthis\label{eq:lssbnd-eq10}
\end{align*}
We first bound from above each term inside the summand. 
Define the finite stopping time 
\[
T_{st}:= \arg\min \lc j\leq n-1 :\lc\indicator_k(X_j,a_j,X_{j+1}) = 1\rc\bigcup\lc\expec\lb \indicator_k(X_j,a_j,X_{j+1}) |X_i, a_i\rb\geq n^{-1} \rc\rc\wedge n.\numberthis\label{eq:stopping_time}
\]
For any 3 positive numbers $c_1,c_2,c_3$, we have the following algebraic inequality 
\[
\lp\sqrt{c_1+c_2}-\sqrt{c_3}\rp^2\leq c_1+\lp\sqrt{c_1}-\sqrt{c_3}\rp^2
\]
By setting 
\begin{align*}
    c_1 & = \sum_{i=0}^{T_{st}-1} \expec\lb \indicator_k(X_i,a_i,X_{i+1}) |X_i, a_i\rb \\ 
    c_2 & = \sum_{i=T_{st}}^{n-1} \expec\lb \indicator_k(X_i,a_i,X_{i+1}) |X_i, a_i\rb \\ 
    c_3 & = \sum_{i=T_{st}}^{n-1}  \indicator_k(X_i,a_i,X_{i+1}), 
\end{align*}
we can write
\begin{align*}
   & \lp\sqrt{\sum_{i=0}^{n-1} \expec\lb \indicator_k(X_i,a_i,X_{i+1}) |X_i, a_i\rb} \rdot \ldot - \sqrt{\sum_{i=0}^{n-1}{\indicator_k(X_i,a_i,X_{i+1})} }\rp^2\\
   & \qquad \leq \sum_{i=0}^{T_{st}-1} \expec\lb \indicator_k(X_i,a_i,X_{i+1}) |X_i, a_i\rb \\
    & \qquad + \lp\sqrt{\sum_{i=T_{st}}^{n-1} \expec\lb \indicator_k(X_i,a_i,X_{i+1}) |X_i, a_i\rb} - \sqrt{\sum_{i=0}^{n-1}{\indicator_k(X_i,a_i,X_{i+1})} }\rp^2\numberthis\label{eq:lssbnd-eq8}. 
\end{align*}   
 It follows from the definition of $T_{st}$ that
\begin{align*}
    \sum_{i=0}^{T_{st}-1} \expec\lb \indicator_k(X_i,a_i,X_{i+1}) |X_i, a_i\rb    \leq \sum_{i=0}^{T_{st}-1} \frac{1}{n} = \frac{T_{st}}{n}\quad \leq 1
\end{align*}
and,
\[
 \sum_{i=0}^{T_{st}-1}{\indicator_k(X_i,a_i,X_{i+1})} = 0
\]
$\prob$-almost everywhere. So the first term of \cref{eq:lssbnd-eq8} can be upper bounded by $1$ and $\sum_{i=0}^{T_{st}-1}{\indicator_k(X_i,a_i,X_{i+1})}$ in the second term vanishes. Therefore,
\begin{align*}
   &\frac{1}{2n}\sum_{k\in m} \expec \lp\sqrt{\sum_{i=0}^{n-1} \expec\lb \indicator_k(X_i,a_i,X_{i+1}) |X_i, a_i\rb} \rdot \ldot - \sqrt{\sum_{i=0}^{n-1}{\indicator_k(X_i,a_i,X_{i+1})} }\rp^2\\
   & \qquad \leq \frac{1}{2n}\sum_{k\in m} \lp1 + \expec\lp\sqrt{\sum_{i=T_{st}}^{n-1} \expec\lb \indicator_k(X_i,a_i,X_{i+1}) |X_i, a_i\rb} - \sqrt{\sum_{i=T_{st}}^{n-1}{\indicator_k(X_i,a_i,X_{i+1})} }\rp^2\rp.\numberthis\label{eq:lssbnd-eq9}
\end{align*}   
The second term of the previous equation is now dealt in $2$ cases.
\paragraph{CASE I.}
\[
\expec\lb \indicator_k(X_{T_{st}},a_{T_{st}},X_{T_{st}}) |X_{T_{st}}, a_{T_{st}}\rb\geq \frac{1}{n} \numberthis\label{eq:lssbnd-eq1}
\]
Recall $(\sqrt a -\sqrt b)^2\leq (a-b)^2/b$ as the algebraic inequality obtained by rationalising $\sqrt a -\sqrt b$ for positive numbers $a,b$. We substitute $a=\sum{\indicator_k(X_i,a_i,X_{i+1})}$ and $b= \sum_{i=T}^{n-1} \expec\lb \indicator_k(X_i,a_i,X_{i+1}) |X_i, a_i\rb$ to get the following upper bound to the right hand side of \cref{eq:lssbnd-eq9}:
\begin{footnotesize}
\begin{align*}
     &\frac{1}{2n}\sum_{k\in m} \expec \lp\sqrt{\sum_{i=0}^{n-1} \expec\lb \indicator_k(X_i,a_i,X_{i+1}) |X_i, a_i\rb} \rdot \ldot - \sqrt{\sum_{i=0}^{n-1}{\indicator_k(X_i,a_i,X_{i+1})} }\rp^2\\
     &\qquad  \leq \frac{1}{2n}\sum_{k\in m}\lp1+\expec\lb\frac{\lp\sum_{i=T}^{n-1} \expec\lb \indicator_k(X_i,a_i,X_{i+1}) |X_i, a_i\rb - \sum_{i=T}^{n-1}{\indicator_k(X_i,a_i,X_{i+1})}\rp^2}{\sum_{i=T_{st}}^{n-1} \expec\lb \indicator_k(X_i,a_i,X_{i+1}) |X_i, a_i\rb}\rb\rp\\ 
    &\qquad  = \frac{1}{2n}\sum_{k\in m}\lp1+\sum_{j=0}^{n-1}\expec\lb\frac{\lp\sum_{i=T_{st}}^{n-1} \expec\lb \indicator_k(X_i,a_i,X_{i+1}) |X_i, a_i\rb - {\sum_{i=T_{st}}^{n}{\indicator_k(X_i,a_i,X_{i+1})}}\rp^2}{\sum_{i=T_{st}}^{n-1} \expec\lb \indicator_k(X_i,a_i,X_{i+1}) |X_i, a_i\rb}\indicator_{T_{st}=j}\rb\rp\numberthis\label{eq:lssbnd-eq2},
\end{align*}  
\end{footnotesize}
where the equality follows since $\sum_j \indicator_{T_{st}=j}=1$.
Observe that 
\[
\frac{\lp\sum_{i=T_{st}}^{n-1} \expec\lb \indicator_k(X_i,a_i,X_{i+1}) |X_i, a_i\rb - {\sum_{i=T_{st}}^{n-1}{\indicator_k(X_i,a_i,X_{i+1})}}\rp^2}{\sum_{i=T_{st}}^{n-1} \expec\lb \indicator_k(X_i,a_i,X_{i+1}) |X_i, a_i\rb}\indicator_{T_{st}=j}
\]
is of the form $(observed-expected)^2/expected$, which is the conditional variant of the well-known goodness of fit (G.O.F.) statistic.
The following lemma provides an upper bound to this G.O.F. statistic. 
\begin{lemma}~\label{lemma:cond-gof}
The G.O.F. statistic satisfies,
\begin{align*}
    \expec\lb \mathrm{G.O.F.}  \rb \leq \expec\lb\sum_{i=j}^{n-1}\frac{ \expec\lb \indicator_k(X_i,a_i,X_{i+1}) |X_i, a_i\rb}{\sum_{i=j}^{n-1} \expec\lb \indicator_k(X_i,a_i,X_{i+1}) |X_i, a_i\rb}\indicator_{T_{st}=j}\rb.
\end{align*}
\end{lemma}
\noindent Next, we write the following algebraic inequality for $n$ many bounded positive real numbers $z_i$.
\begin{lemma}~\label{lemma:prp-lsbnd-alg}
    For any integer $j\leq n$, $n$ many bounded positive real numbers $z_i$ 
    \begin{align*}
        \sum_{p=j}^{n-1} \frac{z_p}{\sum_{i=j}^{p} z_i} \leq 1+\log n - \log z_j.
    \end{align*}
\end{lemma}
The proofs of the previous two lemmas follow similarly to the proof of Claims B.1 and B.2 in \cite{sart_estimation_2014}. 
From an application of Lemmas \ref{lemma:cond-gof} and \ref{lemma:prp-lsbnd-alg} we get that for $j\leq n-2$
\begin{align*}
    \expec\lb \text{G.O.F.}  \rb  & \leq \expec\lb\sum_{i=j}^{n-1}\frac{ \expec\lb \indicator_k(X_i,a_i,X_{i+1}) |X_i, a_i\rb}{\sum_{i=j}^{n-1} \expec\lb \indicator_k(X_i,a_i,X_{i+1}) |X_i, a_i\rb}\indicator_{T_{st}=j}\rb\\ 
    & \leq \expec\lb \lp 1+\log n - \log \expec\lb \indicator_k(X_i,a_i,X_{i+1}) |X_i, a_i\rb\rp\indicator_{T_{st}=j}\rb.\numberthis\label{eq:lssbnd-eq3}
\end{align*}
But, from \cref{eq:lssbnd-eq1} we have $\expec\lb \indicator_k(X_i,a_i,X_{i+1}) |X_i, a_i\rb\geq n^{-1}$.
Thus, it follows that,
\begin{align*}
     \expec\lb \text{G.O.F.}\rb & \leq \expec\lb\lp 1+\log n - \log \frac{1}{n} \rp\indicator_{T_{st}=j}\rb\\ 
     & = \lp1+2\log n\rp\prob(T_{st}=j).
\end{align*}
Substituing this upper bound on the right hand side of \cref{eq:lssbnd-eq2} it now follows that 
\begin{align*}
    \expec\lb \Hcal^2(\bar \density_m,\hat \density_m) \rb & \leq \frac{1}{2n} \sum_{k\in m} \lp 1+(1+2\log n)\sum_{j=0}^{n-2}\prob(T_{st}=j)+ \rdot\\ 
    &\qquad \ldot  \expec\lb\frac{\lp \expec\lb \indicator_k(X_{n-1},a_{n-1},X_n) |X_{n-1}, a_{n-1}\rb - \indicator_k(X_{n-1},a_{n-1},X_{n})\rp^2}{\expec\lb \indicator_k(X_{n-1},a_{n-1},X_n) |X_{n-1}, a_{n-1}\rb}\indicator_{T_{st}=n-1}\rb\rp 
\end{align*}
But when $T_{ts}=n-1$, $\indicator_k(X_{n-1},a_{n-1},X_{n})=0$, and using the fact $\expec\lb \indicator_k(X_{n-1},a_{n-1},X_n) |X_{n-1}, a_{n-1}\rb\in[0,1]$, we get 
\[
\lp\expec\lb \indicator_k(X_{n-1},a_{n-1},X_n) |X_{n-1}, a_{n-1}\rb\rp^2< \expec\lb \indicator_k(X_{n-1},a_{n-1},X_n) |X_{n-1}, a_{n-1}\rb.
\] 
Collecting all the previous facts and substituting them into the right hand side of \cref{eq:lssbnd-eq2} we now get,  
\begin{align*}
    & \frac{1}{2n}\sum_{k\in m}\lp1+\sum_{j=0}^{n-1}\expec\lb\frac{\lp\sum_{i=T_{st}}^{n-1} \expec\lb \indicator_k(X_i,a_i,X_{i+1}) |X_i, a_i\rb - {\sum_{i=T_{st}}^{n}{\indicator_k(X_i,a_i,X_{i+1})}}\rp^2}{\sum_{i=T_{st}}^{n-1} \expec\lb \indicator_k(X_i,a_i,X_{i+1}) |X_i, a_i\rb}\indicator_{T_{st}=j}\rb\rp\\
    & \qquad \leq \frac{1}{2n} \sum_{k\in m} \Bigg( 1+(1+2\log n)\sum_{j=0}^{n-2}\prob(T_{st}=j) \\
    & \qquad \qquad +\expec\lb\frac{\expec\lb \indicator_k(X_{n-1},a_{n-1},X_n) |X_{n-1}, a_{n-1}\rb}{\expec\lb \indicator_k(X_{n-1},a_{n-1},X_n) |X_{n-1}, a_{n-1}\rb}\indicator_{T_{st}=n-1}\rb\Bigg)\\ 
    &\qquad \leq \frac{1}{2n} \sum_{k\in m} \lp 1+(1+2\log n)\sum_{j=0}^{n-2}\prob(T_{st}=j)+\prob(T_{st}=n-1)\rp\\
    &\qquad  \leq  \frac{1}{2n} \sum_{k\in m} \lp 2+(1+2\log n)\sum_{j=0}^{n-2}\prob(T_{st}\neq n-1)\rp\\
    &\qquad  \leq \frac{1}{2n} |m|(3+2\log n).
\end{align*}    
It now follows from \cref{eq:lssbnd-eq10} that $\expec\lb \Hcal^2(\bar \density_m,\hat \density_m) \rb\leq \frac{1}{2n} |m|(3+2\log n)$ as required.
\paragraph{CASE II.}
\[
\indicator_k(X_{T_{ts}},a_{T_{ts}},X_{T_{ts}}) = 1 \textit{ and } \expec\lb \indicator_k(X_{T_{ts}},a_{T_{ts}},X_{T_{ts}}) |X_{T_{ts}}, a_{T_{ts}}\rb< \frac{1}{n} \numberthis\label{eq:lssbnd-eq5}
\]
For this case, we use the inequality $\lp\sqrt a -\sqrt b\rp^2\leq (a-b)^2/b$ by substituting $b=\sum{\indicator_k(X_i,a_i,X_{i+1})}$ and $a= \sum_{i=T}^{n-1} \expec\lb \indicator_k(X_i,a_i,X_{i+1}) |X_i, a_i\rb$ and create the G.O.F.$_1$ statistic $(observed-expected)^2/observed$. Then, we proceed similarly as before to get the following counterpart to \cref{eq:lssbnd-eq3} 
\begin{align*}
    \expec[\text{G.O.F.}_1] &  \leq \expec\lb \lp 1+\log n - \log \indicator_k(X_i,a_i,X_{i+1}) \rp\indicator_{T_{st}=j}\rb\\ 
    & =   \expec\lb \lp 1+\log n - \log 1 \rp\indicator_{T_{st}=j}\rb\\ 
    & =   \expec\lb \lp 1+\log n \rp\indicator_{T_{st}=j}\rb.
\end{align*}
Which in turn implies that,
\begin{align*}
    \expec\lb \Hcal^2(\bar \density_m,\hat \density_m) \rb \leq \frac{1}{2n} |m|(3+\log n)
\end{align*}
which can be trivially upper bounded by $|m|(3+2\log n)/2n$.
This completes the proof.
\end{proof}

\subsection{Proof of Proposition \ref{prop:main-concentration}}\label{sec:prf-mainconc}
\begin{proof}
    We divide the proof of this proposition in two parts.
    \paragraph*{CASE I  $\bigg( T(\hat \density_m,\hat \density_{\hat m})-pen(\hat m) +pen(m) \bigg)\geq 0$ :}  
    Following Proposition \ref{prop:conc-m1m2}, it holds with probability at most $\exp\lp -n(pen(m)+pen(\hat m))/\kappa-n\zeta \rp$ that 
    \begin{align*}
        \frac{3}{4}\lp 1-\frac{1}{\sqrt{2}}\rp\Hcal^2(\hat \density_m,\hat \density_{\hat m}) & \leq \frac{3}{4}\lp 1-\frac{1}{\sqrt{2}}\rp\Hcal^2(\hat \density_m,\hat \density_{\hat m})+T(\hat \density_m,\hat \density_{\hat m})-pen(\hat m) +pen(m) \\
        & \leq \frac{5}{4}\lp 1+\frac{1}{\sqrt{2}} \rp\Hcal^2(\hat \density_m,\hat \density_m) + 2pen(m)\\
        & \leq \frac{5}{4}\lp 1+\frac{1}{\sqrt{2}} \rp\Hcal^2(\hat \density_m,\hat \density_m) + 2pen(m)+\zeta.
    \end{align*}
    Since $\exp\lp -n(pen(m)+pen(\hat m))/\kappa-n\zeta \rp$ can be upper bounded trivially by $6\exp(-n\zeta)$, the rest follows. We now proceed to address the other case.
    \paragraph{CASE II $\bigg( T(\hat \density_m,\hat \density_{\hat m})-pen(\hat m) +pen(m) \bigg)\leq 0$ : } 
    Observe that $T(f_1,f_2)=-T(f_2,f_1)$. Therefore, $T(\hat \density_{\hat m},\hat \density_m)+pen(\hat m) -pen(m)\geq 0$. This further implies that,
    \begin{align*}
        \frac{3}{4}\lp 1-\frac{1}{\sqrt{2}}\rp\Hcal^2(\hat \density_m,\hat \density_{\hat m}) & \leq \frac{3}{4}\lp 1-\frac{1}{\sqrt{2}}\rp\Hcal^2(\hat \density_m,\hat \density_{\hat m})+T(\hat \density_{\hat m},\hat \density_m)+pen(\hat m) -pen(m)\\
    \end{align*}    
    We now require the following lemma which serves to provide an upper and lower bound for $\gamma(m)$.
    \begin{lemma}~\label{lemma:g-ublb}
        Let $\gamma$ be as defined in \cref{def:gamma}. Then,
        \begin{align*}
            & \sup_{m'\in \Mcal_\lcal} \lb \frac{3}{4}\lp 1-\frac{1}{\sqrt{2}}\rp\Hcal^2(\hat \density_m,\hat \density_{m'})+T(\hat \density_m,\hat \density_{m'}) -pen(m')\rb+pen(m) \leq \gamma(m)  \\
            &  \gamma(m) \leq \sup_{\substack{f\in \density_{m'}\\m'\in \Mcal_\lcal}} \lb \frac{3}{4}\lp 1-\frac{1}{\sqrt{2}}\rp\Hcal^2(\hat \density_m,f)+T(\hat \density_m,f) -pen(m')\rb+2pen(m)
        \end{align*}
    \end{lemma}
    The proof of the first inequality is by using Proposition \ref{prop:partition} Item \ref{assume:part4} and some careful book-keeping. It follows similarly to that of Lemma B.2 in \cite{sart_estimation_2014}. The proof of the second inequality can be found in Section \ref{sec:prf-gublb}.
    Using Lemma \ref{lemma:g-ublb}, we get 
    \begin{align*}
        & \frac{3}{4}\lp 1-\frac{1}{\sqrt{2}}\rp\Hcal^2(\hat \density_m,\hat \density_{\hat m})+T(\hat \density_{\hat m},\hat \density_m)+pen(\hat m) -pen(m)\\
        & \leq \gamma(\hat m)\\
        & \leq \gamma(m)+\frac{1}{n}\\
        & \leq \sup_{\substack{f\in \density_{m'}\\m'\in \Mcal_\lcal}} \lb \frac{3}{4}\lp 1-\frac{1}{\sqrt{2}}\rp\Hcal^2(\hat \density_m,f)+T(\hat \density_m,f) -pen(m')\rb+2pen(m)+\frac{1}{n}
    \end{align*}
    where the second inequality follows form \cref{eq:model} and the last inequality follows from the fact that $\hat \density_m\in \bigcup_{\substack{f\in \density_{m'}\\m'\in \Mcal_\lcal}}f$ for all $m$.
    It is now enough to show that it happens with low probability that
    \begin{align*}
        \sup_{\substack{f\in \density_{m'}\\m'\in \Mcal_\lcal}} \lb \frac{3}{4}\lp 1-\frac{1}{\sqrt{2}}\rp\Hcal^2(\hat \density_m,f)+T(\hat \density_m,f) -pen(m')\rb+2pen(m)+\frac{1}{n}
    \end{align*}
    Taking an union bound over $f\in \density_{m'}$, we get 
    \begin{align*}
        & \prob\Bigg(\sup_{\substack{f\in \density_{m'}\\m'\in \Mcal_\lcal}} \lb \frac{3}{4}\lp 1-\frac{1}{\sqrt{2}}\rp\Hcal^2(\hat \density_m,f)+T(\hat \density_m,f) -pen(m')\rb+pen(m)+\frac{1}{n} \\
        & \qquad \leq  \frac{5}{4}\lp1+\frac{1}{\sqrt{2}}\rp\Hcal^2(\density ,f_1)+2pen(m)+\zeta  + \frac{1}{n} \Bigg )\numberthis\label{eq:union-bound}\\
        & \leq \sum_{\substack{f\in \density_{m'}\\m'\in \Mcal_\lcal}}\prob\Bigg( \lb \frac{3}{4}\lp 1-\frac{1}{\sqrt{2}}\rp\Hcal^2(\hat \density_m,f)+T(\hat \density_m,f) -pen(m')\rb+pen(m)+\frac{1}{n} \\
        & \qquad \leq  \frac{5}{4}\lp1+\frac{1}{\sqrt{2}}\rp\Hcal^2(\density ,f_1)+2pen(m)+\zeta  + \frac{1}{n} \Bigg )
    \end{align*}
    We can now upper bound the probability using Proposition \ref{prop:conc-m1m2} by substituting $\zeta$ by $\zeta+n^{-1}$. We get,
    \begin{align*}
        & \sum_{\substack{f\in \density_{m'}\\m'\in \Mcal_\lcal}}\prob\Bigg( \lb \frac{3}{4}\lp 1-\frac{1}{\sqrt{2}}\rp\Hcal^2(\hat \density_m,f)+T(\hat \density_m,f) -pen(m')\rb+pen(m)+\frac{1}{n} \\
        & \qquad \leq  \frac{5}{4}\lp1+\frac{1}{\sqrt{2}}\rp\Hcal^2(\density ,f_1)+2pen(m)+\zeta  + \frac{1}{n} \Bigg )\\
        & \leq \sum_{\substack{f\in \density_{m'}\\m'\in \Mcal_\lcal}} \exp(-n (pen(m)+pen(m'))/\kappa-n\zeta -1 ).
    \end{align*}
    To calculate this sum, we now need to compute the cardinality of $\bigcup_{\substack{f\in \density_{m'}\\m'\in \Mcal_\lcal}}f$.
    It follows by the construction in \cref{def:SM} that the cardinality of the set
    \[
    \lv\bigcup_{i=0}^\lcal\lc \frac{a}{b \mu_{\Ibb} \lp K_i\pow 2\rp\muc \lp K_i\pow 3\rp }:a\in\lc0,\dots,n\rc, b\in\lc 1,\dots,n \rc \rc\rv
    \]
    is $\lcal(n+1)n$. Since $\lcal\leq n$, then $\lcal(n+1)n\leq n^2(n+1)$ which in turn can be upper bounded as $n^2(n+1)\leq 1.5 n^3$ as long as $n\geq 3$. It follows that
    \[
    \lv \density_{m'}\rv\leq 1.5^{|m'|}n^{3|m'|} = \exp(|m'|(3\log(n)+\log(1.5))
    \]
    Recall from Remark \ref{remark:penalty} that $pen(m')$ was defined to be $L|m'|(1.5+\log n)/n$ for some $L\geq 3$. It therefore follows that 
    \begin{align*}
        |\density_{m'}|\exp(-22n\times pen(m')/L) & \leq \exp(|m'|(3\log(n)+\log(1.5))-22|m'|(1.5+\log n)/L)\\
        & \leq \exp(-1.824 |m'|)\\
        & \leq \exp(-|m'|).
    \end{align*}
    It therefore follows from Proposition \ref{prop:partition} Item \ref{assume:part1} that
    \begin{align*}
        \sum_{\substack{f\in \density_{m'}\\m'\in \Mcal_\lcal}} \exp(-22n\times pen(m')/L) & \leq \sum_{m'\in\Mcal_\lcal} \exp(-|m'|)\\
        & \leq 15.
    \end{align*}
    Trivially bounding $\exp(-22n\times pen(m)/L)$ from above by $1$, we get
    \begin{align*}
        \sum_{\substack{f\in \density_{m'}\\m'\in \Mcal_\lcal}} \exp(-22n (pen(m)+pen(m'))/L-n\zeta -1 ) & \leq  15 \exp(-n\zeta -1 )\\
        & \leq 6\exp(-n\zeta).
    \end{align*}
    This completes the proof.
\end{proof}

\subsection{Proof of Proposition \ref{prop:suffdepth}}~\label{sec:prf-suffdepth}
\begin{proof}
    Let $\lcal\geq n$ and $m_1,m_2\in \Mcal_\infty$ and let $K\in m_1$. We define the $\gamma_K$ as
    \[
    \gamma_K(m_1,m_2) := \frac{\sqrt{2}-1}{2\sqrt{2}}\Hcal^2(\hat \density_{m_1}\indicator_K,\hat \density_{m_2}\indicator_K)+T(\hat \density_{m_1}\indicator_K,\hat \density_{m_2}\indicator_K)-pen(m_2\vee K).
    \]
    $\gamma_K$ compares the relative performance of the histograms $\hat \density_{m_1}$ and $\hat \density_{m_2}$ on the set $K\in m_1$.
    Let $m_2^\star:= \argmax_{m_2\in \Mcal_\infty}\gamma_K(m_1,m_2)$. Using the fact that $\Hcal^2(\cdot,\cdot)\leq 1$ and $|T(\cdot,\cdot)|\leq 2$ we get
    \[
    -2-pen(\chi\times\Ibb\times\chi)\leq \gamma_K(m_1,\chi\times\Ibb\times\chi){\leq} \gamma_K(m_1,m_2^\star)\leq 3- pen(m_2^\star\vee K)
    \]
    with the second inequality following by definition.
    Since $pen(m)=L (1.5+\log n)|m|/n$, and $|\chi\times\Ibb\times\chi|=1$
    \[
    -2-L(1.5+\log n)/n\leq 3-L(1.5+\log n)|m_2^\star\vee K|/n.
    \]
    This, with a bit of rearrangement implies
    \[
    |m_2^\star\vee K|\leq 1+\frac{5n}{L(1.5+\log n)}\leq n.
    \]
    Therefore, there exists $m_2^\oplus$ such that $m_2^\oplus\in\Mcal_n$ and $m_2^\oplus\vee K=m_2^\star\vee K$, which implies $m_2^\oplus$ also maximises $\gamma_k(m_1,m_2)$. Therefore,
    \[
    \max_{m_2\in \Mcal_\infty}\gamma_K(m_1,m_2) = \max_{m_2\in \Mcal_n}\gamma_K(m_1,m_2).
    \]
    We define $m^\star:=\argmin_{m\in \Mcal_\lcal}\gamma(m)$. It is obvious from definition that  $\gamma(m^\star)\leq \gamma(\chi\times\Ibb\times\chi)\leq 3+L (1.5+\log n)/n$. We observe from Lemma \ref{lemma:g-ublb} that 
    \begin{align*}
        \gamma(m^\star)& \geq \sup_{m'\in\Mcal_\lcal}\sum_K\gamma_K(m^\star,m')+pen(m^\star)\\
        & \geq -2-pen(\chi\times\Ibb\times\chi)+pen(m^\star)\\
        & \geq -2 - \frac{L (1.5+\log n)(|m^\star|-1)}{n}
    \end{align*}
    Some simple calculations now show that $|m^\star|\leq 2+5n/(1.5+\log n)\leq n$, which implies $m^\star\in\Mcal_n$.
\end{proof}
{ 
\subsection{Proof of Lemma \ref{lemma:erg-vs-recurring}}\label{sec:prf-ergvsrec}
\begin{proof}
    The basic idea is to create a recurring sequence whose C\'esaro sum does not converge. We consider $\Ibb=\{-1,1\}$, $\chi=\{-1,1\}$ and let $\mu_\Ibb$ and $\mu_\chi$ be counting measures. Similar counter examples can be easily constructed for more general spaces. Define
\[
    \prob\bigl(X_{i+1} = -1 \,\bigm|\,
    a_i = -1,\,X_i = s\bigr) := 1
    \quad\text{and}\quad
    \prob\bigl(X_{i+1} = 1 \,\bigm|\,
    a_i = 1,\,X_i = s\bigr) := 1
    \quad\forall\ i\geq 0, s\in\chi.
\]
Set the controls as \(a_i = (-1)^{\lfloor \log_2(i)\rfloor}\). By construction, the waiting times are deterministic and finite, so that \(T(\Scal) < \infty\).

A trite but straightforward calculation shows that 
\begin{equation*}
\nu_n\lp(1,1)\rp
  = \frac{4^{\lceil k/2\rceil}-1}{6n}
    + \frac{1 + (-1)^{k}}{4n}\,(r+1),
\qquad
k = \bigl\lfloor \log_2 n \bigr\rfloor,
\;
r = n - 2^{k}.
\end{equation*}
Hence, \(\lim_{n\to\infty}\nu_n\bigl((1,1)\bigr)\) does not exist. The same argument applies to \(\nu_n\bigl((-1,1)\bigr)\), \(\nu_n\bigl((1,-1)\bigr)\), and so on. This completes the proof.

\end{proof}
}
\subsection{Proof of Proposition \ref{prop:1-better-than-2}}\label{sec:prf-1betthan2}
\begin{proof}
We actually compare the remainder terms obtained via Proposition \ref{prop:detlos}. The only difference is the extra $r_n$ term does not appear in $R\pow 1(n)$. This makes the comparison fair, since otherwise we are comparing $h^2$ to $h_n^2$.

Let $\chi=\Ibb=[0,1/2)\times[0,1/2)\bigcup[1/2,1]\times[1/2,1]$. We set $\muc$ and $\mu_\Ibb$ to be Lebesgue measures. Let the true $\density$ be such that  
\begin{align*}
    \density(x,l,y) & = \begin{cases}
        2 & \text{if } l,y\in [0,1/2)\\
        2 & \text{if } l,y\in [1/2,1]\\
        0 & \text{otherwise}.
    \end{cases}
\end{align*}
    Therefore, for all $i\geq 0$ the states $X_i$'s are i.i.d Uniform distributions over $[0,1/2)$ or $[1/2,1]$ in accordance with the value of $l$.

    Observe that $\density$ is a piecewise constant function on a dyadic partition. So it can be perfectly approximated by histograms on $\Mcal_\infty$. Therefore,
    \begin{align*}
    \sup_{m\in\Mcal_\infty}\lc h^2(\density,V_m)+pen(m)\rc=\sup_{m\in\Mcal_\infty}\lc h_n^2(\density,V_m)+pen(m)\rc=\frac{L(1.5+\log n)}{n}    
    \end{align*}
    for some universal constant $L$ with the minimum achieved by the partition $\chi\times\Ibb\times\chi$. 

    The controls $a_i$ are as follows: For a fixed integer $i_0\geq 1$, and $i\in\{0,\dots,i_0-1\}$,
    \begin{align*}
        a_i \sim \begin{cases}
      \text{Uniform[0,1/2) with probability} \ \frac{1}{i_0}\\ 
     \text{Uniform[1/2,1] otherwise},
    \end{cases}
    \end{align*}
    and for $i\geq d$
    \begin{align*}
        a_i \sim \begin{cases}
      \text{Uniform[0,1/2) with probability} \ \frac{1}{2}\\ 
     \text{Uniform[1/2,1] otherwise}.
    \end{cases}
    \end{align*}
    In essence, $(X_i,a_i)$ is an i.n.i.d sequence taking values in $[0,1/2)\times[0,1/2)\bigcup[1/2,1]\times[1/2,1]$. Let $\density\pow {\nu_n}$ denote the density of $\nu_n$ and $\density \pow \nu$ denote the density of $\nu$.

    The following form for total variation distance will be useful. Let $A^+$ be any set such that $\inf_{(x,l)\in A^+} \{\density\pow \nu(x,l)-\density\pow {\nu_n}(x,l)\}\geq 0$ and $A^-$ be any set such that $\inf_{(x,l)\in A^-} \{\density\pow \nu(x,l)-\density\pow {\nu_n}(x,l)\}\leq 0$. Note that
    \begin{align*}
        \|\nu-\nu_n\|_{TV}  = \max & \lc \sup_{A^+}\int_{(x,l)\in A^+}\lp\density\pow \nu(x,l)-\density\pow {\nu_n}(x,l)\rp dxdl \ , \rdot\\
        & \qquad \ldot\sup_{A^-}\int_{(x,l)\in A^-}\lp\density\pow {\nu_n}(x,l)-\density\pow \nu(x,l)\rp dxdl \rc\numberthis\label{eq:tv-def}
    \end{align*}
    We remark that we have suppressed the dependence of $A^+$ and $A^-$ on $n$ from the notation.

    Now we derive $\density\pow\nu$ and $\density\pow{\nu_n}$. It can be easily seen that $\nu$ is an uniform distribution on $\chi\times \Ibb$. We denote its density by $\density\pow \nu$ where
    \begin{align*}
    \density\pow \nu(x,l) = \begin{cases}
            2 & \text{ if } \ (x,l)\in [0,1/2)\times[0,1/2)\\
            2 & \text{ if } \ (x,l)\in [1/2,1]\times[1/2,1]\\
            0 & \text{otherwise}.
        \end{cases} 
    \end{align*}
    
    We denote the density of $(X_0,a_0)$ by $\density\pow{\nu_0}$ where
    \begin{align*}
        \density\pow{\nu_0}(x,l) = \begin{cases}
            \frac{4}{i_0} & \text{ if } \ (x,l)\in [0,1/2)\times[0,1/2)\\
            4(1-\frac{1}{i_0}) & \text{ if } \ (x,l)\in [1/2,1]\times[1/2,1]\\
            0 & \text{otherwise}.
        \end{cases}
    \end{align*}
    For simplicity, let $n\geq i_0$ and recall that $r_n:=\|\nu-\nu_n\|_{TV}$. Observe that by the linearity of the differential operator that 
    \begin{align*}
        \density\pow{\nu_n} = \frac{i_0}{n}\density\pow{\nu_0}+\frac{n-i_0}{n}\density\pow{\nu}.
    \end{align*}

    Using \cref{eq:tv-def} it is now easy to see that $r_n= \Theta(i_0/n)$. We turn to $T(\Scal)$.
  
    
    Let $\Scal_\dagger\subseteq [0,1/2)\times [0,1/2)$. $\indicator[(X_i,a_i)\in\Scal_\dagger]$ are independent Bernoulli trials with probability of success $4\Vol{\Scal_\dagger}/i_0$ if $i\in{0,\dots,i_0
    -1}$, and $4\Vol{\Scal_\dagger}$ if $i\geq i_0$. 
    Consider $\tau_{\Scal_\dagger}\pow 1$. Therefore, $T(\Scal_\dagger)=\expec[\tau_{\Scal_\dagger}\pow 1]\geq i_0/\Vol{\Scal_\dagger}$.

    Recall that the partition minimising the oracle risk was $\chi\times\Ibb\times\chi$. Therefore, $m_{ref}\pow 2=\chi\times\Ibb$, and we can write the following expressions for $\Rcal\pow 1 (n)$ and $\Rcal\pow 2(n)$.  The only important thing to note here is the fact that the multiplicative term for $n$ in the numerator of the exponents is larger for $\Rcal\pow 1(n)$ for all values of $i_0$. 

    Define
    \begin{align*}
    \Scal_\star:=\argmax_{\Scal_r\in m_{ref}\pow 2}    \exp\lp- \frac{\Constant_pn\nu_n^2(\Scal_{r})}{4\Constant_\Delta\sup_{i,j}\sqrt{\prob\lp (X_i,a_i)\in \Scal_r,(X_j,a_j)\in \Scal_r\rp} +4n^{-1}+2\nu_n(\Scal_{r})(\log n)^2}\rp.
    \end{align*}
    Then,
    \begin{align*}
            \Rcal\pow 1(n) & = 2^{2} \exp\lp- \frac{\Constant_pn\nu^2(\Scal_{\star})-2n\Constant_pr_n}{4\Constant_\Delta\rho_\star(\Scal_\star) +4n^{-1}+2\nu(\Scal_{\star})(\log n)^2+2r_n(\log n)^2}\rp\\
            & = 4\exp\lp- \frac{4\Constant_pn\Vol{\Scal_{\star}}^2-2i_0\eta_1\Constant_p}{4\Constant_\Delta\rho_\star(\Scal_\star) +4n^{-1}+4\Vol{\Scal_{\star}}(\log n)^2+2\eta_2i_0\frac{(\log n)^2}{n}}\rp
    \end{align*}
    where $\eta_1$ and $\eta_2$ are some positive constants.    Similarly,
    \begin{align*}
            \Rcal\pow 2(n) = 2^{2} \exp\lp- \frac{ \frac{\Constant_pn\Vol{\Scal_\star}^2}{4i_0^2}}{4\Constant_\Delta\rho_\star(\Scal_\star) +\frac{(4+(\log n)^2)\Vol{\Scal_\star}}{2i_0}}\rp.
    \end{align*}   

    Since the multiplicative term for $n$ in the numerator of of $\Rcal\pow 1(n)$ is larger, it immediately follows that $\Rcal\pow 1(n)/\Rcal\pow 2(n)\rightarrow0$. Thus $\Rcal\pow 1(n)=\ocal\lp\Rcal\pow 1(n)\rp$.
    
    Next, by setting $i_0 = \Theta\lp\sqrt{n(\log n)^2\log(n\log n)}\rp$, we get that $R\pow 1(n)=\Ocal(\log n/n)$. The rest of the proof follows.

\end{proof}

\subsection{Proof of Lemma \ref{lemma:mixing-lemma}}~\label{sec:prf-mixinglemma}
\begin{proof}
Recall from \cite[eq. 1.2]{bradley_basic_2005} the definition of $\phi$ mixing coefficients. Now using  \cite[eq. 1.11]{bradley_basic_2005} we get $\alpha_{i,j}\leq \phi_{i,j}$. It is therefore sufficient to bound $\phi_{i,j}$.
Define the weak mixing coefficients $\bar\theta_{i,j}$ as
\[
\bar\theta_{i,j}:= \sup_{s_1,s_2\in\chi,l_1,l_2\in \Ibb}\| \probl\lp X_j,a_j|X_i=s_1,a_i=l_1\rp-\probl\lp X_j,a_j|X_i=s_2,a_i=l_2\rp \|_{TV}\numberthis\label{eq:weak-mixing},
\]
and observe from \cite[Lemma 1]{banerjee_off-line_2025} that $\phi_{i,j}\leq \bar\theta_{i,j}$.
Therefore, it is sufficient to prove
\[
\bar\theta_{i,j}\leq \lp1-\Vol{\chi_0}\kappa\rp^{j-i-1}.
\]
 Let the density of $a_i$ be denoted by $\density\pow i(x,l')$ defined as
\[
\density\pow i(x,l') := \prob\lp a_i\in dl'|X_i=x \rp.
\]
We make note that $(X_i,a_i)$ forms an inhomogenous Markov chain with the probability of transition from $(x,l)$ to $(y,l')$ at time point $i$ is $\density(x,l,y)\density\pow i(y,l')$. 
It follows from \citet[Theorem 2]{hajnal_weak_1958} that
\begin{align*}
\bar\theta_{i,j}
& \leq \prod_{p=i}^{j-1}\lp 1-\min_{(s_1,l_1),(s_2,l_2)\in\chi\times\Ibb}\int_{(t,l')\in \chi\times\Ibb} \min\lc \lp\density(x_1,l_1,t)\density\pow i(t,l')\rp,\lp\density(x_2,l_2,t)\density\pow i(t,l')\rp \rc dl'dt\rp.\numberthis\label{eq:mark-asseq1}
\end{align*}
Recall that by hypothesis
\[
\min_{x\in\chi,l\in\Ibb}\density(x,l,t)>\kappa,
\]
for any $t\in\chi_0$.
This implies that for all $t\in\chi_0$,
\[
\min\lc \lp\density(x_1,l_1,t)\density\pow i(t,l')\rp,\lp\density(x_2,l_2,t)\density\pow i(t,l')\rp \rc\geq \kappa \density\pow i(t,l').
\]
Decomposing the integral over $(t,l')\in\chi\times\Ibb$ in \cref{eq:mark-asseq1} into an intergral over $(t,l)\in(\chi\backslash\chi_0)\times\Ibb$ and $(t,l)\in \chi_0\times\Ibb$ and substituting $\kappa \density\pow i(t,l')$ as the appropriate lower bound we get,
\begin{align*}
    \int_{(t,l')\in \chi\times\Ibb} & \min\lc \lp\density(x_1,l_1,t)\density\pow i(t,l')\rp,\lp\density(x_2,l_2,t)\density\pow i(t,l')\rp \rc dl'dt \\
    & \geq \int_{t\in\chi_0,l'\in \Ibb} \min\lc \lp\density(x_1,l_1,t)\density\pow i(t,l')\rp,\lp\density(x_2,l_2,t)\density\pow i(t,l')\rp \rc dl'dt\\
    & \geq \int_{t\in\chi_0,l'\in \Ibb} \kappa \density\pow i(t,l')dl'dt\\
    & = \Vol{\chi_0}\kappa.
\end{align*}
Now it follows that the right hand side of \cref{eq:mark-asseq1} can be upper bounded by 
\begin{align*}
   \mathrm{R.H.S.\ of\ \cref{eq:mark-asseq1}} \leq \prod_{p=i}^{j-1}\lp 1-\Vol{\chi_0}\kappa\rp \ = \lp1-\Vol{\chi_0}\kappa\rp^{j-i-1},
\end{align*}
which completes our initial claim.
\end{proof}

\subsection{Proof of Proposition \ref{prop:return-time-markov}}\label{sec:prf-rtm}

\begin{proof} 
{ 
\renewcommand{\Scal}{\mathcal{S}}
We begin by representing $\tau_{\Scal}\pow{i}$ in terms of $\tau_{\Scal}\pow{i,\star,j}$'s. Observe that $\tau_\Scal\pow{i,\star,j}$ is constructed so that $\tau\pow i_\Scal$ is $\tau_{\Scal}\pow{i,\star,1}$ if the state at the corresponding time is inside $\Scal_\chi$; it is $ \tau_{\Scal}\pow{i,\star,1}+\tau_{\Scal}\pow{i,\star,2}$ if the state was not in $\Scal_\chi$ after $\tau_{\Scal}\pow{i,\star,1}$ time points and $\Scal_\chi$ after $\tau_{\Scal}\pow{i,\star,1}+\tau_{\Scal}\pow{i,\star,2}$ time points, and so on. Formally, this means
\begin{align*}
    \tau_{\Scal}\pow{i+1}=\begin{cases}
   \tau_{\Scal}\pow{i,\star,1}\ & \text{ if } \lc X_{\sum_{p=1}^i \tau_{\Scal}\pow{p}+ \tau_{\Scal}\pow{i,\star,1}}\in\Scal_\chi\rc\\
    \tau_{\Scal}\pow{i,\star,1}+\tau_{\Scal}\pow{i,\star,2}\  &\text{ if } \lc X_{\sum_{p=1}^i \tau_{\Scal}\pow{p}+ \tau_{\Scal}\pow{i,\star,1}}\notin \Scal_\chi\text{ and } \ X_{\sum_{p=1}^i \tau_{\Scal}\pow{p}+ \tau_{\Scal}\pow{i,\star,1}+ \tau_{\Scal}\pow{i,\star,2}}\in \Scal_\chi\rc\\
    \quad \vdots & \vdots
    \end{cases}
\end{align*}

Therefore,
\begin{small}
\begin{align*}
    \tau_{\Scal}\pow{i+1} & =  \tau_{\Scal}\pow{i,\star,1}\indicator\lb X_{\sum_{p=1}^i \tau_{\Scal}\pow{p}+ \tau_{\Scal}\pow{i,\star,1}}\in\Scal_\chi\rb\\
    & \ + \lp\tau_{\Scal}\pow{i,\star,1}+\tau_{\Scal}\pow{i,\star,2}\rp\indicator\lb X_{\sum_{p=1}^i \tau_{\Scal}\pow{p}+ \tau_{\Scal}\pow{i,\star,1}}\notin \Scal_\chi,\ X_{\sum_{p=1}^i \tau_{\Scal}\pow{p}+ \tau_{\Scal}\pow{i,\star,1}+ \tau_{\Scal}\pow{i,\star,2}}\in \Scal_\chi\rb\\
    & \ + \dots\ ,
\end{align*}
and taking a conditional expectation on both side provides the following identity
\begin{align*}
    & \expec[\tau_\Scal\pow{i+1}|\Fcal_{\sum_{p=1}^{i-1} \tau_\Scal\pow{p}}] \\
    & \ =  \expec\lb\tau_\Scal\pow{i,\star,1}\indicator\lb X_{\sum_{p=1}^i \tau_\Scal\pow{p}+ \tau_\Scal\pow{i,\star,1}}\in\Scal_\chi\rb|\Fcal_{\sum_{p=1}^{i-1} \tau_\Scal\pow{p}}\rb\\
    & \quad + \expec\lb\lp\tau_\Scal\pow{i,\star,1}+\tau_\Scal\pow{i,\star,2}\rp\indicator\lb X_{\sum_{p=1}^i \tau_\Scal\pow{p}+ \tau_\Scal\pow{i,\star,1}}\notin \Scal_\chi,\ X_{\sum_{p=1}^i \tau_\Scal\pow{p}+ \tau_\Scal\pow{i,\star,1}+ \tau_\Scal\pow{i,\star,2}}\in\Scal_\chi\rb|\Fcal_{\sum_{p=1}^{i-1} \tau_\Scal\pow{p}}\rb\\
    & \quad + \dots\\
    & = \text{Term 1}+\text{Term 2}+\dots\ .\numberthis\label{eq:ret-t-markeq4}
\end{align*}
\end{small}

To compute an upper bound to $\expec[\tau_\Scal\pow{i}]$, it is thus sufficient to individually find an upper bound to each term of the summation in the right-hand side of the previous equation by a careful bookkeeping of the conditional expectations.

\textbf{Term 1: }Applying the law of conditional expectation to the first term we get
\begin{align*}
    & \expec\lb\tau_\Scal\pow{i,\star,1}\indicator\lb X_{\sum_{p=1}^i \tau_\Scal\pow{p}+ \tau_\Scal\pow{i,\star,1}}\in\Scal_\chi\rb|\Fcal_{\sum_{p=1}^{i-1} \tau_\Scal\pow{p}}\rb \\
    &\  =  \expec\lb \expec\lb\tau_\Scal\pow{i,\star,1}\indicator\lb X_{\sum_{p=1}^i \tau_\Scal\pow{p}+ \tau_\Scal\pow{i,\star,1}}\in \Scal_\chi\rb\gn\tau_\Scal\pow{i,\star,1} \rb|\Fcal_{\sum_{p=1}^{i-1} \tau_\Scal\pow{p}} \rb\\
    &\ = \expec\lb\tau_\Scal\pow{i,\star,1}\underbrace{\prob\lp X_{\sum_{p=1}^i \tau_\Scal\pow{p}+ \tau_\Scal\pow{i,\star,1}}\in\Scal_\chi| \tau_\Scal\pow{i,\star,1} \rp}_{\mathrm{=:A}}|\Fcal_{\sum_{p=1}^{i-1} \tau_\Scal\pow{p}}\rb \numberthis\label{eq:ret-t-markeq2}
\end{align*}
where the second equality follows from tower property since \[\Fcal_{\sum_{p=1}^{i-1} \tau_{\Scal}\pow{p}}\subseteq \Fcal_{\sum_{p=1}^{i-1} \tau_{\Scal}\pow{p}+\sum_{p=1}^{j-1}\tau_{\Scal}\pow{i,\star,p}}.\]

Recall from \cref{eq:uniell-1} that $\density(x,l,y)\leq 1/\epsilon_0$. Therefore, for any time point $p$ and any history $\history_0^{p-1}$,
\begin{align*}
    \prob\lp X_p\in\Scal_\chi\gn \History_0^{p-1}=\history_0^{p-1}\rp\leq \Vol{\Scal_\chi}/\epsilon_0, \  \ \mathrm{ and }\tag{P.I}\\
     \prob\lp X_p\notin \Scal_\chi\gn \History_0^{p-1}=\history_0^{p-1}\rp\leq 1- \epsilon_0\Vol{\Scal_\chi^c}. & \tag{P.II}
\end{align*}
It follows from (P.I) that, $\mathrm{A}\leq \Vol{\Scal_\chi}/\epsilon_0$. Substituting this value in the right hand side of \cref{eq:ret-t-markeq2}, we get the following upper bound to Term 1
\[
 \expec\lb\tau_{\Scal}\pow{i,\star,1}\indicator\lb X_{\sum_{p=1}^i \tau_{\Scal}\pow{p}+ \tau_{\Scal}\pow{i,\star,1}}\in \Scal\rb|\Fcal_{\sum_{p=1}^{i-1} \tau_{\Scal}\pow{p}}\rb\leq \expec\lb \tau_{\Scal}\pow{i,\star,1}\frac{\Vol{\Scal_\chi}}{\epsilon_0}|\Fcal_{\sum_{p=1}^{i-1} \tau_{\Scal}\pow{p}}\rb\leq T_\star(\Scal) \frac{\Vol{\Scal_\chi}}{\epsilon_0}.
\]

\textbf{Term 2: } We turn to Term 2. We introduce the notation $\expec^*$ for convenience where
\[
\expec^*[\cdot]=\expec[\cdot|\Fcal_{\sum_{p=1}^{i-1} \tau_{\Scal}\pow{p}}]
\]
\textbf{Term 2: } We introduce some notation for convenience. Define 
\[
\expec^*[\cdot]:=\expec[\cdot|\Fcal_{\sum_{p=1}^{i-1} \tau_\Scal\pow{p}}]
\]
and proceed similarly as before to get
\begin{small}
\begin{align*}
    &\expec^*\lb\lp\tau_\Scal\pow{i,\star,1}\rp\indicator\lb X_{\sum_{p=1}^i \tau_\Scal\pow{p}+ \tau_\Scal\pow{i,\star,1}}\notin \Scal,\ X_{\sum_{p=1}^i \tau_\Scal\pow{p}+ \tau_\Scal\pow{i,\star,1}+ \tau_\Scal\pow{i,\star,2}}\in \Scal\rb\rb\\ 
    & = \expec^*\lb\lp\tau_\Scal\pow{i,\star,1}\rp\underbrace{\prob\lp X_{\sum_{p=1}^i \tau_\Scal\pow{p}+ \tau_\Scal\pow{i,\star,1}}\notin \Scal,\ X_{\sum_{p=1}^i \tau_\Scal\pow{p}+ \tau_\Scal\pow{i,\star,1}+ \tau_\Scal\pow{i,\star,2}}\in\Scal|\tau_\Scal\pow{i,\star,1},\tau_\Scal\pow{i,\star,2}\rp}_{=:\mathrm{B}}\rb.\numberthis\label{eq:ret-t-markeq3}
\end{align*}
\end{small}
We decompose B into 
\begin{align*}
& \underbrace{\prob\lp X_{\sum_{p=1}^i \tau_\Scal\pow{p}+ \tau_\Scal\pow{i,\star,1}+ \tau_\Scal \pow{i,\star,2}}\in \Scal|\tau_\Scal\pow{i,\star,1},\tau_\Scal\pow{i,\star,2}, X_{\sum_{p=1}^i \tau_\Scal\pow{p}+ \tau_\Scal\pow{i,\star,1}}\notin \Scal\rp}_{=:\mathrm{C}}\\
 & \qquad \qquad \times \underbrace{\prob\lp X_{\sum_{p=1}^i \tau_\Scal\pow{p}+ \tau_\Scal\pow{i,\star,1}}\notin \Scal|\tau_\Scal\pow{i,\star,1},\tau_\Scal\pow{i,\star,2}\rp}_{=:\mathrm{D}}
\end{align*}
We bound C using P.I, and D using P.II. This gives us
\begin{align*}
    \text{Right hand side of \cref{eq:ret-t-markeq3}}&\leq \frac{\Vol{\Scal_\chi}}{\epsilon_0} \times \lp1-\epsilon_0\Vol{\Scal_\chi}\rp \expec^*[\tau_\Scal\pow{i,\star,1}]\\
    & \leq \frac{\Vol{\Scal_\chi}}{\epsilon_0} \times \lp1-\epsilon_0\Vol{\Scal_\chi}\rp T\pow \star (\Scal).
\end{align*}
We similarly get
\begin{align*}
    &\expec^*\lb\lp\tau_\Scal\pow{i,\star,2}\rp\indicator\lb X_{\sum_{p=1}^i \tau_\Scal\pow{p}+ \tau_\Scal\pow{i,\star,1}}\notin \Scal,\ X_{\sum_{p=1}^i \tau_\Scal\pow{p}+ \tau_\Scal\pow{i,\star,1}+ \tau_\Scal\pow{i,\star,2}}\in \Scal\rb\rb\\
    &\qquad  \leq \frac{\Vol{\Scal_\chi}\lp1-\epsilon_0\Vol{\Scal_\chi}\rp}{\epsilon_0} T\pow \star (\Scal).
\end{align*}
Therefore, 
\begin{align*}
    \mathrm{Term 2} \leq 2\frac{\Vol{\Scal_\chi}}{\epsilon_0} \times\lp1-\epsilon_0\Vol{\Scal_\chi}\rp T\pow \star (\Scal).
\end{align*}

Proceeding similarly, we can find an upper bound to each term. Substituting these terms back into \cref{eq:ret-t-markeq4} we get
\begin{align*}
    \expec[\tau_\Scal\pow{i+1}|\Fcal_{\sum_{p=1}^{i-1} \tau_\Scal\pow{p}}] & \leq \sum_{j=1}^{\infty} j\frac{\Vol{\Scal_\chi}}{\epsilon_0} \times\lp1-\epsilon_0\Vol{\Scal_\chi}\rp^{j-1} T\pow \star (\Scal). \numberthis\label{eq:ret-t-markeq1}   
\end{align*}

By integrating the first inequality of eq. (\ref{eq:uniell-1}) with respect to $y\in\chi$, we have
\begin{align*}
    0<\Vol{\chi}\epsilon_0\leq 1 
\end{align*}
Consequently, $1-\Vol{\chi}\epsilon_0<1$ and $1-\Vol{\Scal_\chi}\epsilon_0<1$ for all $\Scal_\chi\subseteq\chi$. This makes the series in the right hand side of \cref{eq:ret-t-markeq1} convergent and we finally get,
\begin{align*}
    \sum_{j=1}^{\infty} j\frac{\Vol{\Scal_\chi}}{\epsilon_0} \times\lp1-\epsilon_0\Vol{\Scal_\chi}\rp^{j-1} T\pow \star (\Scal) & = \frac{T\pow \star(\Scal)}{\epsilon_0^2}\sum_{j=1}^\infty j\epsilon_0\Vol{\Scal_\chi}\lp1-\epsilon_0\Vol{\Scal_\chi}\rp^{j-1}\\
    & = \frac{T\pow \star(\Scal)}{\epsilon_0^3\Vol{\Scal_\chi}}.
\end{align*}
}
\end{proof}

\subsection{Proof of Proposition \ref{prop:non-markov}}\label{sec:prf-nonmarkov}
{ 
\begin{proof}
    Observe that under conditions described in equations (\ref{eq:uniell-1}) and (\ref{eq:minorisation}) 
    \begin{align*}
        \prob\lp (X_p,a_p)\in\Scal | \Hcal_0^{p-1} \rp> \epsilon_0\epsilon_1\Vol{\Scal}
    \end{align*}
    for any positive integer $p$.
    This implies,
    \begin{align*}
        \prob\lp (X_p,a_p)\notin\Scal | \Hcal_0^{p-1}\in\history_0^{p-1} \rp<1- \epsilon_0\epsilon_1\Vol{\Scal}.
    \end{align*}

    Using this fact recursively, we get 
    \begin{align*}
        \prob\lp (X_{p+q},a_{p+q})\notin\Scal,\dots,(X_p,a_p)\notin\Scal | \Hcal_0^{p-1} \in\history_0^{p-1}\rp<\lp 1- \epsilon_0\epsilon_1\Vol{\Scal}\rp^{q+1}
    \end{align*}
    for any $q\geq 0$.

    Now, let $p$ be when $X_{i-1},a_{i-1}\in \Scal$, for the $\tau_\star$-th time. Then,
    \begin{align*}
        \prob\lp \tau_\Scal\pow {\tau_{\star}}>q|\History_0^{p}\in\history_0^{p} \rp &=\prob\lp (X_{p+q},a_{p+q})\notin\Scal,\dots,(X_p,a_p)\notin\Scal | \Hcal_0^{p-1} \in\history_0^{p-1}\rp\\
        & <\lp 1- \epsilon_0\epsilon_1\Vol{\Scal}\rp^{q+1}.
    \end{align*}
    It now follows that 
    \begin{align*}
        \expec[ \tau_\Scal\pow {\tau_{\star}}|\History_0^{p}\in\history_0^{p} ]&\leq \sum_{q\geq 1} \prob\lp \tau_\Scal\pow {\tau_{\star}}>q|\History_0^{p}\in\history_0^{p} \rp\\
        & \leq\sum_{q\geq 1} \prob\lp (X_{p+q},a_{p+q})\notin\Scal,\dots,(X_p,a_p)\notin\Scal | \Hcal_0^{p-1} \in\history_0^{p-1}\rp\\
        & \leq \sum_{q\geq 1}\lp 1- \epsilon_0\epsilon_1\Vol{\Scal}\rp^{q}+1\\
        & \leq \frac{\epsilon_0\epsilon_1\Vol{\Scal}}{ 1- \epsilon_0\epsilon_1\Vol{\Scal}}+1.
    \end{align*}
    
    This completes the proof. 
\end{proof}
}
\subsection{Proof of Lemma \ref{lemma:mino-but-not-mixing}}\label{sec:prf-minbnmix}
\begin{proof}
    Let \(\chi=\Ibb=[0,1]\). We assume \(\{X_i\}\) are i.i.d.\ uniform random variables on \([0,1]\). We also assume \(a_0\) is uniformly distributed on \([0,1]\). For \(i\ge 1\), we define \(\{a_i\}\) independently of \(\{X_i\}\) through conditional densities \(\density_{a_i}\), where


\begin{align*}
    \density_{a_i}(l|a_0) = \begin{cases}
        1 & \ \text{ if }\ l\in[0,1] \ \text{ and }\ a_0\in [0,1/2)\\
        \frac{1}{4} & \ \text{ if }\ l\in[0,1/2) \ \text{ and }\ a_0 \in [1/2,1]\\
        \frac{7}{4} & \ \text{ if }\ l\in[1/2,1]\ \text{ and }\ a_0 \in [1/2,1]
    \end{cases}
\end{align*}
Now, by setting $\Vcal(D)=\int_{D}(1/4)\mu_{\Ibb}(dl)$, one can see that for any $A\in \Fcal_{0}^{p-1},C\subseteq\chi, D\subseteq [0,1/2)$
\begin{align*}
        \prob\lp a_p\in D|X_p\in C, A \rp & \geq \Vcal(D).
\end{align*}

However, to show that $(X_i,a_i)$ is \textbf{not} $\alpha$-mixing, we note that for any $p\geq 1$
\begin{align*}
    \prob\lp a_p\in [1/2,1]\bigcap a_0\in[1/2,1] \rp & = \frac{7}{16},
\end{align*}
and
\begin{align*}
     \prob\lp a_p\in [1/2,1]\rp\prob\lp a_0\in[1/2,1] \rp & = \frac{11}{16}.
\end{align*}
Therefore, 
\begin{align*}
    \lv  \prob\lp a_p\in [1/2,1]\bigcap a_0\in[1/2,1] \rp -  \prob\lp a_p\in [1/2,1]\rp\prob\lp a_0\in[1/2,1] \rp \rv = \frac{1}{4},
\end{align*}
which in turn implies that 
\begin{align*}
    \alpha_{i,j} = \sup_{A,B}\lv\prob\lp \History_{0}^{i}\in A\bigcap\History_{j}^{\infty}\in B \rp-\prob\lp \History_0^i\in A\rp\prob\lp \History_j^\infty\in B \rp\rv\geq \frac{1}{4}
\end{align*}
for all $1\leq i<j$. This completes the proof.

\end{proof}

\subsection{Proof of Proposition \ref{prop:conc-m1m2}}\label{sec:prf-m1m2conc}
\begin{proof}
{
  For notational clarity, we introduce two intermediate objects, $\psi(c_1,c_2)$ and $\bar{f}$, defined by

\begin{small}
\begin{align*}
    \psi(c_1,c_2) & :=\frac{1}{\sqrt{2}}\frac{\sqrt{c_2}-\sqrt{c_1}}{\sqrt{c_2+c_1}}\numberthis\label{eq:psi}\\
    \bar f(x,l,y) & := \frac{f_1(x,l,y)+f_2(x,l,y)}{2}.
\end{align*}    
\end{small}

Next, for two functions $f_1$ and $f_2$, define $Z_i$ by

\begin{small}
\begin{align*}
    Z_i(f_1,f_2) & := \psi\lp f_1(X_i,a_i,X_{i+1}),f_2(X_i,a_i,X_{i+1})\rp - \expec[\psi\lp f_1(X_i,a_i,X_{i+1}),f_2(X_i,a_i,X_{i+1})\rp \mid X_i,a_i].\numberthis\label{def:Zi}
\end{align*}
\end{small}

We can now state the lemma, whose proof is provided in Section \ref{sec:prf-bervb}:
\begin{lemma}\label{lemma:bernstein-var}
\begin{align*}
        \int \psi(f_1,f_2)^2 \density \ d\lambda_n \;\leq\; 3\Bigl[ \Hcal^2\bigl(\density ,f_2\bigr)+\Hcal^2\bigl(\density ,f_1\bigr) \Bigr].
\end{align*}
\end{lemma}

We also state the following lemma, proved by algebraic manipulations in Section \ref{sec:prf-mbeb}:
\begin{lemma}\label{prop:mb-eb}
Recall from \cref{eq:psi} that $\phi(c_1,c_2)=(\sqrt{c_2}-\sqrt{c_1})/\sqrt{2(c_1+c_2)}$. Then
\begin{align*}
    \bigl(1-\tfrac{1}{\sqrt{2}}\bigr) \Hcal^2\bigl(\density ,f_2\bigr)+ \Test\bigl(f_1,f_2\bigr) \;\leq\;  \bigl(1+\tfrac{1}{\sqrt{2}}\bigr)\Hcal^2\bigl(\density ,f_1\bigr)+\frac{1}{n}\sum_{i=0}^{n-1} Z_i\bigl(f_1,f_2\bigr).
\end{align*}
\end{lemma}

To proceed with the proof, we first adopt from \cite{sart_estimation_2014} the following iteration of Bernstein's inequality. As before, let $\{\Fcal_0^i\}_{i\geq 0}$ be a filtration and $|g_i|\leq b$ be a bounded random variable adapted to it. 
Then we have the following lemma.
\begin{lemma}~\label{lemma:bernstein-massart}
     Define the sum $\density_n:=\sum_{i=0}^n \lp g_i-\expec[g_i|\Fcal_0^i]\rp$ and $V_n := \sum_{i=0}^n\expec[g_i^2|\Fcal_0^i]$. Then 
     \begin{align}~\label{eq:bernstein-massart}
         \prob\lp \density_n\geq \frac{V_n}{2(\kappa-b)}+x\kappa \rp\leq \exp\lp-x\rp
     \end{align}
     for all $\kappa>b$, and $x>0$.
\end{lemma}


Using $Z_i$ as in \cref{def:Zi}, set $\density_n = \sum_{i=0}^{n-1} Z_i$ and 
\[
g_i = \psi\bigl(f_1(X_i,a_i,X_{i+1}),\,f_2(X_i,a_i,X_{i+1})\bigr).
\]
Then, Lemma \ref{lemma:bernstein-massart} asserts
\begin{align*}
    \prob\Bigl(\density_n\geq \frac{V_n}{2(\kappa-b)}+x\kappa\Bigr)\leq \exp(-x).\numberthis\label{eq:m1m2conc-eq1}
\end{align*}

A simple rearrangement shows $V_n$ reduces to $n\int \psi\bigl(f_1, f_2\bigr)^2 \density \,d\lambda_n$. Lemma \ref{lemma:bernstein-var} then bounds $\int \psi\bigl(f_1, f_2\bigr)^2 \density \,d\lambda_n$ by
\[
\int \psi\bigl(f_1, f_2\bigr)^2 \density \,d\lambda_n \;\leq\; 3\Bigl[\Hcal^2\bigl(\density ,f_2\bigr)+\Hcal^2\bigl(\density ,f_1\bigr)\Bigr].
\]
From \cref{eq:m1m2conc-eq1}, we obtain
\begin{align*}
    \prob\Bigl(\density_n\geq \frac{3n\bigl[\Hcal^2(\density ,f_2)+\Hcal^2(\density ,f_1)\bigr]}{2(\kappa-b)}+x\kappa\Bigr)\;\leq\;\exp(-x).
\end{align*}
Equivalently,
\begin{align*}
    \prob\Bigl(\frac{\density_n}{n}\,\geq\, \frac{3\bigl[\Hcal^2(\density ,f_2)+\Hcal^2(\density ,f_1)\bigr]}{2(\kappa-b)}+\frac{x\kappa}{n}\Bigr)\;\leq\;\exp(-x).
    \numberthis\label{eq:m1m2conc-eq2}
\end{align*}

By Lemma \ref{prop:mb-eb},
\[
     \bigl(1-\tfrac{1}{\sqrt{2}}\bigr) \Hcal^2\bigl(\density ,f_2\bigr)+ \Test\bigl(f_1,f_2\bigr) \;-\; \bigl(1+\tfrac{1}{\sqrt{2}}\bigr)\Hcal^2\bigl(\density ,f_1\bigr)\;\leq\;\frac{\density_n}{n}.
\]
Substituting this into \cref{eq:m1m2conc-eq2} yields, with probability at most $\exp\bigl(-x\bigr)$,
\begin{align*}
& \bigl(1-\tfrac{1}{\sqrt{2}}\bigr)\Hcal^2\bigl(\density ,f_2\bigr)+ \Test\bigl(f_1,f_2\bigr) \;-\; \bigl(1+\tfrac{1}{\sqrt{2}}\bigr)\Hcal^2\bigl(\density ,f_1\bigr) \\
&\quad \leq\; \frac{3\bigl[\Hcal^2(\density ,f_2)+\Hcal^2(\density ,f_1)\bigr]}{2(\kappa-b)}\;+\;\frac{x\kappa}{n}.
\end{align*}

Next, observe that $\psi\leq 1/\sqrt{2}$. We set 
\[
b = 1/\sqrt{2},\quad
x = \frac{n\bigl(pen(m_1)+pen(m_2)+\kappa\zeta\bigr)}{\kappa},\quad
\kappa = \frac{2+11\sqrt{2}}{2\sqrt{2}-2},
\]
implying $1.5\times (\kappa-b) = \bigl(1-1/\sqrt{2}\bigr)/4$. Hence, with probability at most
$\exp\Bigl(-\,n\,\frac{pen(m_1)+pen(m_2)}{\kappa}-n\,\zeta\Bigr)$,
\begin{align*}
& \bigl(1-\tfrac{1}{\sqrt{2}}\bigr)\Hcal^2\bigl(\density ,f_2\bigr)+ \Test\bigl(f_1,f_2\bigr) \;-\; \bigl(1+\tfrac{1}{\sqrt{2}}\bigr)\Hcal^2\bigl(\density ,f_1\bigr) \\
&\quad \leq\;  \frac14\Bigl(1-\tfrac{1}{\sqrt{2}}\Bigr)\Bigl[\Hcal^2\bigl(\density ,f_2\bigr)+\Hcal^2\bigl(\density ,f_1\bigr)\Bigr]\;+\;\frac{x\kappa}{n}.
\end{align*}

By rearranging terms and bounding
$\bigl(1-0.5^{0.5}\bigr)\Hcal^2\bigl(\density ,f_1\bigr)$ by
$\bigl(1+0.5^{0.5}\bigr)\Hcal^2\bigl(\density ,f_1\bigr)$, we conclude

\begin{small}
\begin{align*}
   \frac34\Bigl(1-\tfrac{1}{\sqrt{2}}\Bigr)\Hcal^2\bigl(\density ,f_2\bigr) \;+\; \Test\bigl(f_1,f_2\bigr) 
   \;\leq\; \frac54\Bigl(1+\tfrac{1}{\sqrt{2}}\Bigr)\Hcal^2\bigl(\density ,f_1\bigr)\;+\;pen(m_1)\;+\;pen(m_2)\;+\;\zeta. 
\end{align*}
\end{small}

This completes the proof.
}
\end{proof}

\subsection{Proof of Proposition \ref{prop:partition}}\label{sec:prf-partition}
\begin{proof}
    \textbf{1.} That $\Mcal_\lcal\subset \Mcal_{\lcal+1}$ is obvious by construction. We prove $|m|\leq {2}^{\lcal(2d_1+d_2)}$ by induction. It obviously is true for $\lcal = 0$. Now let it be true for a given value $\lcal$. Let $m\in\Mcal_{\lcal+1}$ be an element of $\Mcal_{\lcal+1}$. From construction, either $m\in\Mcal_\lcal$, or $m\in \bigcup_{m}\bigcup_{k}\Scal(m, k)$ where $\Scal(m, k)$ is as in Definition \ref{def:dyadic-cuts}. If $m\in\Mcal_{\lcal+1}$ then $|m|\leq {2}^{\lcal(2d_1+d_2)}$ and we have proved the induction step. If $m\in \bigcup_{m}\bigcup_{k}\Scal(m, k)$, then $|m|\leq  {2}^{(\lcal+1)(2d_1+d_2)}-1$ by construction the induction step is satisfied. Finally, we observe that 
    \[
    \sum_{m\in\Mcal_\infty}e^{-|m|} = \sum_{\substack{m\in \Mcal_\infty\\|m|= {2}^{\lcal(2d_1+d_2)}}} e^{-|m|} =  \sum_{\substack{m\in \Mcal_\infty\\|m|= {2}^{\lcal(2d_1+d_2)}}} e^{-{2}^{\lcal(2d_1+d_2)}}\leq\sum_{\lcal\geq 0}{2}^{\lcal(2d_1+d_2)} e^{-{2}^{\lcal(2d_1+d_2)}}\leq \frac{e}{e-1}
    \]
    That $\frac{e}{e-1}\leq 15$ is obvious.

    \textbf{2.} is an easy observation from construction. We prove \textbf{3.} using induction. It holds trivially for $\lcal = 0$. Let the statement be true for a given $\lcal$. Now, let $m_{\lcal+1}$ be an element of $\Mcal_{\lcal+1}$. As previously, observe that either $\exists m_{\lcal}\in \Mcal_{\lcal+1}\backslash\Mcal_\lcal$ such that $K\in m_{\lcal}$, or by Definition \ref{def:dyadic-cuts}, $K\in S(m,k)$ for some pair $m,k$. In the former case, $\exists \lc K_1,\dots,K_\lcal \rc$ such that $K\subset K_i$. We set $K_{\lcal+1}=K_\lcal$, completing the proof. 
    
    The later case can again be subdivided into two distinct cases. Either $K\in m\backslash k$, in which case, the proof proceeds similarly to the previous step, or $K\in \{k_1,k_2,\dots,k_{2^{d_2+2d_1}}\}$, in which case, we set $K_{\lcal+1}=k$ and the proof is complete.

    \textbf{4.} We first recall the definition of $m\vee m'$ from \cref{eq:vee-def}
     \[
    m\vee m' = \bigcup_{K'\in m'}\lc m\vee K' \rc \text{ where } m\vee K' := \lc K'\cap K: K\in m,K'\cap K\neq \text{\O}  \rc.
    \]
    For any two dyadic partitions $m$ and $m'$ let 
    \[
    \Scal_{agree}(m,m'):=\lc K: K\in m \text{ and } K\in m'\rc.
    \] 
    Observe from Definition \ref{def:dyadic-cuts} that if $K'\in m'$ and $K'\notin m$, the it is constructed by dyadically partitioning some element of $m$. Let that element be $K$, and we have $K\cap K'=K'$. Observe that if there exists another $K^\star\in m$ such that $K^\star\cap K'=K'$, then either $K\subset K^\star$ or $K^\star\subseteq K$. To avoid overcounting, we always let $K$ be the smallest such set and write following definition.
    \[
    \Scal_{disagree}(K,m'):=\lc K': K'\in m',  K'\notin m \text{ and } K'\subset K \text{ for some smallest } K\in m \rc.
    \]
    $\Scal_{disagree}(K',m)$ can be defined similarly. Since $ m\vee m'$ is the set of non-empty intersections of $m'$ with the elements of $m$, it follows that 
    \begin{align*}
    |m\vee m'| & = |\Scal_{agree}(m,m')| + \lv\bigcup_{K\in m\cap\Scal_{agree}(m,m')^c}\Scal_{disagree}(K,m')\rv\\
    & \quad +\lv\bigcup_{K'\in m'\cap\Scal_{agree}(m,m')^c}\Scal_{disagree}(K',m)\rv
    \end{align*}
    We observe the following facts
    \begin{enumerate}
        \item $|\Scal_{agree}(m,m')|\leq |m|+|m'|$,
        \item  $|\cup_{K\in m\cap\Scal_{agree}(m,m')^c}\Scal_{disagree}(K,m')|\leq |m'|$,
        \item $|\cup_{K'\in m'\cap\Scal_{agree}(m,m')^c}\Scal_{disagree}(K',m)|\leq |m|$.
    \end{enumerate}
    This gives us the required result.
\end{proof}

\subsection{Proposition \ref{prop:detls2} and proof of its upper bound}\label{sec:prf-detls2lbub}
\begin{proposition}\label{prop:detls2}    
    Assume the conditions of Theorem \ref{thm:detlos-2}, and let $\tilde\Scal_\star:=\argmax_{\Scal\in m_{ref}\pow 2}T(\Scal)$, $\lcal\leq n$, and $d_1\geq 12$ . Then, 
\begin{enumerate}
    \item if 
    \[
    \frac{n}{(\log n)^3}\geq \constant\Constant_p^{-1}T(\Scal_\star)^2\lp \Constant_\Delta\rho_\star(\Scal_\star)+\frac{1}{T(\Scal_\star)}\rp\log\lp T\lp\tilde \Scal_\star\rp\rp.\numberthis\label{eq:thmdetls2-eq1-copy}
    \]
    Then, $\Rcal(n)\leq 4/n$
    \item if  
    \[
     n\leq \Constant_p^{-1}T(\Scal_\star)^2\lp \Constant_\Delta\rho_\star(\Scal_\star)+\frac{1}{T(\Scal_\star)}\rp,
    \]
    then $\Rcal(n)>1/2$, and 
    there exists a controlled Markov chain such that there exists no estimator $\hat s$ satisfying
    \[
    \expec[h_n^2(\density ,\hat \density)]\leq \frac{1}{2(1+\pi^2)}.
    \]
    \end{enumerate}
\end{proposition}
Broadly, our strategy is to pose the question of tightness of $\Rcal(n)$ in terms of sample complexity, and then follow the usual techniques from \cite{tsybakov_introduction_2009} to show minimaxity.

    We first establish a few facts required for the proof:
    \paragraph{Fact 1.} With $N_\Scal:=\sum_{i=1}^n\indicator_{[(X_i,a_i)\in\Scal]}$, $\expec[N_\Scal]\geq \frac{n}{2T(\Scal)}$.
    \begin{proof}[Proof of Fact 1.]
    Recall from Lemma \ref{lemma:KAC-lower} that,
    \[
    \expec[N_\Scal]\geq \frac{n}{T(\Scal)}-1
    \]
    Since $n\geq 2T\lp\tilde\Scal_\star\rp$, it follows from the definition of $\tilde \Scal_\star$ that $n\geq 2T(\Scal)$. The rest follows by observing that  for $T(\Scal)\geq 1$, $n/T(\Scal)-1\geq n/(2T(\Scal))$.
    \end{proof}
    \paragraph{Fact 2.} $T\lp\tilde \Scal_\star\rp\geq 4^{ld-1}$.
    \begin{proof}[Proof of Fact 2.]
    This fact is proved using Fact 1. Summing over $\Scal\in m_{ref}\pow 2$ on both sides of $\expec[N_\Scal]\geq \frac{n}{2T(\Scal)}$, we get that,
    \begin{align*}
       \underbrace{\sum_{\Scal\in m_{ref}\pow 2} \expec[N_\Scal]}_{=:\text{LHS}}\geq \sum_{\Scal\in m_{ref}\pow 2}\frac{n}{2T(\Scal)}\geq \sum_{\Scal\in m_{ref}\pow 2}\frac{n}{2T\lp\tilde \Scal_\star\rp}=\underbrace{2^{\lcal(d_1+d_2)} \frac{n}{2T\lp\tilde \Scal_\star\rp}}_{=:\text{RHS}}. 
    \end{align*}
    Observing that
    \begin{align*}
        \text{LHS} = \expec\lb \sum_{\Scal\in m_{ref}\pow 2}  N_\Scal\rb = \expec[n]=n,
    \end{align*}
    we can cancel $n$ from both LHS and RHS to get $T\lp\tilde \Scal_\star\rp>2^{\lcal(d_1+d_2)-1}$. The rest now follows.
    \end{proof}
    \paragraph{Fact 3.}
    \[\frac{ \frac{\Constant_pn}{4T(\Scal_\star)^2}}{4\Constant_\Delta\rho_\star(\Scal_\star) +\frac{4+(\log n)^2}{2T(\Scal_\star)}}\geq \frac{\frac{\Constant_p n}{T(\Scal_\star)^2}}{(\log n)^2\lp \Constant_\Delta\rho_\star(\Scal_\star)+\frac{1}{T(\Scal_\star)}\rp}\]
    \begin{proof}[Proof of Fact 3.]
        We begin by observing that 
        \begin{align*}
            4\Constant_\Delta\rho_\star(\Scal_\star) +\frac{4+(\log n)^2}{2T(\Scal_\star)} & = \frac{(\log n)^2}{2}\lp \frac{8}{(\log n)^2}\Constant_\Delta\rho_\star(\Scal_\star) +\frac{\frac{8}{(\log n)^2}+1}{T(\Scal_\star)}\rp\\
            & \leq \frac{(\log n)^2}{2}\lp \Constant_\Delta\rho_\star(\Scal_\star) +\frac{2}{T(\Scal_\star)}\rp\\
            & \leq (\log n)^2 \lp \Constant_\Delta\rho_\star(\Scal_\star) +\frac{1}{T(\Scal_\star)}\rp,
        \end{align*}
        where the first inequality follows from the fact that $8/(\log n)^2\leq 1$. The rest of the proof now follows.
    \end{proof}
    \begin{proof}[Proof of the Upper bound of Proposition \ref{prop:detls2}]
        We first prove the first part. Let,
        \[
        \frac{n}{(\log n)^3}\geq \constant\Constant_p^{-1}T(\Scal_\star)^2\lp \Constant_\Delta\rho_\star(\Scal_\star)+\frac{1}{T(\Scal_\star)}\rp\log\lp T\lp\tilde \Scal_\star\rp\rp.
        \]
        Then,
        \begin{align*}
                    \frac{n}{(\log n)^2}& \geq \constant\Constant_p^{-1}T(\Scal_\star)^2\lp \Constant_\Delta\rho_\star(\Scal_\star)+\frac{1}{T(\Scal_\star)}\rp\log\lp T\lp\tilde \Scal_\star\rp\rp\log n\\
                    & \geq  \constant\Constant_p^{-1}T(\Scal_\star)^2\lp \Constant_\Delta\rho_\star(\Scal_\star)+\frac{1}{T(\Scal_\star)}\rp\lp\log\lp T\lp\tilde \Scal_\star\rp\rp+\log n\rp\\
                    & =  \constant\Constant_p^{-1}T(\Scal_\star)^2\lp \Constant_\Delta\rho_\star(\Scal_\star)+\frac{1}{T(\Scal_\star)}\rp\log\lp nT\lp\tilde \Scal_\star\rp\rp
        \end{align*}
        This implies that 
        \begin{align*}
            \frac{\frac{\Constant_p n}{T(\Scal_\star)^2}}{(\log n)^2\lp \Constant_\Delta\rho_\star(\Scal_\star)+\frac{1}{T(\Scal_\star)}\rp}\geq \log\lp nT\lp\tilde \Scal_\star\rp\rp.
        \end{align*}
        Using Fact 3, we get
        \begin{align*}
             \frac{ \frac{\Constant_pn}{4T(\Scal_\star)^2}}{4\Constant_\Delta\rho_\star(\Scal_\star) +\frac{4+(\log n)^2}{2T(\Scal_\star)}}\geq \log\lp nT\lp\tilde \Scal_\star\rp\rp.
        \end{align*}
        Using Fact 2, we get
        \begin{align*}
            \frac{ \frac{\Constant_pn}{4T(\Scal_\star)^2}}{4\Constant_\Delta\rho_\star(\Scal_\star) +\frac{4+(\log n)^2}{2T(\Scal_\star)}}\geq \log\lp n2^{\lcal(d_1+d_2)-1}\rp.
        \end{align*}
        Now taking negative sign on both sides and exponentiating, we get
        \[
        2^{\lcal(d_1+d_2)}\exp\lp -\frac{ \frac{\Constant_pn}{4T(\Scal_\star)^2}}{4\Constant_\Delta\rho_\star(\Scal_\star) +\frac{4+(\log n)^2}{2T(\Scal_\star)}} \rp\leq \frac{4}{n}
        \]
        Now with $\Rcal(n)$ as defined in Theorem \ref{thm:detlos-2}, we get $\Rcal(n)\leq 4/n$ which completes the proof.
    \end{proof}
    
\subsection{Proof of the lower bound of Proposition \ref{prop:detls2}}\label{sec:prf-detls2lblb}

\paragraph{Assoud's Reduction} We begin with observing the simple fact that 
\begin{align*}
    \expec[h_n^2(\density ,\hat \density)]=\int_{\eps^2\in(0,1)}\prob( h_n^2(\density ,\hat \density)>\eps^2)d\eps^2.
\end{align*}
So it is enough to show that without $n$ sufficiently large and for any $\eps\in (0,1/32)$
\[
\prob( h_n^2(\density ,\hat \density)>\eps^2)>\frac{1}{2(1+\pi^2)}
\]
for any estimator $\hat \density$ of $\density $. 

We follow the recipe of Assoud's reduction scheme \cite[Chapter 2]{tsybakov_introduction_2009}. Without losing generality let $\chi\times\Ibb=[0,1]^{d_1+d_2}$. Let $\Dcal$ be ``some" class of controlled Markov chains (specified below). We use $\Pcal$ to denote an element of $\Dcal$. One can write $\Pcal=(\density,\{\pcal\pow i\}_{i\geq 0})$, where $\density$ is the transition density and $\pcal\pow i$ is the distribution of the control $a_i$ at time point $i$ given the previous history. Let $\hat \density$ be any estimator of $\density$. We will show that, as long as 
\[
    n \geq \constant\Constant_p^{-1}T(\Scal_\star)^2\lp \Constant_\Delta\rho_\star(\Scal_\star)+\frac{1}{T(\Scal_\star)}\rp,
\]
we have
\[
\inf_{\hat \density}\sup_{\Pcal\in \Dcal}\prob\lp d_2 h_n^2(\hat \density,\density) >\eps^2 \rp >\frac{1}{2(1+\pi^2)}.   
\] 
\paragraph{Construction of $\Dcal$}
Let $d_1$ be an even integer divisible by $3$ greater than $12$. We simply let $\pcal\pow i$ to be the uniform distribution on $[0,1]^{d_2}$. 
Now we carefully construct the transition densities. 
Let $\iota$ be a known real number between $1/32$ and $31/64$ and furthermore, let $\Ccal=\{k_1\pow \chi,\dots,k_{d_1}\pow \chi\}$ and $\Ical=\{k_1\pow \Ibb,\dots,k_{d_2}\pow \Ibb\}$ be uniform partitions of $\chi$ and $\Ibb$ into $d_1$ and $d_2$ distinct cubes respectively. 
Let each integer $l'$ such that $k_{l'}\pow \Ibb\in\Ical$, let $\xi\pow{p}=(\xi_1\pow{l'},\dots,\xi_{d_1/3}\pow{l'})$ be some vector in $\lc0,1\rc^{d_1/3}$ such that that $\xi\pow{l'}\neq (0,\dots,0)$ for at least some $l'$.
We consider $\density(x,l,y)$ to be piecewise constant functions on the partition $\Ccal\times\Ical\times\Ccal$. 
In other words, $\density(x,l,y)=M_{i,j}\pow {l'}$ for all $x\in k_i\pow\chi,y\in k_j\pow\chi,l\in k_{l'}\pow \Ibb$. 
We can represent $M_{i,j}\pow{l'}$ by the following matrix which depends only on $\iota$ and $\xi\pow {l'}$
\begin{align*}
M_{\iota,\xi^{(l')}}^{(l')}=    d_1\times\begin{bmatrix}
\boldsymbol{C}_{\iota} & \boldsymbol{R}_{\xi^{(l')}} \\
\boldsymbol{J}_{\iota} & \boldsymbol{L}_{\iota}
\end{bmatrix},\numberthis\label{eq:mm-ex2eq2}
\end{align*}
where the blocks
$\boldsymbol{C}_{\iota} \in \Rbb^{d_1/3\times d_1/3}$,
$  \boldsymbol{L}_{\iota}  \in\Rbb^{2d_1/3\times2d_1/3}$,
$\boldsymbol{J}_{\iota}\in \Rbb^{2d_1/3\times d_1/3}$,
and
$\boldsymbol{R}_{\xi^{(l')}} \in\Rbb^{d_1/3\times 2d_1/3}$
are given by
\begin{small}
\begin{align*}
&\boldsymbol{R}_{\xi^{(l')}}  = \frac{1}{2}
\begin{bmatrix}
1 +  \xi^{(l')}_1\eps-2\iota & 1 -  \xi^{(l')}_1\eps-2\iota  & \frac{3\iota}{d_1-3} & \frac{3\iota}{d_1-3} & \hdots &  \frac{3\iota}{d_1-3} \\
\frac{3\iota}{d_1-3} & \frac{3\iota}{d_1-3} & 1 + \xi^{(l')}_2 \eps-2\iota & 1-  \xi^{(l')}_2 \eps -2\iota  & \hdots & \frac{3\iota}{d_1-3} \\
\vdots & \vdots & \vdots  & \vdots & \vdots & \vdots \\
\frac{3\iota}{d_1-3} & \hdots & \hdots & \hdots & 1 +  \xi^{(l')}_{d_1/3} \eps -2\iota & 1-  \xi^{(l')}_{d_1/3} \eps-2\iota\\
\end{bmatrix},
\end{align*}
\end{small}
$\boldsymbol{L}_{\iota}$ is a matrix with every element equal to $3(1-\iota)/2d_1$, and, $\boldsymbol{C}_\iota$  and $\boldsymbol{J}_\iota$ are matrices with every element equal to $3\iota/d_1$. 
It can be verified by integrating that for each $l$ and $x$, $\density(x,l,\cdot)$ is a valid transition density.
\paragraph{Some preliminary results} Here, we derive some properties of CMC's that are elements of $\Dcal$ in the form of the following two results.
\begin{lemma}~\label{lemma:stationary-dist}
For each $l\in k_{l'}\pow \Ibb$, stationary distribution $\Pi\pow{l,\iota}(\cdot)$ of a Markov chain with transition density $\density(\cdot,l,\cdot)$ given in the previous construction is a piecewise constant function on $\Ccal$. 
\begin{align*}
    \Pi\pow{l,\iota}(x) = \begin{cases}
        \iota\ \forall\ x\in \bigcup_{i=1}^{d_1/3} k_i\pow\chi\\
        \frac{\iota(1-\xi_1\pow{l'}\eps-\iota)}{2}+\frac{d_1\iota^2}{2(d_1-3)}+\frac{(1-\iota)^2}{2} \ \forall\ x\in k_{d_1/3+1}\pow \chi\\
         \frac{\iota(1+\xi_1\pow{l'}\eps-\iota)}{2}+\frac{d_1\iota^2}{2(d_1-3)}+\frac{(1-\iota)^2}{2} \ \forall\ x\in k_{d_1/3+2}\pow \chi\\
           \vdots\\
        \frac{\iota(1-\xi_{d_1/3}\pow{l'}\eps-\iota)}{2}+\frac{d_1\iota^2}{2(d_1-3)}+\frac{(1-\iota)^2}{2} \ \forall\ x\in k_{d_1-1}\pow \chi\\
        \frac{\iota(1+\xi_{d_1/3}\pow{l'}\eps-\iota)}{2}+\frac{d_1\iota^2}{2(d_1-3)}+\frac{(1-\iota)^2}{2}  \ \forall\ x\in k_{d_1}\pow \chi.
    \end{cases}\numberthis\label{eq:mm-exstdb}
\end{align*}
\end{lemma}
\noindent The proof follows by verifying $\int\Pi\pow{l,\iota}(y)\density(x,l,y)dy=\Pi\pow{l,\iota}(x)$ and is straightforward. Therefore, we omit it.
\begin{remark}
    Let $(X_i,a_i)$ be a controlled Markov chain with transition density $\density$ and the distribution over controls $\pcal\pow i$ such that $(\density,\{\pcal\pow i\})\in \Dcal$. Since $\pcal\pow i$ is uniform and independent of the history, one can easily see in the light of the previous lemma that the paired process $(X_i,a_i)$ forms a Markov chain with stationary distribution $\Pi(x,l)= \Pi\pow{l,\iota}(x)$ for all $x\in\chi$ and $\l\in\Ibb$.   
\end{remark}
\begin{proposition}\label{prop:mm-ex2prop1}
Let $\indexeddata$ be a sample from a CMC which is an element of $\Dcal$ with initial distribution $\Pi(x,l)= \Pi\pow{l,\iota}(x)$. Then, 

\begin{enumerate}
    \item For any $\Scal\subset k_i\pow\chi\times k_j\pow \Ibb$ and any $i\in\{1,\dots,d_1/3\}$, the expected return time $T$ as defined in \cref{def:return-time} satisfies 
\[
T(\Scal) = \frac{4}{5\iota^2\Vol{\Scal}}
\]
\item The $\alpha$-coefficients of this controlled Markov chain satisfy $\alpha_{i,j}\leq (1-\iota)^{j-i-1}$. In particular, $\constant_p$ as written in Assumption \ref{assume:alpha-mix} is only depends upon $\iota$.
 \item Let $\Scal_{i,j} = k_i\pow\chi\times k_j\pow \Ibb $ such that $ i\in\{1,\dots,d_1/3\}$. Then, $\rho_\star(\Scal_{i,j})$ (as defined in Theorem \ref{thm:detlos-2}) satisfies
 \begin{align*}
     \rho_\star(\Scal_{i,j}) < \frac{9(1-\iota)}{2d_1d_2}. 
 \end{align*}
\end{enumerate}
\end{proposition}

\paragraph{Simplification of the Sample Complexity} We can now substitute upper bounds derived from Proposition \ref{prop:mm-ex2prop1} in the right hand side of \cref{eq:thmdetls2-eq1}. For ease of perusal, we first rewrite the expression the right hand side of \cref{eq:thmdetls2-eq1} below
\[
 \Constant_p^{-1}T(\Scal_\star)^2\lp \Constant_\Delta\rho_\star(\Scal_\star)+\frac{1}{T(\Scal_\star)}\rp.
\]
We now note the following facts.
\begin{enumerate}
    \item $\Constant_p$ only depends upon $\constant_p$ from Assumption \ref{assume:alpha-mix}, which in turn only depends upon $\iota$ for the class of CMC's we consider (by Proposition \ref{prop:mm-ex2prop1} part 2).
    \item $\Constant_\Delta$ only depends upon $\iota$.
    \item Since $k_i\pow\chi\times k_j\pow \Ibb$ create $d_1d_2$ uniform cubes of $\chi\times \Ibb$, for any $\Scal_{i,j} = k_i\pow\chi\times k_j\pow \Ibb$, $\Vol{\Scal_{i,j}}=(d_1d_2)^{-1}$. 
\end{enumerate}
Using the previous facts, and substituting the bounds from Proposition \ref{prop:mm-ex2prop1} into the right hand side of \cref{eq:thmdetls2-eq1} we get
\[
 \Constant_p^{-1}T(\Scal_\star)^2\lp \Constant_p^{-1}\rho_\star(\Scal_\star)+\frac{1}{T(\Scal_\star)}\rp\leq \Constant_\iota\lp \Constant_\Delta\frac{16d^4}{25\iota^4}\times\frac{9(1-\iota)}{2d^2} + \frac{4d^2}{5\iota^2}.\rp\leq \Constant_\iota d_1d_2,
\]
where $\Constant_\iota$ is an appropriately large constant depending only upon $\iota$.
All we need to show now is that unless $n\geq \Constant_\iota'd_1d_2$ for some constant $\Constant_\iota'$, there exists no estimator $\hat \density$ such that 
\[
\prob\lp d_2h_n^2(\density,\hat \density)>\eps^2\rp\leq \frac{1}{1+\pi^2}.
\]
\paragraph{Separation of $h_n^2(\cdot,\cdot)$} Recall from the construction that  $\chi=[0,1]^{d_1}$ and $\Ibb=[0,1]^{d_2}$. Furthermore, $\iota$ is known, and for all $l\in k_{j}\pow\Ibb,j\in\{1,\dots,d_2\}$, the only unknown terms in the density $\density(x,l,y)$ are $\{\xi_1\pow{j},\xi_2\pow{j},\dots,\xi_{d_1/3}\pow{j}\}$. Therefore, we only need to estimate $d_1d_2/3$ many $0$'s and $1$'s. For ease of notation, we will use $\xi$ to denote this vector of $d_1d_2/3$ many terms. To be precise 
\[
\xi = \{\xi_1\pow 1,\dots, \xi_{d_1/3}\pow 1,\dots, \xi_1\pow {d_2},\dots, \xi_{d_1/3}\pow {d_2}\}
\]

Let $\density\pow\xi$ to be the corresponding estimate of the density. Now let $\Xi$ to be another $d_1d_2/3$ dimensional vector of $0$'s and $1$'s with corresponding density $\density\pow\Xi$ such that 
\[
\xi_1\pow l\neq\Xi_1\pow l \numberthis\label{eq:non-equalitycondition} 
\]
for all $l\in \{ 1,\dots,d \}$
Now, we decompose $h_n^2$. We write 
\begin{align*}
    h_n^2(s\pow \xi,s\pow \Xi) & =   \int_{x,l,y\in [0,1]^{2d_1+d_2}} \lp \sqrt{\density\pow\xi(x,l,y)}-\sqrt{\density\pow\Xi(x,l,y)}  \rp^2\mu_\chi(dy)\nu_n(dx,dl)\\
    & > \int_{x\in[0,1]^{d_1}}\sum_{j\in \{1,\dots,d\}}\int_{l\in k_j\pow\Ibb}\underbrace{\int_{y\in k_1\pow\chi}\lp \sqrt{\density\pow\xi(x,l,y)}-\sqrt{\density\pow\Xi(x,l,y)}  \rp^2\mu_\chi(dy)}_{=:A}\nu_n(dx,dl).\numberthis\label{eq:mm-ex2eq4}
\end{align*}

We first carefully analyse the term $A$ in the previous expression. 
\begin{align*}
    \int_{y\in k_1\pow\chi}&\lp \sqrt{\density\pow\xi(x,l,y)}-\sqrt{ \density\pow\Xi(x,l,y)}  \rp^2\mu_\chi(dy)\\
    &  = \frac{1}{d_1} \lp \sqrt{d_1(1 +  \xi^{(1)}_1\eps-2\iota)/2}-\sqrt{d_1(1 +  \Xi^{(1)}_1\eps-2\iota)/2} \rp^2\\
    & \qquad  + \frac{1}{d_1} \lp \sqrt{d_1(1 -  \xi^{(1)}_1\eps-2\iota)/2}-\sqrt{d_1(1 -  \Xi^{(1)}_1\eps-2\iota)/2} \rp^2.\numberthis\label{eq:mm-ex2eq3}
\end{align*}
Note the two following facts: 
\begin{itemize}
    \item[Fact 1.] $\lp \sqrt{d_1(1 +  \xi^{(1)}_1\eps-2\iota)/2}-\sqrt{d_1(1 +  \Xi^{(1)}_1\eps-2\iota)/2} \rp^2>\frac{d_1\eps^2}{4}$.

    To show this fact, we write,   
    \begin{align*}
    & \lp \sqrt{d_1(1 +  \xi^{(1)}_1\eps-2\iota)/2}-\sqrt{d_1(1 +  \Xi^{(1)}_1\eps-2\iota)/2} \rp^2\\
    &\qquad = \frac{d_1}{2} \lp \sqrt{1 +  \xi^{(1)}_1\eps-2\iota)}-\sqrt{1 +  \Xi^{(1)}_1\eps-2\iota)} \rp^2\\
    &\qquad = \frac{d_1\eps^2(\Xi_j\pow1-\xi_1\pow 1)^2}{2\lp \sqrt{(1 +  \xi^{(1)}_1\eps-2\iota)}+\sqrt{(1 +  \Xi^{(1)}_j\eps-2\iota)} \rp^2}\\
    &\qquad =\frac{d_1\eps^2}{2\lp \sqrt{(1 +  \xi^{(1)}_1\eps-2\iota)}+\sqrt{(1 +  \Xi^{(1)}_1\eps-2\iota)} \rp^2}\\
    &\qquad  >\frac{d_1\eps^2}{4},
\end{align*}
where the last line follows by the trivial inequality $\lp \sqrt{(1-2\iota)}+\sqrt{(1 +\eps-2\iota)} \rp^2<2$ which holds for our admissible range of $\eps$ and $\iota$.
\item[Fact 2.] Similarly to Fact 1,
\begin{small}
    \begin{align*}
    \lp \sqrt{d_1(1 -  \xi^{(1)}_1\eps-2\iota)/2}-\sqrt{d_1(1 -  \Xi^{(1)}_1\eps-2\iota)/2} \rp^2 & >\frac{d_1\eps^2}{4},
\end{align*}
\end{small}
\end{itemize}
Substituting this lower bound into the right hand side of \cref{eq:mm-ex2eq3} we get $A>d_1\eps^2/24$,
Substituting this lower bound of $A$ into the right hand side of \cref{eq:mm-ex2eq4} we get 
\begin{align*}
     d_2h_n^2(\density\pow \xi,\density\pow \Xi) >d_2\int_{x\in[0,1]^{d_1}}\sum_{j\in \{1,\dots,d\}}\int_{l\in k_1\pow\Ibb}A\,\nu_n(dx,dl)\geq \sum_{j\in \{1,\dots,d\}}\int_{x\in[0,1]^{d_1}}A\,\nu_n(dx)=\frac{d_2\eps^2}{24}.
\end{align*}
Let $\hat \density$ be any arbitrary estimate of $\density$ and let $\Xi_\star\in \{0,1\}^{d_1d_2/3}$ such that $\Xi_\star=\argmin_{\Xi} h_n^2(\hat \density,\density\pow \Xi)$. For any $\Xi_0\neq \Xi_\star$ satisfying \cref{eq:non-equalitycondition} 
\begin{align*}
    \frac{d_2\eps^2}{24}<d_2h_n^2(\density\pow \Xi_0,\density\pow {\Xi_\star})\leq d_2h_n^2(\density\pow \Xi_0,\hat \density)+d_2h_n^2(\hat \density,\density\pow {\Xi_\star})\leq 2 d_2h_n^2(\density\pow \Xi_0,\hat \density)
\end{align*}
Therefore, 
    \[ \underbrace{\{ \Xi_0: \text{$\Xi_0\neq\Xi_\star$}\}}_{=:\Ebb}\subseteq\{ h_n^2(\density\pow \Xi_0,\hat \density)>\eps^2/48 \}.\numberthis\label{eq:mm-ex2eq5}
\]
\paragraph{Lower Bounds on Touring Time} One can see that for any random variable $\Tbb$ and a given number of samples $n$,

\[
\prob(h_n^2(\density,\hat \density)>\eps^2)>\underbrace{\prob(h_n^2(\density,\hat \density)>\eps^2|\Tbb>n)}_{\text{Probability of Error}}\prob\lp \Tbb>n\rp\numberthis\label{eq:mm-errprb}
\]
We define $\Tbb$ to be the first time all of the sets $k_{i}\pow \chi \times k_{j}\pow \Ibb, i\in\{1,\dots,d_1/3\}$ are visited.
That is,
\[
\Tbb=\min\lc p\geq 0: \bigcap_{i\in\{1,\dots,d_1/3\}}\lc\bigcup_{q=0}^p\lc (X_q,a_q)\in k_{i}\pow \chi \times k_{j}\pow \Ibb\rc\rc \neq \text{\O}\rc.
\]
The following lemma establishes the lower bound on $\Tbb$. Its proof is given in Section \ref{sec:prf-cvrtm}. 
\begin{lemma}\label{lemma:TTLb}
    If $n<d_1d_2/(6\iota)\log(d_1d_2/3)$ then, $\prob\lp\Tbb>n\rp\geq (1+\pi^2)^{-1}$.
\end{lemma}

We now have all the tools to derive the lower bound. 
\paragraph{Lower Bound on the Probability of Error} Throughout this part, we will assume that $n<d_1d_2/(6\iota)\log(d_1d_2/3)$, so that $\prob\lp\Tbb>n\rp\geq (1+\pi^2)^{-1}$. Using \cref{eq:mm-ex2eq4} and Lemma \ref{lemma:TTLb} we get, 
\begin{align*}
    \prob(h_n^2(\density,\hat \density) >\eps^2|\Tbb>n) \prob(\Tbb>n) & >\prob(\ \Ebb\ |\ \Tbb>n)\prob(\Tbb>n)\\
    & >\frac{1}{1+\pi^2}\prob(\ \Ebb\ |\ \Tbb>n)
\end{align*}
Now, if $\Tbb>n$, there exists $i_0,j_0$ such that $\sum_{i=1}^n\indicator_{\lb(X_i,a_i)\in k_{i_0}\pow\chi\times k_{j_0}\pow\Ibb\rb}=0$. That is $(X_i,a_i)$ never visits the set $k_{i_0}\pow\chi\times k_{j_0}\pow\Ibb$ during the first $n$ time points. Therefore, for any $(x,y)\in k_{i_0}\pow\chi\times k_{j_0}\pow\Ibb$ the best estimate of $\density(x,l,y)$ is to choose uniformly over all possible values of $\xi_1\pow{j_0}$. Since $\{0,1\}$ are the only two possibilities, 
\[
\prob(\ \Ebb\ |\ \Tbb>n)=\frac{1}{2}.
\]
Therefore,
\[
\prob(h_n^2(\density,\hat \density) >\eps^2|\Tbb>n) \prob(\Tbb>n) >\frac{1}{2(1+\pi^2)}.
\]
The rest of the proof now follows.

\subsection{Proof of the upper bound in Lemma \ref{lemma:g-ublb} }\label{sec:prf-gublb}

\begin{proof}
    We only prove need to prove
    \begin{align*}
           & \sup_{m'\in \Mcal_\lcal} \lb \frac{3}{4}\lp1-\frac{1}{\sqrt 2}\rp\Hcal^2(\hat \density_m,\hat \density_{m'})+T(\hat \density_m,\hat \density_{m'}) -pen(m')\rb+pen(m) \leq \gamma(m)  \\
            &  \gamma(m) \leq \sup_{\substack{f\in \density_{m'}\\m'\in \Mcal_\lcal}} \lb \frac{3}{4}\lp1-\frac{1}{\sqrt 2}\rp\Hcal^2(\hat \density_m,f)+T(\hat \density_m,f) -pen(m')\rb+2pen(m),
    \end{align*}
    and the rest follows.
    The main objective of this proof is to construct a suitable set which allows us to exchange the order of the summation and the supremum in \cref{def:gamma}. Let $\hat \density_m$ be the set of all piecewise constant functions on $m$ whose values matches with ``some" histogram. Formally,
    \begin{align*}
        \hat \density_m = \lc \sum_{K\in m} \hat \density_{m_K}\indicator_K, \ \forall K\in m, m_K\in \Mcal_\lcal \rc. 
    \end{align*}
    Obviously, for every $K\in m$ there are multiple functions $\hat f\in \hat \density_m$ which agree with $\hat \density_m$ on $K$. The following procedure selects the coarsest one. For any function $\hat f\in \hat \density_m$, let $m_K(\hat f)$ be such that
    \begin{align*}
        m_K(\hat f):=\argmin_{m'\in \Mcal_\lcal} \lc |m'\vee K |, \hat f\indicator_K = \hat \density_{m'}\indicator_K \rc.
    \end{align*}
    and set the partition $m(\hat f)=\bigcup_{K\in m} m_K(\hat f)$.
    We observe that 
    \begin{align*}
        \gamma (m) & = \sum_{K\in m} \sup_{m'\in \Mcal_\lcal}  \lb \frac{3}{4}\lp1-\frac{1}{\sqrt 2}\rp\Hcal^2(\hat \density_m\indicator_K,\hat \density_{m'}\indicator_K) +T(\hat \density_m\indicator_K,\hat \density_{m'}\indicator_K) -pen(m'\vee K)\rb +2pen (m)\\
        & = \sup_{\hat f\in \hat \density_m}  \lb \frac{3}{4}\lp1-\frac{1}{\sqrt 2}\rp\Hcal^2(\hat \density_m\indicator_K,\hat f\indicator_K) +T(\hat \density_m\indicator_K,\hat f\indicator_K) -pen(m')\rb +2pen (m)
    \end{align*}
    Furthermore, it follows by construction that if $\hat f\in\hat \density_m$, then $\hat f\in  \density_{m(\hat f)}$. Therefore,
    \begin{align*}
        \gamma(m) \leq \sup_{\substack{f\in \density_{m'}\\m'\in \Mcal_\lcal}} \lb \frac{3}{4}\lp1-\frac{1}{\sqrt 2}\rp\Hcal^2(\hat \density_m,f)+T(\hat \density_m,f) -pen(m')\rb+2pen(m).
    \end{align*}
\end{proof}

\subsection{Proof of Lemma \ref{lemma:bernstein-var}}\label{sec:prf-bervb}
\begin{proof}
    The proof will then follow by integrating both sides with respect to $\lambda_n$. It is enough to prove,
    \begin{align*}
        \lp \frac{\sqrt{f_2}-\sqrt{f_1}}{\sqrt{\bar f}} \rp^2 \density\leq 3\lb\lp\sqrt{\density}-\sqrt{f_2}\rp^2+\lp\sqrt{\density}-\sqrt{f_1}\rp^2\rb.
    \end{align*}
    This is equivalent to proving
    \begin{align*}
        \lp \sqrt{f_2}-\sqrt{f_1} \rp^2 \density\leq 3\bar f\lb\lp\sqrt{\density}-\sqrt{f_2}\rp^2+\lp\sqrt{\density}-\sqrt{f_1}\rp^2\rb.
    \end{align*}
    It holds by algebra that $\density\leq 2\lb(\sqrt{\density}-\sqrt{\bar f})^2+\bar f\rb$. The left hand side can now be rewritten as
    \begin{align*}
         \lp \sqrt{f_2}-\sqrt{f_1} \rp^2 \density & \leq 2  \lp \sqrt{f_2}-\sqrt{f_1} \rp^2\lb (\sqrt{\density}-\sqrt{\bar f})^2+\bar f \rb\\
         & = 2\bar f \lp \sqrt{f_2}-\sqrt{f_1} \rp^2\lb \frac{(\sqrt{\density}-\sqrt{\bar f})^2}{\bar f}+1 \rb\\
         & = 2\bar f \lb \frac{(\sqrt{\density}-\sqrt{\bar f})^2}{\bar f}\lp \sqrt{f_2}-\sqrt{f_1} \rp^2+\lp \sqrt{f_2}-\sqrt{f_1} \rp^2\rb \numberthis\label{eq:bervb-eq1}
    \end{align*}
    Observe that $\lp \sqrt{f_2}-\sqrt{f_1} \rp^2/\bar f\leq (\sqrt{\max\{f_1,f_2\}})^2/\bar f$ which in turn can be upper bounded by 2.
    Thus,
    \begin{align*}
         \frac{(\sqrt{\density}-\sqrt{\bar f})^2}{\bar f}\lp \sqrt{f_2}-\sqrt{f_1} \rp^2&\leq 2(\sqrt{\density}-\sqrt{\bar f})^2\\
         & \leq 2\frac{(\sqrt{f_2}-\sqrt{s})^2+(\sqrt{f_1}-\sqrt{s})^2}{2},
    \end{align*}
    where the second inequality follows from the convexity of the function $x\rightarrow (\sqrt{x}-\sqrt{\density})^2$ and Jensen's inequality. 
    Since the fact $\lp \sqrt{f_2}-\sqrt{f_1} \rp^2\leq 2\lb \lp \sqrt{f_2}-\sqrt{s} \rp^2+\lp \sqrt{f_1}-\sqrt{s} \rp^2\rb $ holds algebraically, we now have
    \begin{align*}
         \frac{(\sqrt{\density}-\sqrt{\bar f})^2}{\bar f}\lp \sqrt{f_2}-\sqrt{f_1} \rp^2+\lp \sqrt{f_2}-\sqrt{f_1} \rp^2\leq 3\lb (\sqrt{f_2}-\sqrt{s})^2+(\sqrt{f_1}-\sqrt{s})^2\rb.
    \end{align*}
    This, when combined with \cref{eq:bervb-eq1} completes the proof of our lemma.
\end{proof}

\subsection{Proof of Lemma \ref{prop:mb-eb}}\label{sec:prf-mbeb}
\begin{proof}
    The proof of this Lemma share similarities with the proofs of Propositions 2 and 3 in \cite{baraud_estimator_2011} or that of Claim B3 in \cite{sart_estimation_2014}. To begin, observe that it is enough to show
    \begin{align*}
         &\Hcal^2(\density ,f_2)+ \Test(f_1,f_2)-\Hcal^2(\density ,f_1)\leq \frac{1}{\sqrt{2}}\lp \Hcal^2(\density ,f_2)+\Hcal^2(\density ,f_1) \rp +\frac{1}{n}\sum_{i=0}^{n-1} Z_i(f_1,f_2).
    \end{align*}
   
    Starting from the left hand side, we substitute the expression for $\Test$ from \cref{eq:Tn}, expand all squares, and cancel relevant terms. To be precise, we can write,
    \begin{align*}
         \text{L.H.S} & = \int \lp\sqrt{f_2}-\sqrt{\density}\rp^2 d\lambda_n-\int \lp\sqrt{f_1}-\sqrt{\density}\rp^2 d\lambda_n+ \frac{1}{n}\sum_{i=0}^{n-1}\psi\lp f_1(X_i,a_i,X_{i+1}),f_2(X_i,a_i,X_{i+1})\rp\\
         &\qquad +\int \sqrt{\bar f}\lp \sqrt{f_2}-\sqrt{f_1} \rp d\lambda_n+\int \lp f_1-f_2 \rp d\lambda_n.\\
         & = -2\rho(f_2,\density)+2\rho(f_1,\density)+ \frac{1}{n}\sum_{i=0}^{n-1}\psi\lp f_1(X_i,a_i,X_{i+1}),f_2(X_i,a_i,X_{i+1})\rp\\
         & \qquad +\int \sqrt{\bar f}\lp \sqrt{f_2}-\sqrt{f_1} \rp d\lambda_n\\
         & = -2\rho(f_2,\density)+2\rho(f_1,\density) + \frac{1}{n}\sum_{i=0}^{n-1} Z_i(f_1,f_2)+\int \psi(f_1,f_2)\ \density \ d\lambda_n+\int \sqrt{\bar f}\lp \sqrt{f_2}-\sqrt{f_1} \rp d\lambda_n
    \end{align*}
    All that is now left to show is
    \[
    -2\rho(f_2,\density)+2\rho(f_1,\density)+\int \psi(f_1,f_2)d\lambda_n+\int \sqrt{\bar f}\lp \sqrt{f_2}-\sqrt{f_1} \rp d\lambda_n
    \]
    can be bounded above from by $0.5^{0.5}\lp \Hcal^2(\density ,f_2)+\Hcal^2(\density ,f_1)  \rp  $. As before, we start with the left hand side and observe that
    \begin{align*}
         & -2\rho(f_2,\density)+2\rho(f_1,\density) + \int \psi(f_1,f_2)\ \density \ d\lambda_n+\int \sqrt{\bar f}\lp \sqrt{f_2}-\sqrt{f_1} \rp d\lambda_n \\
         & = \int \lb -2\sqrt{f_2\density} + 2\sqrt{f_1\density}+ \frac{\sqrt{f_2}-\sqrt{f_1}}{\sqrt{\bar f}}\density +\sqrt{\bar f}\lp \sqrt{f_2}-\sqrt{f_1} \rp\rb d\lambda_n\\
         & = \int \lb \sqrt{\frac{f_2}{\bar f}}\lp \sqrt{\bar f}-\sqrt{\density} \rp^2-\sqrt{\frac{f_1}{\bar f}}\lp \sqrt{\bar f}-\sqrt{\density} \rp^2 \rb d\lambda_n\\
         & \leq \int\sqrt{\frac{f_2}{\bar f}}\lp \sqrt{\bar f}-\sqrt{\density} \rp^2 d\lambda_n\\
         & \leq \sqrt{2}\Hcal^2(\bar f, \density).
    \end{align*}
    The first inequality follows trivially. The second inequality follows from the fact that $f_2/\bar f\leq 2$. 
    Now, observe that the function $x\rightarrow (\sqrt{x}-\sqrt{\density})^2$ is convex in $x$ when $x>0$. Therefore, using Jensen's inequality, we can write $\sqrt{2}\Hcal^2(\bar f, \density)\leq \lb\Hcal^2( f_1, \density)+\Hcal^2(f_2, \density)\rb/\sqrt{2}$. 
    This completes the proof.
\end{proof}

\subsection{Sketch of Proofs of Corollaries \ref{cor:holder} and \ref{cor:besov}}\label{sec:prf-corbesov}
\begin{proof}
 Corollary \ref{cor:holder} is proved similarly to part 1 of the proof of \cite[Proposition 3]{baraud_estimating_2009}. \whiteqed
To prove Corollary \ref{cor:besov}, we first use Theorem \ref{thm:main-riskbd} to get,
\begin{align*}
      \Constant \expec\lb \Hcal^2(\density ,\hat \density) \rb\leq \inf_{m\in \Mcal_\lcal} \lc \expec\lb \Hcal^2\lp \density ,V_m \rp \rb+pen(m) \rc.
\end{align*}
Now, it is easy to see that under part 1 of Assumption \ref{assume:besov}, $\expec \Hcal^2\lp \density ,V_m \rp\leq \Gamma\Vol{A} d_2^2(\sqrt{\density },V_m)$ where $d_2$ is the $L_2$ norm. Substituting this into the previous equation we get 
\begin{align}
      \Constant \expec\lb \Hcal^2(\density ,\hat \density) \rb\leq \inf_{m\in\Mcal_\lcal}\lc \Vol A \Gamma d_2^2(\sqrt{\density },V_m)+pen(m)\rc.
\end{align} 
The rest of the proof follows similarly to part 2 of the proof of \cite[Proposition 3]{baraud_estimating_2009} to prove Corollary 2.     
\end{proof}

\subsection{Proof of Proposition \ref{prop:mm-ex2prop1}}
\begin{proof}
    We first prove 1. Recall the definition of atoms from \cite{meyn_markov_2012} and observe that $(X_i,a_i)$ is a stationary Markov chain with atoms $\lc k_i\pow\chi\times k_j\pow \Ibb \rc$ with $i,j\in \{1,\dots,d\}$. It follows now from Kac's theorem \cite[Theorem 10.2.2]{meyn_markov_2012} for any atom $\alpha$, 
    \[
    \expec[T(\alpha)] = \frac{1}{\int_{x,l\in\alpha} \Pi(x,l)dxdl}.\numberthis\label{eq:ex2prop1-eq1}
    \]
    We simply verify that $\Pi(x,l)>3\iota/2$ for any $(x,l)\in\chi\times \Ibb$. Recall from hypothesis that $\eps<1/32$. This implies that, for any $\xi\in\lc0,1\rc$ \[1-\xi\eps-\iota>31/32-\iota>\iota\] whenever $\iota<31/64$. 
    Thus,
    \[
    \frac{3(1-\xi\eps-\iota)\iota}{2}>\frac{3\iota^2}{2}>\frac{3\iota^2}{4}.
    \]
    Similarly, for $d\geq 12$, $d/(d-3)>1$, and for $\iota\in(1/32,31/64)$, $1-\iota>\iota$. Thus
    \[
    \frac{d\iota^2}{2(d-3)}>\frac{\iota^2}{2}>\frac{\iota^2}{4},\qquad  \text{ and, } \qquad \frac{(1-\iota)^2}{2}>\frac{\iota^2}{2}>\frac{\iota^2}{4}.
    \]
    Finally, for $\iota\in (1/32,31/64)$, $\iota>5\iota^2/4$. Thus, $\Pi(\cdot,\cdot)>5\iota^2/4$. Now, since any $\Scal\subset \alpha$ is also an atom (subsets of atoms are atoms by definition), the rest of the proof follows.
    
   Turning to 2 let $\chi_0=\bigcup_{i=1}^{d_1/3}\kappa_i\pow\chi$ and $\kappa=3\iota$. Observe that $\Vol{\chi_0}=1/3$. Now using Lemma \ref{lemma:mixing-lemma}, we arrive at the conclusion.

   Turning to 3, we first recall the definition of $\rho_\star$ from Theorem \ref{thm:detlos-2}:
   \[
   \rho_\star(\Scal) = \sup_i\max\lc \prob((X_i,a_i)\in \Scal), \sup_{j>i}\sqrt{\prob\lp (X_i,a_i)\in \Scal,(X_j,a_j)\in \Scal\rp}\rc\numberthis\label{eq:ex2prop1-eq2}.
   \]
   Now we can upper bound each term separately. Fix $i_0$ and $j_0$ and consider the following joint probability 
   \[
   \prob\lp (X_i,a_i)\in \Scal_{i_0,j_0},(X_j,a_j)\in \Scal_{i_0,j_0}\rp = \underbrace{ \prob\lp (X_j,a_j)\in \Scal_{i_0,j_0}| (X_i,a_i)\in \Scal_{i_0,j_0}\rp}_{=:\mathrm{Term1}}\underbrace{\prob\lp (X_i,a_i)\in \Scal_{i_0,j_0}\rp}_{=:\mathrm{Term2}}
   \]
   Since $(X_i,a_i)$ is a stationary Markov chain, it follows from Lemma \ref{lemma:stationary-dist} that 
   \begin{align*}
      \mathrm{Term2} = \Pi(\Scal_{i_0,j_0}) & = \int_{x\in k_{i_0}\pow \chi,l\in \kappa_{j_0}\pow \Ibb}\Pi\pow{\iota,\xi\pow {l}}\lp x\rp dxdl\\
      &<\frac{3(1+\frac{1}{32}-\iota)}{2}\int_{x\in k_{i_0}\pow \chi,l\in \kappa_{j_0}\pow \Ibb}dxdl\\
      &=\frac{3(33-32\iota)}{64d_1d_2}.
    \end{align*}
    For the Term1, we only show the case when $j=i+1$. When $j>i+1$, the proof follows very similarly using Champman-Kolmogorov decompositions. There are $2$ possible combinations given by whether $i_0$ lies in the set $\{1,\dots,d_1/3\}$ or not. 
    \begin{enumerate}
        \item[Case 1.] $(i_0\geq d_1/3+1)$. Since $a_{i+1}$ is a uniform random variable independent of the history, 
        \begin{align*}
          \prob\lp (X_{i+1},a_{i+1})\in \Scal_{i_0,j_0}| (X_i,a_i)\in \Scal_{i_0,j_0}\rp & = \int_{l\in k_{j_0}\pow \Ibb}  \prob\lp X_{i+1}\in k_{i_0}\pow \chi| (X_i,a_i)\in \Scal_{i_0,j_0}\rp dl\\
          & = \frac{\prob\lp X_{i+1}\in k_{i_0}\pow \chi| (X_i,a_i)\in \Scal_{i_0,j_0}\rp}{d_2}.
        \end{align*}
        Next, we observe that the transition density $\density(x,l,y)=\frac{3(1-\iota)}{2}$ for all $x,l\in\Scal_{i_0,j_0}$. In particular, it is independent of $x,l$. Thus,
        \[
        \prob\lp X_{i+1}\in k_{i_0}\pow \chi| (X_i,a_i)\in \Scal_{i_0,j_0}\rp = \int_{x\in k_{i_0}\pow \chi} \frac{3(1-\iota)}{2}dx=\frac{3(1-\iota)}{2d_1}.
        \]
        So we get, $\mathrm{Term1} = 3(1-\iota)/(2d^2)<9(1-\iota)/(2d_1d_2)$ as required.
        \item[Case 2.] $(i_0\leq d_1/3)$. Similar to above, we only need to find $\prob\lp X_{i+1}\in k_{i_0}\pow \chi| (X_i,a_i)\in \Scal_{i_0,j_0}\rp$. And by a reasoning similar to before,
        \[
        \prob\lp X_{i+1}\in k_{i_0}\pow \chi| (X_i,a_i)\in \Scal_{i_0,j_0}\rp = \frac{3\iota}{d-3}<\frac{9(1-\iota)}{2d_1d_2}
        \]
        when $\iota\in(1/32,31/64)$ and $d\geq 12$. 
    \end{enumerate}
    We finally get $\mathrm{Term1}< 9(1-\iota)/2d_1d_2$. This implies 
        \[
        \prob\lp (X_i,a_i)\in \Scal_{i_0,j_0},(X_j,a_j)\in \Scal_{i_0,j_0}\rp< \frac{3(33-32\iota)}{64d_1d_2}\times \frac{9(1-\iota)}{2d_1d_2}< \lp\frac{9(1-\iota)}{2d_1d_2}\rp^2 
        \]
        in our given range of $\iota$ and $d$.
        It can be easily seen from the calculations of Case 1. that $ \prob((X_i,a_i)\in \Scal)<9(1-\iota)/2d_1d_2$. By substituting all upper bounds into \cref{eq:ex2prop1-eq2} that
        \[
        \rho_\star(\Scal_{i_0,j_0})<\frac{9(1-\iota)}{2d_1d_2}.
        \]
\end{proof}

\subsection{Proof of Theorem \ref{thm:detlos-1}}\label{sec:prf-detls}

We first prove the following proposition

\begin{proposition}~\label{prop:detlos} Let $m_{ref}\pow 2$ be the partition of $A$ into uniform cubes of edge length $2^{-l}$. Assume that $\lc (X_i,a_i)\rc_{i=0}^n$ is a sequence from a controlled Markov chain satisfying Assumption \ref{assume:alpha-mix}. Then, the histogram estimator $\hat s$ satisfies the following risk bound
     \begin{align*}
        \Constant\expec\lb h_n^2\lp \density ,\hat \density \rp \rb\leq \inf_{m\in \Mcal_\lcal} \lc  h_n^2\lp \density ,V_m \rp +pen(m) \rc+\Rcal(n).
    \end{align*}
    where
    \begin{small}
         \begin{align*}
            \Rcal(n) = \sum_{\Scal_{r}\in m_{ref}\pow 2} \exp\lp- \frac{\Constant_pn\nu_n^2(\Scal_{r})}{4\Constant_\Delta\sup_{i,j}\sqrt{\prob\lp (X_i,a_i)\in \Scal_r,(X_j,a_j)\in \Scal_r\rp} +4n^{-1}+2\nu_n(\Scal_{r})(\log n)^2}\rp.
        \end{align*}   
    \end{small}
    is a remainder term. $\Constant_\Delta$ is as in Assumption \ref{assume:alpha-mix} and $\Constant_p$ only depends upon $\constant_p$ in Assumption \ref{assume:alpha-mix}
\end{proposition}

\begin{proof}
     Let $A':= \{(x,l):\exists y\in \chi,\ (x,l,y)\in A \}$. In words, $A'$ is the set given by the first two coordinates of elements in $A$. Let $m_{ref}\pow 1$ and $m_{ref}\pow 2$ be the partitions of $A$ and $A'$ into uniform cubes of edge-length $2^l$ respectively. Let $\Psi$ be the tail event given by

    \[
    \Psi = \{ \forall f_1,f_2\in V_{m_{ref}\pow 1} : h_n^2 (f_1,f_2)\leq 2\Hcal^2(f_1,f_2). \}
    \]
    We can decompose the risk as follows. 
    \begin{align*}
        \expec\lb h_n^2(\density,\hat \density) \rb & = \expec\lb h_n^2(\density,\hat \density) \indicator_\Psi\rb+\expec\lb h_n^2(\density,\hat \density) \indicator_{\Psi^c}\rb\\
        & = \textit{Term 1}+\textit{Term 2}.
    \end{align*}
    
    \textit{Term 1:} Observe that if $m\in \Mcal_\lcal$ then $V_m\subseteq V_{m_{ref}\pow 1}$. Let $\bar \density_m:= \argmin_{f_1\in V_m} \{ h_n^2(\density,f_1) \}$. 
    \begin{align*}
        \expec\lb h_n^2(\density,\hat \density) \indicator_\Psi\rb & \leq \expec\lb h_n^2(\density,\bar \density_m) \indicator_\Psi\rb + \expec\lb h_n^2(\bar \density_m,\hat \density) \indicator_\Psi\rb\\
        & \leq  \expec\lb h_n^2(\density,\bar \density_m) \indicator_\Psi\rb + 2 \expec\lb \Hcal^2(\bar \density_m,\hat \density) \indicator_\Psi\rb\\
        & \leq \expec\lb h_n^2(\density,\bar \density_m) \indicator_\Psi\rb + 2 \expec\lb \Hcal^2(\density,\hat \density) \indicator_\Psi\rb+2 \expec\lb \Hcal^2(\bar \density_m, \density) \indicator_\Psi\rb\\
        & \leq \expec\lb h_n^2(\density,\bar \density_m) \indicator_\Psi\rb + 2 \expec\lb \Hcal^2(\density,\hat \density)\rb + 2 \expec\lb \Hcal^2(\bar \density_m, \density) \rb
    \end{align*}
    We bound $\expec\lb \Hcal^2(\density,\hat \density)\rb\leq \inf_{m\in \Mcal_\lcal} \lc \expec\lb \Hcal^2\lp \density ,V_m \rp \rb+pen(m) \rc$ by Theorem \ref{thm:main-riskbd}. 
    
    \emph{Term 2: } Since the $h_n^2(\cdot,\cdot)\leq 1$, the second term can be bounded as follows $\expec\lb  \indicator_{\Psi^c}\rb = \prob\lp \Psi^c \rp$.    Observe that,
    \begin{align*}
         \Psi^c & = \{ \exists f_1,f_2\in V_{m_{ref}\pow 1} : h_n^2 (f_1,f_2)\geq 2\Hcal^2(f_1,f_2). \}\\
         & \subseteq \lc \exists \Scal_{r}\in m_{ref}\pow 2 : \nu_n(\Scal_{r})\geq \frac{2}{n} \sum_{i=0}^{n-1} \indicator_{\Scal_r}(X_i,a_i) \rc\\
         & \subseteq \bigcup_{\Scal_{r}\in m_{ref}\pow 2}\lc\nu_n(\Scal_{r})\geq \frac{2}{n} \sum_{i=0}^{n-1} \indicator_{\Scal_r}(X_i,a_i) \rc\\
         & = \bigcup_{\Scal_{r}\in m_{ref}\pow 2}\lc -\nu_n(\Scal_{r})\geq \frac{2}{n} \sum_{i=0}^{n-1} \indicator_{\Scal_r}(X_i,a_i) -2\nu_n(\Scal_r)\rc\\
         & = \bigcup_{\Scal_{r}\in m_{ref}\pow 2}\lc -\nu_n(\Scal_{r})\geq \frac{2}{n} \sum_{i=0}^{n-1} \indicator_{\Scal_r}(X_i,a_i) -\frac{2}{n}\expec\lb \sum_{i=0}^{n-1} \indicator_{\Scal_r}(X_i,a_i) \rb\rc\\
         & =  \bigcup_{\Scal_{r}\in m_{ref}\pow 2}\lc -\frac{n}{2}\nu_n(\Scal_{r})\geq  \sum_{i=0}^{n-1} \indicator_{\Scal_r}(X_i,a_i) -\expec\lb \sum_{i=0}^{n-1} \indicator_{\Scal_r}(X_i,a_i) \rb\rc.
    \end{align*}
    In the previous equation, the second equality follows since $\nu_n(\Scal_r) =\expec\lc \sum \indicator_{\Scal_r}(X_i,a_i)/n \rc$. Now it follows that,
    \begin{align*}
        \prob\lp \Psi^c\rp\leq \sum_{\Scal_{r}\in m_{ref}\pow 2}\prob\lp -\frac{n}{2}\nu_n(\Scal_{r})\geq  \sum_{i=0}^{n-1} \indicator_{\Scal_r}(X_i,a_i) -\expec\lb \sum_{i=0}^{n-1} \indicator_{\Scal_r}(X_i,a_i) \rb\rp.
    \end{align*}
    Let $Y_i:= \indicator_{\Scal_r}(X_i,a_i) -\expec\lb\indicator_{\Scal_r}(X_i,a_i) \rb$ and $\vee^2:= \sup_i\lc\Var(Y_i)+2\sum_{j\geq i} \Cov(Y_i,Y_j)\rc$. 
    Using the concentration inequality for $\alpha$-mixing processes (Theorem 2) from \cite{merlevede_bernstein_2009} we get
    \begin{align*}
        \prob\lp \Psi^c\rp & \leq \sum_{\Scal_{r}\in m_{ref}\pow 2} \exp\lp- \frac{\Constant_p\frac{n^2}{4}\nu_n^2(\Scal_{r})}{n\vee^2 +1+\frac{n}{2}\nu_n(\Scal_{r})(\log n)^2}\rp\\
        & = \sum_{\Scal_{r}\in m_{ref}\pow 2} \exp\lp- \frac{\Constant_pn^2\nu_n^2(\Scal_{r})}{4n\vee^2 +4+2n\nu_n(\Scal_{r})(\log n)^2}\rp\\
        & = \sum_{\Scal_{r}\in m_{ref}\pow 2} \exp\lp- \frac{\Constant_pn\nu_n^2(\Scal_{r})}{4\vee^2 +4n^{-1}+2\nu_n(\Scal_{r})(\log n)^2}\rp
    \end{align*}
    where $\Constant_p$ is a constant depending only upon $\constant_p$ as defined in Assumption \ref{assume:alpha-mix}.   
     All that is left is to upper bound $\vee^2$. We use the slightly stronger version of Davydov's covariance bound for $\alpha$-mixing processes. Its proof is in Section \ref{sec:prf-alpcovbnd}.
     \begin{lemma}~\label{lemma:alpha-covbound}
         If $Y_1$ and $Y_2$ are two random variables adapted to $\Hcal_0^i$ and $\Hcal_{i+j}^\infty$, such that $I_1=\indicator_{[Y_1\in A]}$ and $I_2=\indicator_{[Y_2\in A]}$ then $\Cov(I_1,I_2)\leq \sqrt{\alpha_{i,j}\prob(Y_1\in A,Y_2\in A)}$
     \end{lemma} Using Lemma \ref{lemma:alpha-covbound}, we get 
     \[\vee^2\leq \sup_i\lc\Var(Y_i)+2\sum_{j>i} {\sqrt{\alpha_{i,j}\prob\lp (X_i,a_i)\in \Scal_r,(X_j,a_j)\in \Scal_r\rp}}\rc.\numberthis\label{eq:cov-bound}\]
     Since $Y_i =\indicator_{\Scal_r}(X_i,a_i) -\expec\lb\indicator_{\Scal_r}(X_i,a_i) \rb$, $\Var(Y_i)\leq \prob\lp (X_i,a_i)\in \Scal_r\rp(1-\prob\lp (X_i,a_i)\in \Scal_r\rp) \leq \prob\lp (X_i,a_i)\in \Scal_r\rp $. It now follows from Assumption \ref{assume:alpha-mix} that,
     \[
     \vee^2 \leq \lp 1+\sum_{j\geq i}{\alpha_{i,j}} \rp \sup_i\max\lc \prob\lp (X_i,a_i)\in \Scal_r\rp,\sup_{j\geq i}\sqrt{\prob\lp (X_i,a_i)\in \Scal_r,(X_j,a_j)\in \Scal_r\rp}\rc \leq \Constant_\Delta \rho_\star(\Scal_r) .
     \]
    Therefore, 
    \begin{align*}
        \prob\lp \Psi^c\rp \leq  \sum_{\Scal_{r}\in m_{ref}\pow 2} \exp\lp- \frac{\Constant_pn\nu_n^2(\Scal_{r})}{4\Constant_{\Delta}\rho_\star(\Scal_r) +4n^{-1}+2\nu_n(\Scal_{r})(\log n)^2}\rp.
    \end{align*}
    This completes the proof.
\end{proof}

\paragraph{Proof of Theorem \ref{thm:detlos-1}}
\begin{proof}
    We first upper bound $h^2(\cdot, \cdot)$. Let $f,g$ be two conditional densities. We observe that
    \begin{align*}
        h^2(f,g) & = \int_{\chi\times\Ibb\times\chi} \lp\sqrt{f(x,l,y)}-\sqrt{g(x,l,y)}\rp^2\nu(dx,dl)\mu_\chi(dy)\\
        & =  \int_{\chi\times\Ibb\times\chi} \lp\sqrt{f(x,l,y)}-\sqrt{g(x,l,y)}\rp^2 \lp\nu_n(dx,dl)-\nu_n(dx,dl)+\nu(dx,dl)\rp\mu_\chi(dy)\\
        & \leq \int_{\chi\times\Ibb} 2 \lp\nu(dx,dl)-\nu_n(dx,dl)\rp+\int_{\chi\times\Ibb\times\chi} \lp\sqrt{f(x,l,y)}-\sqrt{g(x,l,y)}\rp^2\nu_n(dx,dl)\mu_\chi(dy)\\
        & = \mathrm{Term 1}+\mathrm{Term 2}
    \end{align*}
    where the previous inequality follows from the trivial bound 
    \[\int_\chi\lp\sqrt{f(x,l,y)}-\sqrt{g(x,l,y)}\rp^2\mu_\chi(dy) \leq 2.\] 
    Observe that 
    \[\mathrm{Term 1} = \int_{\chi\times\Ibb\times\chi} \lp\sqrt{f(x,l,y)}-\sqrt{g(x,l,y)}\rp^2\nu_n(dx,dl)\mu_\chi(dy)=h_n^2(f,g)\] 
    Turning to $\mathrm{Term 2}$, we write
    \begin{align*}
        \mathrm{Term 2} & =  \int_{\chi\times\Ibb}  \lp\nu(dx,dl)-\nu_n(dx,dl)\rp\\
        & \leq \int_{\lc x,l:\nu(dx,dl)-\nu_n(dx,dl)>0\rc}    \lp\nu(dx,dl)-\nu_n(dx,dl)\rp\\
        &\leq \|\nu_n-\nu\|_{TV}=r_n        
    \end{align*}
    we get
    \[
    h^2(f,g)\leq h_n^2(f,g)+2r_n
    \]
    Now following Proposition \ref{prop:detlos} we only need to upper bound $\Rcal(n)$ where
    \begin{align*}
        \Rcal(n) =\sum_{\Scal_{r}\in m_{ref}\pow 2} \exp\lp - \frac{\Constant_pn\nu_n^2(\Scal_{r})}{4\Constant_{\Delta}\sup_i \prob(X_i,a_i\in\Scal_r) +4n^{-1}+2\nu_n(\Scal_{r})(\log n)^2}\rp. 
    \end{align*}
    Next, we produce a lower bound for $\nu_n$. Recall from Definition \ref{assume:conv-gap} the definition of $r_n$
    \[r_n = \lV \nu_n-\nu \rV_{TV}.\]
    It follows that $\sup_{\Acal}| \nu_n(\Acal)-\nu(\Acal) |=r_n$ for any measurable set $\Acal$. Observe that this implies 
    \[\sup_{\Acal}| \nu_n^2(\Acal)-\nu^2(\Acal) | =\sup_{\Acal} | \nu_n(\Acal)-\nu(\Acal) |\lp \nu_n(\Acal)+\nu\lp\Acal\rp \rp\leq 2r_n\]
    Consequently, \[\sup_\Acal\lc \nu_n(\Acal)- \nu(\Acal)\rc \leq r_n \text{ and } \inf_\Acal\lc \nu_n^2(\Acal)- \nu^2(\Acal)\rc \geq -2r_n.\] 
    Now substituting the above lower bounds for $\nu^2_n(\Scal_{r})$ and $\nu_n(\Scal_{r})$ it follows that,
    \begin{align*}
        \Rcal(n) \leq   \sum_{\Scal_{r}\in m_{ref}\pow 2}  \exp\lp- \frac{\Constant_pn\nu^2(\Scal_{r})-2n\Constant_pr_n}{4\Constant_{\Delta}\sup_i \prob(X_i,a_i\in\Scal_r) +4n^{-1}+2\nu(\Scal_{r})(\log n)^2+2r_n(\log n)^2}\rp.
    \end{align*}
    Therefore, we get
     \[
     \Rcal(n)\leq  \sum_{\Scal_{r}\in m_{ref}\pow 2}  \exp\lp- \frac{\Constant_pn\nu^2(\Scal_{r})-2n\Constant_pr_n}{4\Constant_{\Delta}\sup_i \prob(X_i,a_i\in\Scal_r)+4n^{-1}+2\nu(\Scal_{r})(\log n)^2+2r_n(\log n)^2}\rp
     \]
     Observe that the term in the exponent of the right hand side of the previous equation is maximised by some small set $\Scal_{min}$. Let
     \[
     \Scal_{min}:= \argmax_{\Scal_{r}\in m_{ref}\pow 2} \exp\lp- \frac{\Constant_pn\nu^2(\Scal_{r})-2n\Constant_pr_n}{4\Constant_\Delta {\prob\lp (X_i,a_i)\in \Scal_r\rp}+4n^{-1}+2\nu(\Scal_{r})(\log n)^2+2r_n(\log n)^2}\rp
     \]
    Then we get,
    \begin{align*}
        \Rcal(n) &\leq  2^{\lcal(d_1+d_2)}  \exp\lp- \frac{\Constant_pn\nu^2(\Scal_{min})-2n\Constant_pr_n}{4\Constant_\Delta {\prob\lp (X_i,a_i)\in \Scal_{min}\rp}+4n^{-1}+2\nu(\Scal_{min})(\log n)^2+2r_n(\log n)^2}\rp
    \end{align*}
    where the inequality follows from the construction of $m_{ref}\pow 2$. Observe that $\nu(\Scal_{min})\leq 1$ and $(4+2(\log n)^2)n^{-1}\leq 1$ for $n\geq 5$ The rest of the proof follows using some simple algebra.
\end{proof}

\subsection{Proof of Theorem \ref{thm:detlos-2}}\label{sec:prf-detls2}
\begin{proof}
    We first state the following lemma whose proof is in Section \ref{sec:prf-kaclwr}. Recall the definition of $T(\cdot)$ in \ref{eq:return_time_def}. Then we have,
\begin{lemma}~\label{lemma:KAC-lower}
For any $\Scal\subseteq\chi\times\Ibb$
    \[
    \nu_n(\Scal)\geq \frac{1}{T(\Scal)}-\frac{1}{n}.
    \]    
\end{lemma}
Using the previous lemma and the fact that $n\geq 2T(S_\star)\geq 2T(S_r)$ for all $S_{r}\in m_{ref}\pow 2$ we get 
\[
\nu_n(S_r)\geq \frac{1}{T(S_r)}-\frac{1}{n}\geq \frac{1}{2T(S_r)}\label{eq:detlb2-eq1}.
\]
The rest of the proof follows by substituting the previous lower bound in Proposition \ref{prop:detlos}.
\end{proof}

\subsection{Proof of Lemma \ref{lemma:TTLb}}\label{sec:prf-cvrtm}

\begin{proof}
We introduce the notation 
\[
\chi':= \lc(k_1\pow \chi\times k_1\pow \Ibb),\dots,(k_{d_1/3}\pow\chi\times k_1\pow \Ibb),(k_1\pow \chi\times k_2\pow \Ibb),\dots,(k_{d_1/3}\pow \chi\times k_{d_2}\pow \Ibb) \rc.
\]
Observe that $ \Tbb$ can be written as,
\begin{align*}
    \Tbb:=\sum_{\Upsilon=0}^{d_1d_2/3-1} U_\Upsilon\numberthis\label{eq:lem-covt1}
\end{align*}
where $U_\Upsilon$ is the time spent between the $\Upsilon$-th and the $\Upsilon+1$-th unique element visited in $\chi'$.
Next, we observe two facts. Firstly, observe that for any element $(k_t\pow\chi,k_{l'}\pow\Ibb)$ belonging to $\chi'$ we have
\begin{align*}
    \prob\lp (X_i,a_i)\in (k_t\pow\chi,k_{l'}\pow\Ibb)|\History_0^{i-1}=\history_0^{i-1} \rp = \frac{3\iota}{d_1d_2}
\end{align*}
independent of any history $\History_0^{i-1}$. 
Secondly, observe that the probability of visiting a new state-control pair in $\chi'$ when $\Upsilon$ unique states have already been visited is ${3\iota\lp{d_1d_2}/{3}-\Upsilon\rp}/d_1d_2$.
Together, these facts imply that
\begin{align*}
    U_\Upsilon\overset{d}{=}X_\Upsilon \text{ where } X_\Upsilon\sim Geometric\lp\lp\frac{d_1d_2}{3}-\Upsilon\rp\frac{3\iota}{d_1d_2} \rp.\numberthis~\label{eq:lem-covt2}
\end{align*}
It follows from \cref{eq:lem-covt2} that, 
\begin{align*}
    \expec[\Tbb] & = \lp \frac{d_1d_2}{3\iota}\sum_{ \Upsilon=0}^{d_1d_2/3-1}\frac{1}{d_1d_2/3-\Upsilon}\rp\\
    \intertext{where we have dropped the superscript $l$ from $\Upsilon\pow{l}$ for convenience. Rewriting the previous equation we get, }
     \expec[\Tbb] & = \frac{d_1d_2}{3\iota}\sum_{ \Upsilon=1}^{d_1d_2/3}\frac{1}{\Upsilon}\\
    & > \frac{d_1d_2}{3\iota}\log\lp d_1d_2/3+1\rp~\numberthis\label{eq:lem-covt4}.
\end{align*}
where the last inequality follows from the Euler-Maclaurin (see for example, \cite{apostol_elementary_1999}) approximation of a sum by its integral.
We also observe that,
\begin{align*}
    \Var(U_\Upsilon)=\frac{d^2k^2}{9\iota^2}\lp\frac{d_1d_2}{3}-\Upsilon\rp^{-2}\lb1-\lp\frac{d_1d_2}{3}-\Upsilon\rp\frac{3\iota}{d_1d_2} \rb.
\end{align*}
The term inside the square brackets is a probability, and can be upper bounded by $1$. 
Observe that when $\Upsilon\leq d_1d_2/3-1$ we can upper bound $\Var(\Tbb)$ as
\begin{align*}
    \Var( \Tbb) & \leq \sum_{\Upsilon=0}^{d_1d_2/3-1} \frac{d^2k^2}{9\iota^2}\lp\frac{d_1d_2}{3}-\Upsilon\rp^{-2}\\
    & = \sum_{\Upsilon=1}^{d_1d_2/3} \frac{d^2k^2}{9\iota^2}\frac{1}{\Upsilon^2}\\
    & < \frac{d^2k^2}{9\iota^2}\frac{\pi^2}{6}\\
    & < \frac{d^2k^2}{9\iota^2}\frac{\pi^2}{4}.~\numberthis\label{eq:lem-covt3}
\end{align*}
where the second inequality follows from the fact that $\sum_{\Upsilon\geq 1}1/\Upsilon^2=\pi^2/6$. 
Using Cantelli's inequality \citep[Equation 5]{ghosh_probability_2002}, we obtain, for all $0<\theta<{\expec[ \Tbb]}/{\sqrt{\Var( \Tbb)}}$,
\[
\prob\lp \Tbb>  \frac{d_1d_2}{3\iota}\log\lp \frac{d_1d_2}{3}+1\rp -\theta \frac{d_1d_2}{3\iota}\frac{\pi}{2}\rp\geq \frac{\theta^2}{1+\theta^2}.
\]
From the equations \ref{eq:lem-covt4} and \ref{eq:lem-covt3}, we get that ${\expec[ \Tbb]}/\lp{\sqrt{\Var( \Tbb)}}\rp>\lp\log(d_1d_2/3)+1 \rp/\pi$. Substituting $\theta={(\log(d_1d_2/3)+1)/\pi}$ we get
\begin{align*}
    \prob\lp\Tbb>  \frac{d_1d_2}{6\iota}\lp\log\lp\frac{d_1d_2}{3}\rp+1\rp \rp\geq \frac{1}{1+\lp\frac{\pi}{\log(d_1d_2/3)+1}\rp^2}> \frac{1}{1+\pi^2}.
\end{align*}
This proves the lemma.
\end{proof}

We now have all the tools to derive the lower bound. 
\paragraph{Lower Bound on the Probability of Error} Throughout this part, we will assume that $n<d_1d_2/(6\iota)\log(d_1d_2/3)$, so that $\prob\lp\Tbb>n\rp\geq (1+\pi^2)^{-1}$. Using \cref{eq:mm-ex2eq4} and Lemma \ref{lemma:TTLb} we get, 
\begin{align*}
    \prob(h_n^2(\density,\hat \density) >\eps^2|\Tbb>n) \prob(\Tbb>n) & >\prob(\ \Ebb\ |\ \Tbb>n)\prob(\Tbb>n)\\
    & >\frac{1}{1+\pi^2}\prob(\ \Ebb\ |\ \Tbb>n)
\end{align*}
Now, if $\Tbb>n$, there exists $i_0,j_0$ such that $\sum_{i=1}^n\indicator_{\lb(X_i,a_i)\in k_{i_0}\pow\chi\times k_{j_0}\pow\Ibb\rb}=0$. That is $(X_i,a_i)$ never visits the set $k_{i_0}\pow\chi\times k_{j_0}\pow\Ibb$ during the first $n$ time points. Therefore, for any $(x,y)\in k_{i_0}\pow\chi\times k_{j_0}\pow\Ibb$ the best estimate of $\density(x,l,y)$ is to choose uniformly over all possible values of $\xi_1\pow{j_0}$. Since $\{0,1\}$ are the only two possibilities, 
\[
\prob(\ \Ebb\ |\ \Tbb>n)=\frac{1}{2}.
\]
Therefore,
\[
\prob(h_n^2(\density,\hat \density) >\eps^2|\Tbb>n) \prob(\Tbb>n) >\frac{1}{2(1+\pi^2)}.
\]
The rest of the proof now follows.

\subsection{Proof of Lemma \ref{lemma:alpha-covbound}}\label{sec:prf-alpcovbnd}
 \begin{proof}
     Recall that we denoted our probability space by $\Omega,\Fcal,\Fbb,\Pbb$. For convenience of notation, we will denote $\int_{\omega\in\Omega}(\cdot) \prob(d\omega)$ simply by $\int(\cdot)$ We begin by writing explicitly $\Cov(I_1,I_2)$ and observing the upper bound 
     \begin{align*}
         \Cov(I_1,I_2) & =  \int \lp I_1I_2- \int I_1\int I_2\rp\\
         & \leq\int_{I_1I_2=1}\lp I_1I_2- \int I_1\int I_2\rp \\
         & = \int {I_1I_2}\lp I_1I_2- \int I_1\int I_2\rp 
     \end{align*}
     which follows trivially because the term inside is whole square is negative unless $I_1I_2=1$.
     The second inequality follows since, $\int_{I_1I_2=1}\lp I_1I_2- \int I_1\int I_2\rp \in [0,1]$. Similarly, 
     \[
     \lp I_1I_2- \int I_1\int I_2\rp I_1I_2\leq \sqrt{\lp I_1I_2- \int I_1\int I_2\rp I_1I_2}.
     \]
     Now using Cauchy-Schwarz inequality we get 
     \begin{align*}
         \int \sqrt{\lp I_1I_2- \int I_1\int I_2\rp I_1I_2}& \leq \sqrt{\lp\int  I_1I_2- \int I_1\int I_2\rp\lp \int (I_1I_2) \rp}
     \end{align*}
     The first term equals to $\prob(Y_1\in A\cap Y_2\in A)-\prob(Y_1\in A)\prob(Y_2\in A)$ which can be trivially upper bounded by $\alpha_{i,j}$.
     This completes our proof.
 \end{proof}
 \subsection{Proof of Lemma \ref{lemma:KAC-lower}}\label{sec:prf-kaclwr}
\begin{proof}
We begin by fixing an $\Scal$.
\paragraph{Case I:} $(T(\Scal)=\infty)$ In this case, the left hand side is a positive real number and the right hand side becomes negative. Thus, the result holds trivially. We now turn to the non-trivial case.
\paragraph{Case II:} $(T(\Scal)<\infty)$
    Define the random variable $\{Z_{\Scal}\pow{p}\}$ and the filtration $\Fcal_p'$ as,
\begin{align*}
    & Z_{\Scal}\pow{0} := 0\\
    & Z_{\Scal}\pow{p} := \frac{\sum_{i=1}^p\tau_{\Scal}\pow{i}}{T(\Scal)}-p\\
    & \Fcal_p':=\Fcal_{\sum_{i=1}^p\tau_{\Scal}\pow{i}}.
\end{align*}
Observe that 
\begin{align*}
    \expec[ Z_{\Scal}\pow{p}|\Fcal_{p-1}'] & =\frac{\expec[\sum_{i=1}^p\tau_{\Scal}\pow{i}|\Fcal_{p-1}']}{T(\Scal)}-p\\
    & = \frac{\expec[\sum_{i=1}^{p-1}\tau_{\Scal}\pow{i}|\Fcal_{p-1}']}{T(\Scal)}-(p-1)+\frac{\expec[\tau_{\Scal}\pow{p}|\Fcal_{p-1}']}{T(\Scal)}-1\\
    & \leq \expec[{Z_{\Scal}\pow{p-1}}|\Fcal_{p-1}']+\frac{T(\Scal)}{T(\Scal)}-1\\
    & =Z_{\Scal}\pow{p-1},
\end{align*}
where the last inequality follows because $\expec[\tau_{\Scal}\pow{p}|\Fcal_{p-1}']\leq T(\Scal)$ by \cref{eq:return_time_def} and the last equality follows because $Z_{\Scal}\pow{p-1}$ is $\Fcal_{p-1}'$ measurable.
It follows that, $\{Z_{\Scal}\pow{p}\}$ is a supermartingale.
Now, define 
\[
N:=\min\{p\leq n+1:\sum_{i=1}^p\tau_{\Scal}\pow{i}>n\}.
\]
It can be seen easily that $N$ is a valid stopping time. 
Moreover, since the return times $\tau_{\Scal}\pow{i}\geq 1$ $\prob$-almost everywhere, it easily follows that $\prob(N\leq n+1)=1$.
Therefore, it follows from Doob's Optional Stopping Theorem for supermartingales \cite[Theorem 7.1, page 495]{gut_probability_2005} that,
\begin{align*}
    \expec[Z_N]\leq \expec[Z_0].
\end{align*}
Since $Z_0=0$, we can write  
\begin{align*}
    \expec\lb\frac{\sum_{i=1}^N\tau_{\Scal}\pow{i}}{T(\Scal)}-N\rb& \leq 0.\\
         \intertext{This in turn implies}
    \expec\lb\frac{\sum_{i=1}^N\tau_{\Scal}\pow{i}}{T(\Scal)}\rb & \leq \expec[N].
\end{align*}
Let $N_\Scal := \sum_{i=1}^n\indicator_{[(X_i,a_i)\in \Scal]}$ be the number of times the controlled Markov chain returned to the set $\Scal$ in $n$ time steps. Observe that we can write 
\[
N_\Scal=\max\{p\leq n:\sum_{i=1}^p\tau_{\Scal}\pow{i}\leq n\}.
\]
In other words, $N_\Scal=N-1$ $\prob$-almost everywhere. 
It follows that,
\begin{align*}
     \expec\lb\frac{\sum_{i=1}^N\tau_{\Scal}\pow{i}}{T(\Scal)}\rb & \leq \expec[N_\Scal]+1.\\ 
     \intertext{This in turn implies}
      \expec\lb\frac{\sum_{i=1}^N\tau_{\Scal}\pow{i}}{T(\Scal)}\rb-1&\leq \expec[N_\Scal].
\end{align*}
Finally, observe that by definition of $N$, $\sum_{i=1}^N\tau_{\Scal}\pow{i}> n$ $\prob$-almost everywhere.
Therefore,
\begin{align*}
    \frac{n}{T(\Scal)}-1&< \expec[N_\Scal].
\end{align*}
Thus,
\[
\frac{n}{T(\Scal)}-1\leq \expec[N_\Scal] = \expec\lb\sum_{i=1}^n\indicator_{[(X_i,a_i)\in \Scal]}\rb=\sum_{i=1}^n\prob\lp X_i,a_i\in \Scal \rp
\]
Observing $\nu_n(\Scal)=n^{-1}\sum_{i=1}^n\prob\lp X_i,a_i\in \Scal \rp$ and dividing both sides by $n$ completes the proof.
\end{proof}

\end{document}